\documentclass[reqno, 11pt]{amsart} 

\usepackage{amssymb,bbm}
\usepackage{amsfonts}
\usepackage{amsmath, amscd}
\usepackage{graphicx}
\usepackage{color}
\usepackage{enumerate}
\usepackage{soul}
\usepackage{dsfont}
\usepackage[normalem]{ulem}
\usepackage{hyperref}
\hypersetup{
   colorlinks=true,
   linkcolor=blue,
}

\usepackage{tikz}
\usetikzlibrary{automata}
\usetikzlibrary{arrows,shapes,calc}
\usetikzlibrary{positioning}

\makeatletter
\def\l@subsection{\@tocline{2}{0pt}{2.5pc}{5pc}{}}
\renewcommand\tocchapter[3]{%
  \indentlabel{\@ifnotempty{#2}{\ignorespaces#2.\quad}}#3%
}
\newcommand\@dotsep{4.5}
\def\@tocline#1#2#3#4#5#6#7{\relax
  \ifnum #1>\c@tocdepth 
  \else
    \par \addpenalty\@secpenalty\addvspace{#2}%
    \begingroup \hyphenpenalty\@M
    \@ifempty{#4}{%
      \@tempdima\csname r@tocindent\number#1\endcsname\relax
    }{%
      \@tempdima#4\relax
    }%
    \parindent\z@ \leftskip#3\relax \advance\leftskip\@tempdima\relax
    \rightskip\@pnumwidth plus1em \parfillskip-\@pnumwidth
    #5\leavevmode\hskip-\@tempdima{#6}\nobreak
    \leaders\hbox{$\m@th\mkern \@dotsep mu\hbox{.}\mkern \@dotsep mu$}\hfill
    \nobreak
    \hbox to\@pnumwidth{\@tocpagenum{#7}}\par
    \nobreak
    \endgroup
  \fi}
\makeatother
\AtBeginDocument{%
\makeatletter
\expandafter\renewcommand\csname r@tocindent0\endcsname{0pt}
\makeatother
}
\def\l@subsection{\@tocline{2}{0pt}{2.5pc}{5pc}{}}

\newtheorem{theorem}{Theorem}[section]
\newtheorem{proposition}[theorem]{Proposition}
\newtheorem{lemma}[theorem]{Lemma}
\newtheorem{corollary}[theorem]{Corollary}
\newtheorem{definition}[theorem]{Definition}
\newtheorem{remark}[theorem]{Remark}

\newtheorem{claim}[theorem]{Claim}

\newcommand{\fb}{\mathfrak{b}}

\newcommand{\Z}{\mathbb{Z}}

\newcommand{\stkout}[1]{\ifmmode\text{\sout{\ensuremath{#1}}}\else\sout{#1}\fi}

\makeatletter
\@addtoreset{equation}{section}
\makeatother

\renewcommand\thetable{\thesection.\@arabic\c@table}

\bibdata{prob}
\bibstyle{alpha}

\title[Drainage Network Model with Branching]{On the convergence of the drainage network with branching}

\author[R. Santos, G. Valle, L. Zuazn\'abar]{Rafael Santos$^1$, Glauco Valle$^2$, Leonel Zuazn\'abar$^3$}
\thanks{1. Supported by CAPES and FAPERJ}
\thanks{2. Supported Universal CNPq project 421383/2016-0, CNPq grant 307938/2022-0 and FAPERJ grant E-26/202.636/2019.}

\date{\today}

\address{
\newline
\newline
$^1$ UFRJ - Instituto de Matem\'atica.
\newline  Caixa Postal 68530, 21945-970, Rio de Janeiro, Brasil.
\newline
e-mail: {\rm \texttt{rafaels@dme.ufrj.br}}
\newline
\newline
$^2$ UFRJ - Instituto de Matem\'atica.
\newline  Caixa Postal 68530, 21945-970, Rio de Janeiro, Brasil.
\newline
e-mail: {\rm \texttt{glauco.valle@im.ufrj.br}}
\newline
\newline
$^3$ 
UFABC - Centro de Matem\'atica, Computa\c{c}\~ao e Cogni\c{c}\~ao.
\newline   Av.
dos Estados, 5001, 09210-580 Santo Andr\'e, S\~ao Paulo, Brasil.
\newline
e-mail: {\rm \texttt{l.zuaznabar@ufabc.edu.br}}
}


\begin{document}

\begin{abstract}
The Drainage Network is a system of coalescing random walks, exhibiting long-range dependence before coalescence, introduced by Gangopadhyay, Roy, and Sarkar \cite{grs}. Coletti, Fontes, and Dias \cite{cfd} proved its convergence to the Brownian Web under diffusive scaling. In this work, we introduce a perturbation of the system allowing branching of the random walks with low probabilities varying with the scaling parameter. When the branching probability is inversely proportional to the scaling parameter, we show that this drainage network with branching consists of a tight family such that any weak limit point contains a Brownian Net. We conjecture that the limit is indeed the Brownian Net.


\end{abstract}

\maketitle

\tableofcontents

\section{Introduction}
\label{sec:intro}

Here we study convergence in distribution for a diffusively scaled system of non-crossing one-dimensional branching-coalescing random walks starting at each point in the space and time lattice $\mathbb{Z} \times \mathbb{Z}$ that presents long-range dependence before coalescence. To not be repetitive, when we make mention of the convergence of a system of random walks, we always consider one-dimensional random walks and diffusive space-time scaling. For systems of coalescing random walks without branching and independence before coalescence, the limiting system is the so-called Brownian Web introduced by Fontes et al in \cite{finr,finr1}. Although formal descriptions of systems of coalescing Brownian motions had previously been considered initially by Arratia \cite{a,a1} and then by Toth and Werner \cite{tw}. Since \cite{finr,finr1}, the convergence in distribution of systems of coalescing random walks to the Brownian Web and its variants has been an object of extensive study, as we can see in \cite{bmsv,cssc,finr2,nrs,rss,ss,ss2}.

Motivated by an application that appears in Scheidegger \cite{scheidegger} about drainage pattern into an intramontane trench which is related to river networks, Roy, Saha, and Sarkar in \cite{grs} introduced the Drainage Network, see also \cite{dual}. It arises as a natural model for systems of coalescing random walks with dependence before coalescence. We can describe the model as follows: Each site in $\mathbb{Z}^2$ is independently open or closed with a given probability $p \in (0,1)$. From each site $(x,y) \in \mathbb{Z}^2$ departs a directed edge towards the nearest open site in $\mathbb{Z} \times \{y+1\}$, in case there is a unique nearest open site. If we have two options, necessarily of the form $(x+z,y+1)$ and $(x-z,y+1)$ for some $z\in \mathbb{Z}$, then we choose one of them with equal probability. Considering the first coordinate as space and the second one as time, we obtain a system of coalescing random paths called the Drainage Network. Its convergence to Brownian Web was proved first by Coletti, Fontes and Dias \cite{cfd} and later by Coletti and Valle \cite{cv} for a more general version of the model that allows paths to cross each other. 

The convergence of systems of branching-coalescing random walks have been later studied by Sun and Swart \cite{ss1}. They considered systems of coalescing random walks which evolve independently before coalescence with a superposed dynamic that allows the branching of paths with a small probability inversely proportional to the scaling parameter. They also characterized the diffusive limit as a system of branching-coalescing Brownian motions, which they called the Brownian Net. An alternative construction of this random object was given by Newman, Ravishankar and Schertzer in \cite{nrsc}. Besides that, Etheridge, Freeman, and Straulino proved convergence to Brownian Net for spatial $\Lambda$-Fleming-Viot process without selection in \cite{efs}. More information related to Brownian Web and Net can be found in \cite{survey1}.

Considering the Drainage Network again, it is natural to consider an extension of the model where branching can occur. When there are two nearest open sites as destination for an edge departing from a given site, we allow both edges to exist in the system with some given probability $\epsilon_n$, where $n$ is a scaling parameter. We call $\epsilon_n$ the branching probability. Trajectories of the Drainage Network with branching also exhibit long-range dependence. In this work, we study the convergence of the Drainage Network with branching under diffusive scaling to the Brownian Net under specific conditions on the branching probabilities. If $\epsilon_n$ is inversely proportional to $n$, we prove that the Drainage Network is a tight family such that any of its weak limit points contain a Brownian Net. We conjecture that the limit is indeed the Brownian Net, but we could not prove it in this article. 

In a further generalization, we can consider the branching parameter of order $n^{-\alpha}$ for some $\alpha > 0$. If $\alpha < 1$ it is simple to verify that the system does not converge under diffusive scaling. If $\alpha> 1$ we conjecture that the system converges to the Brownian Web, the result is straightforward for $\alpha > 2$ and we shall discuss the case $1<\alpha \le 2$ in another paper.

Our result contributes to the understanding of the universality class related to the Brownian Net. We believe that, like the Brownian Web, the convergence for the Brownian Net is robust and only needs asymptotic independence of the trajectories in the system and proper moment conditions.

The paper is organized as follows: In Section \ref{sec:BNet}, we describe the Brownian Web and Brownian Net. In Section \ref{sec:DNB}, we describe the Drainage Network with Branching and its dual and we also state our main results. In Section \ref{sec:coalescence}, some basic properties of the Drainage Network with branching are obtained. Finally, in Sections \ref{sec:conditionI} and \ref{sec:lefttobrownian}, the main theorems are proved.


\section{Brownian Web and Brownian Net}
\label{sec:BNet}

\vspace{0.2cm}

\subsection{The space of compact sets of paths}\
\label{subsec:compactification}

\vspace{0.2cm}

The Brownian Web and the Brownian Net are random elements of a proper metric space introduced in \cite{finr,finr1}. It is a space of sets of space-time continuous paths on a compactification of the space-time plane $\mathbb{R}^2$ endowed with the Hausdorff topology for the uniform metric on the compactification. We follow the definition in \cite{ss1}, which is equivalent but slightly different from that in \cite{finr2}, for details see \cite[Appendix]{ss1}.

Let $\bar{\mathbb{R}}^2$ be the compactification of the space-time plane $\mathbb{R}^2$ under the metric $\rho$ where:
\begin{equation}
\rho((x_1,t_1),(x_2,t_2)) = \left| \mbox{tanh}(t_1) - \mbox{tanh}(t_2) \right| \ \vee \ \left| \frac{\mbox{tanh}(x_1)}{1+|t_1|} - \frac{\mbox{tanh}(x_2)}{1+|t_2|} \right|,
\label{eq:metricpoint} 
\end{equation}  
and tanh$(x)$ is the hyperbolic tangent of $x$: tanh$(x) = \displaystyle\frac{e^x-e^{-x}}{e^x+e^{-x}}$. More precisely, consider 
$\big([-\infty,\infty] \times (-\infty,\infty)\big) \cup \{ (*,\pm \infty) \},$
endowed with a topology such that $(x_n,t_n) \rightarrow (*,\pm \infty)$ whenever $t_n \rightarrow \pm \infty$. Then $\bar{\mathbb{R}}^2$ can be thought of as the continuous image of this space under the map
\begin{equation}
(x,t) \in [-\infty,\infty] \times (-\infty,\infty) \mapsto \left(\frac{\mbox{tanh}(x)}{1+|t|},\mbox{tanh}(t) \right), 
\label{eq:metricmap} 
\end{equation}  
$$
(*, -\infty) \mapsto (0,-1) \ \textrm{ and } \
(*, \infty) \mapsto (0,1).
$$
Note that this mapping gives a region contained in the square $[-1,1] \times [-1,1]$. 

A path $\pi$ in $\bar{\mathbb{R}}^2$ with starting time denoted by $\sigma_{\pi} \in [-\infty,\infty]$, is a map $\pi$: $[\sigma_{\pi},\infty] \rightarrow [-\infty,\infty]\cup\{*\}$ such that $\pi(\infty) = *$, $\pi(\sigma_{\pi}) = *$ if $\sigma_{\pi} = -\infty$ and $t \rightarrow (\pi(t),t)$ is a continuous map from $[\sigma_{\pi},\infty]$ to $(\bar{\mathbb{R}}^2,\rho)$. Then define $\Pi$ to be the space of all paths $\pi$ in $\bar{\mathbb{R}}^2$ endowed with the metric $d$, where $d(\pi_1,\pi_2)$ is given by    
\begin{equation}
\left| \mbox{tanh}(\sigma_{\pi_1}) - \mbox{tanh}(\sigma_{\pi_2}) \right| \ \vee \ \sup_{t \geq (\sigma_{\pi_1} \wedge \sigma_{\pi_2})} \left| \frac{\mbox{tanh}(\pi_1(t \vee \sigma_{\pi_1}))}{1+|t|} - \frac{\mbox{tanh}(\pi_2(t \vee \sigma_{\pi_2}))}{1+|t|} \right|.
\label{eq:metricpath} 
\end{equation}  
We have that $(\Pi,d)$ is a complete separable metric space. 

Finally, we can define the space $(\mathcal{H},d_{\mathcal{H}})$, where the Brownian Web and Net will be defined. Consider $K \subset \Pi$ and let $\mathcal{H}$ denote the space of compact subsets of $(\Pi,d)$, under the Hausdorff metric $d_{\mathcal{H}}$, that is: 
\begin{equation}
d_{\mathcal{H}}(K_1,K_2) = \sup_{\pi_1 \in K_1} \inf_{\pi_2 \in K_2} d(\pi_1,\pi_2) \vee \sup_{\pi_2 \in K_2} \inf_{\pi_1 \in K_1} d(\pi_1,\pi_2), \ \, K_1, \, K_2 \in \mathcal{H}.
\label{eq:metricweb} 
\end{equation}  
We have that $(\mathcal{H},d_{\mathcal{H}})$ is also a complete separable metric space.

Letting $\mathcal{B}_{\mathcal{H}}$ denote the Borel $\sigma$-algebra associated with the metric $d_{\mathcal{H}}$, we will construct the Brownian Web $\mathcal{W}$ and Brownian Net $\mathcal{N}$ as random elements defined in $(\mathcal{H},\mathcal{B}_{\mathcal{H}})$.

\subsection{The Brownian Web and its dual}\

\vspace{0.2cm}

The Brownian Web $\mathcal{W}$ is a random element of  $(\mathcal{H},\mathcal{B}_{\mathcal{H}})$ characterized by the following: 

\smallskip

\begin{proposition}
\label{Result:brownianweb}
\textbf{(\cite[Theorem $2.1$]{finr1})} There exists an $(\mathcal{H},\mathcal{B}_{\mathcal{H}})$-valued random variable $\mathcal{W}$, called the standard Brownian Web, whose distribution is uniquely determined by the following properties:
\begin{enumerate}
\item[(a)] From any deterministic point $z \in \mathbb{R}^2$, there is almost surely a unique path $\pi_z \in \mathcal{W}$ starting from $z$; \
\item[(b)] For any finite deterministic set of points $z_1, \ldots , z_k \in \mathbb{R}^2$, the collection $(\pi_{z_1}, \ldots , \pi_{z_k})$ is distributed as coalescing standard Brownian motions;  
\item[(c)] For any deterministic countable dense subset $D \subset \mathbb{R}^2$, almost surely, $\mathcal{W}$ is the closure of $\{\pi_z: z \in D \}$ in $(\Pi,d)$.
\end{enumerate}
\end{proposition}

The characterization in Proposition \ref{Result:brownianweb} can be extended to allow the Brownian paths to have a fixed diffusion coefficient $\lambda^2 \neq 1$ and a drift $b \neq 0$. The only difference is in the property (b), where the coalescing Brownian motions may have a diffusion coefficient distinct from one and a non-zero drift. We denote by $\mathcal{W}_{\lambda,b}$ the \emph{Brownian Web with diffusion coefficient $\lambda^2 > 0$ and drift $b\in \mathbb{R}$} and $\mathcal{W}_{\lambda,b} = \mathcal{W}_{\lambda}$ if $b=0$, see \cite[Theorem 1.5]{ss1} for a formal construction. One can show that $\mathcal{W}_{\lambda,b}$ can be obtained as the image of $\mathcal{W}$ by a proper map on $(\mathcal{H},d_{\mathcal{H}})$. This map is defined as
$$
(\pi(t))_{t\ge \sigma_\pi} \in \Pi \mapsto 
(\lambda \pi(t) + b (t-\sigma_\pi))_{t\ge \sigma_\pi} \in \Pi.
$$ 

The Brownian Web $\mathcal{W}$ has a dual process $\widehat{\mathcal{W}}$ which is called the \emph{dual Brownian Web}. $\widehat{\mathcal{W}}$ is a collection of coalescing paths running backward in time which is uniquely determined by the restriction that these paths cannot cross any path from $\mathcal{W}$.

Let us describe the space where $\widehat{\mathcal{W}}$ takes its values. Given a set $A \subset \bar{\mathbb{R}}^2$, denote $-A = \{-z: z \in A\}$. Identifying each path $\pi \in \Pi$ with its graph as a subset of $\bar{\mathbb{R}}^2$, $\widehat{\pi} = -\pi$ defines a path running backward in time, with starting time $\widehat{\sigma}_{\widehat{\pi}} = -\sigma_{\pi}$. Let $\widehat{\Pi} = - \Pi$ denoting the set of all these backward paths endowed with a metric $\widehat{d}$ inherited from $(\Pi,d)$ under the sign change mapping. Let $\widehat{\mathcal{H}}$ be the space of compact subsets of $(\widehat{\Pi},\widehat{d})$ endowed with the Hausdorff metric $d_{\widehat{\mathcal{H}}}$ and Borel $\sigma$-algebra $\mathcal{B}_{\widehat{\mathcal{H}}}$. For any $K \in \mathcal{H}$, $-K$ denotes the set $\{-\pi: \pi \in K\} \in \widehat{\mathcal{H}}$.

The next proposition characterizes the joint law of the Brownian Web $\mathcal{W}$ and its dual $\widehat{\mathcal{W}}$ as random elements taking values in $\mathcal{H} \times \widehat{\mathcal{H}}$ equipped with the product $\sigma$-algebra.

\smallskip

\begin{proposition}
\label{Result:dualbrownianweb}
\textbf{(\cite[Theorem 2.4]{survey1})} There exists an $\mathcal{H} \times \widehat{\mathcal{H}}$-valued random element $(\mathcal{W},\widehat{\mathcal{W}})$, called the double Brownian Web (with $\widehat{\mathcal{W}}$ called the dual Brownian Web), whose distribution is uniquely determined by the following properties:
\begin{enumerate}
\item[(a)] $\mathcal{W}$ and $-\widehat{\mathcal{W}}$  are both distributed as the standard Brownian Web;
\item[(b)] Almost surely, no path $\pi_z \in \mathcal{W}$ crosses any path $\widehat{\pi}_{\widehat{z}} \in \widehat{\mathcal{W}}$, i.e., there are no $z=(x,t)$ and $\widehat{z} = (\widehat{x},\hat{t})$ with $t < \widehat{t}$ such that $(\pi_z(s_1)-\widehat{\pi}_{\widehat{z}}(s_1))(\pi_z(s_2)-\widehat{\pi}_{\widehat{z}}(s_2)) < 0$ for some $t < s_1 < s_2 < \widehat{t}$.
\end{enumerate}
Furthermore, for each $z \in \mathbb{R}^2$, $\widehat{\mathcal{W}}(z)$ a.s. consists of a single path $\widehat{\pi}_z$ which is the unique path in $\widehat{\Pi}$ that does not cross any path in $\mathcal{W}$. Thus $\widehat{\mathcal{W}}$ is a.s. determined by $\mathcal{W}$ and vice versa.
\end{proposition}

Analogously, we can define the double Brownian Web with diffusion coefficient $\lambda^2$ and drift $b$, which we denote by $(\mathcal{W}_{\lambda,b}, \widehat{\mathcal{W}}_{\lambda,b})$.

\subsection{Convergence criteria for the Brownian Web}\
\label{subsec:webcritterion}


The convergence criterion for the Brownian Web was first described in \cite[Theorem 2.2]{finr1}, except for a tightness condition which we call below Condition (T). This condition is unnecessary for systems with non-crossing paths (See \cite[Proposition B2]{finr1}). Condition (T) appears in \cite[Theorem 1.4]{nrs} for the first time.

Consider a sequence of $(\mathcal{H},\mathcal{B}_{\mathcal{H}})$-valued random variables $(\mathcal{Z}_n)_{n \in \mathbb{N}}$. There are four conditions that, if satisfied, guarantees that $(\mathcal{Z}_n)_{n \in \mathbb{N}}$ converges to a Brownian Web $\mathcal{W}_\lambda$:

\vspace{0.2cm}

\noindent \textbf{Condition (T)}: [Tightness criterion] The law of a sequence of $(\mathcal{H},\mathcal{B}_{\mathcal{H}})$-valued random variables $(\mathcal{Z}_n)_{n \in \mathbb{N}}$ is tight if for all $L>0$, $T>0$ and $\eta > 0$,
$$ 
\lim_{\delta \downarrow 0} \delta^{-1} \limsup_{n\rightarrow \infty} \sup_{(x,t) \in [-L,L]\times [-T,T]} P(\mathcal{Z}_n \in A_{\delta,\eta}(x,t)) = 0,
$$
where $A_{\delta,\eta}(x,t)$ is the event (in $\mathcal{B}_{\mathcal{H}}$) that $K \in \mathcal{H}$ contains some path which intersects the rectangle $[x-\frac{\eta}{2} , x+\frac{\eta}{2}] \times [t,t+\frac{\delta}{2}]$, and at a later time, intersects the left or right boundary of a larger rectangle $[x-\eta, x+\eta] \times [t,t+\delta]$.

\vspace{0.2cm}

\noindent \textbf{Condition (I)}: There exists $\pi_{n,z} \in \mathcal{Z}_n$ for each $z \in \mathbb{R}^2$, such that for any deterministic $z_1, \ldots, z_k \in \mathbb{R}^2$, $(\pi_{n,z_i})_{1 \leq i \leq k}$ converge in distribution to coalescing Brownian motions with diffusion coefficient $\lambda^2$ starting at $(z_i)_{1 \leq i \leq k}$.

\vspace{0.2cm}

\noindent \textbf{Condition (B1)}: For all $t>0$
$$
\limsup_{\delta \downarrow 0} \limsup_{n \rightarrow \infty} \sup_{a,t_0 \in \mathbb{R}} P[\eta_{\mathcal{Z}_n}(t_0,t;a,a+\delta) \geq 2 ] =0, 
$$
where given a $(\mathcal{H},\mathcal{B}_{\mathcal{H}})$-valued random element $\mathcal{X}, t_0 \in \mathbb{R},t>0,a<b,$ 
\begin{equation*}
	\eta_{\mathcal{X}}(t_0,t;a,b) = |\{\pi(t_0+t): \pi \in \mathcal{X}, \pi(t_0) \in [a,b]\}|\, ,
\end{equation*}
which counts the number of distinct paths in $\mathcal{X}$ at time $t_0+t$, among all paths that start from interval [a,b] at time $t_0$.

\vspace{0.2cm}

\noindent \textbf{Condition (B2)}: For all $t>0$ 
$$
\limsup_{\delta \downarrow 0} \frac{1}{\delta} \limsup_{n \rightarrow \infty} \sup_{a,t_0 \in \mathbb{R}} P[\eta_{\mathcal{Z}_n}(t_0,t;a,a+\delta) \geq 3 ] = 0.
$$

\vspace{0.2cm}

To conclude, it will be useful in our proofs that Condition (B2) can be replaced by condition (E$'$) described below, \cite[Theorem 1.4 and Lemma 6.1]{nrs}. (E$'$) was introduced in \cite{nrs} to prove Condition (B2) for systems that allow crossings of their paths. 

\noindent \textbf{Condition (E$'$)}: Let $\mathcal{Z}_n^{t_0^{-}}$ be the subset of paths in $\mathcal{Z}_n$ which starts before or at time $t_0$. If $\mathcal{Z}^{t_0}$ is any subsequential limit of $\mathcal{Z}_n^{t_0^{-}}$ for any $t_0 \in \mathbb{R}$, then $\forall t,a,b \in \mathbb{R}$ with $t>0$ and $a<b$, we have:
$$E[\hat{\eta}_{\mathcal{Z}^{t_0}}(t_0,t,a,b)] \leq E[\hat{\eta}_{\mathcal{W}}(t_0,t,a,b)] = \frac{b-a}{\sqrt{\pi t}},$$ 
where
	\begin{equation*}
		\hat{\eta}_{\mathcal{Z}}(t_0,t;a,b) = |\{ \pi(t_0+t) \cap (a,b): \pi \in \mathcal{Z}, \pi(t_0) \in \mathbb{R} \}|\, ,
	\end{equation*}
which counts the number of points in interval $(a,b)$ touched at time $t_0+t$ by paths in $\mathcal{Z}$ starting before or at time $t_0$. 

\subsection{The construction of the Brownian Net}\
\label{subsec:Bnet}

\vspace{0.2cm}

Here we define the Brownian Net, which generalizes the Brownian Web by allowing paths to branch. It was introduced by Rongfeng Sun and J. M. Swart in \cite{ss1}. The first step is to describe what is called the left-right Brownian Web. Consider a collection of space-time random paths characterized as a random element $(l_z,r_z)_{z\in \mathbb{R}^2}$ of $\mathcal{H} \times \mathcal{H}$ where $l_{z}$ and $r_z$ start both at $z$ for each $z\in \mathbb{R}^2$. We call $l_{z}$ an \emph{l-path} (left path) and $r_z$ an \emph{r-path} (right path). 
Also consider the following finite dimensional distributions: For any two finite collections of points $(z_i)_{1 \leq i \leq k}$ and $(z'_j)_{1 \leq j \leq k'}$ in $\mathbb{R}^2$ 
\begin{itemize}
\item The paths $(l_{z_1}, \ldots, l_{z_k}, r_{z'_1}, \ldots, r_{z'_{k'}})$ evolve independently until they meet each other.\

\item The l-paths $(l_{z_1}, \ldots, l_{z_k})$ coalesce when they meet and the same is true for the r-paths $(r_{z'_1}, \ldots, r_{z'_{k'}})$.\

\item Each pair of left-right paths $(l_{z_i},r_{z'_j})$ solves the following system of SDEs:
\begin{equation}
\begin{cases} dL_t = I_{\{L_t \neq R_t\}} dB_t^l + I_{\{L_t = R_t\}} dB_t^s - dt,\\ dR_t = I_{\{L_t \neq R_t\}} dB_t^r + I_{\{L_t = R_t\}} dB_t^s + dt, \end{cases}
\label{eq:diferentialeq} 
\end{equation}  
where the l-path $L$ and the r-path $R$ are driven by independent Brownian motions $B^l$ and $B^r$ when they are apart and driven by the same Brownian motion $B^s$ (independent of $B^l$ and $B^r$) when they coincide. Furthermore, $L$ and $R$ have the restriction that $L_t \leq R_t$ for all $t \geq \inf \{u \geq \sigma_L \vee \sigma_R: L_u \leq R_u \}$, where $\sigma_L$ and $\sigma_R$ being the starting times of $L$ and $R$.\ 
\end{itemize}  
The pair of SDEs \eqref{eq:diferentialeq} has a unique solution, and the above properties uniquely determine the joint law of $(l_{z_1}, \ldots, l_{z_k}, r_{z'_1}, \ldots, r_{z'_{k'}})$, which are called \emph{left-right coalescing Brownian motions}. Extending the starting points to a countable dense set in $\mathbb{R}^2$ and then taking the closure of the resulting set of l-paths and r-paths a.s. determines a random element of $\mathcal{H}^2$ denoted by $(\mathcal{W}^l,\mathcal{W}^r)$, which is called the \emph{left-right Brownian Web}.

The following result (see \cite[Theorem $1.5$]{ss1} or \cite[Theorem $3.2$]{survey1}) characterized the left-right Brownian Web and its dual.

\begin{proposition}
\label{Result:leftrightweb}
\textbf{\cite[Theorem $1.5$]{ss1}} There exists an $\mathcal{H}^2$-valued random variable $(\mathcal{W}^l,\mathcal{W}^r)$, called the (standard) left-right Brownian Web, whose distribution is uniquely determined by the following properties: 
\begin{enumerate}
\item[(i)] For each deterministic $z \in \mathbb{R}^2$, $\mathcal{W}^l$ and $\mathcal{W}^r$ almost surely contain a single path each, that starts from the vertex $z$;

\item[(ii)] For any finite deterministic set of points $z_1, \ldots z_k, z'_1, \ldots z'_{k'} \in \mathbb{R}^2$, the collection of paths $(l_{z_1}, \ldots , l_{z_k},r_{z'_1}, \ldots , r_{z'_{k'}})$ is distributed as a family of left-right coalescing Brownian motions. 

\item[(iii)] For any deterministic countable dense sets $D^l,D^r \subset \mathbb{R}^2$,
$$\mathcal{W}^l = \overline{\{ l_z: z \in D^l \}} \mbox{ and } \mathcal{W}^r = \overline{\{ r_z: z \in D^r \}} \ \mbox{a.s.}$$
\end{enumerate}
\end{proposition}

\medskip

Recall that there exists a \emph{dual left-right Brownian Web} $(\widehat{\mathcal{W}^l},\widehat{\mathcal{W}}^r) \in \widehat{\mathcal{H}}^2$, such that $(\mathcal{W}^l,\widehat{\mathcal{W}}^l)$ (resp. $(\mathcal{W}^r,\widehat{\mathcal{W}}^r)$) is distributed as $(\mathcal{W},\widehat{\mathcal{W}})$ tilted with drift $-1$ (resp. $+1$). From Section 1.8 in \cite{ss1}, we have that $-(\mathcal{W}^l,\mathcal{W}^r,\widehat{\mathcal{W}}^l,\widehat{\mathcal{W}}^r)$) and
$(\widehat{\mathcal{W}}^l,\widehat{\mathcal{W}}^r,\mathcal{W}^l,\mathcal{W}^r)$) are identically distributed.

We can extend Proposition \ref{Result:leftrightweb} to the non-standard case. For fixed constants $\lambda > 0$ and $b>0$, we denote by $(\mathcal{W}_{\lambda,b}^l,\mathcal{W}_{\lambda,b}^r)$ the \emph{left-right Brownian Web with diffusion coefficient $\lambda^2$ and drift coefficient $b$}. Here we consider that the independent Brownian motions $B^l$, $B^r$, and $B^s$ have diffusion coefficient $\lambda^2$ and that in place of \eqref{eq:diferentialeq}, we have
\begin{equation}
\begin{cases} dL_t = I_{\{L_t \neq R_t\}} dB_t^l + I_{\{L_t = R_t\}} dB_t^s - b \, dt,\\ dR_t = I_{\{L_t \neq R_t\}} dB_t^r + I_{\{L_t = R_t\}} dB_t^s + b \, dt. \end{cases}
\label{eq:diferentialeq-b} 
\end{equation} 
This same extension can be done with the \emph{double left-right Brownian Web} (i.e. the joint process of the left-right Brownian Web and its dual)
$(\mathcal{W}^l_{\lambda,b},\mathcal{W}^r_{\lambda,b},\widehat{\mathcal{W}}^l_{\lambda,b},\widehat{\mathcal{W}}^r_{\lambda,b})$.

\medskip

Now we are ready to construct the Brownian Net $\mathcal{N}_{\lambda,b}$ by allowing paths hopping between paths in the left-right Brownian Web $(\mathcal{W}^l_{\lambda,b},\mathcal{W}^r_{\lambda,b})$. If two paths $\pi_1$ and $\pi_2$ satisfy $\pi_1(t) = \pi_2(t)$, we say that a path $\pi$ is obtained by hopping from $\pi_1$ to $\pi_2$ at time $t$ when $\pi = \pi_1$ on $[\sigma_{\pi_1},t]$ and $\pi = \pi_2$ on $[t,\infty]$. We also say that a time $t > \sigma_{\pi_1} \vee \sigma_{\pi_2}$ is a crossing time between $\pi_1$ and $\pi_2$ only if these two paths cross each other at this time, i.e., there exist times $t^{-} < t^{+}$ satisfying $(\pi_1(t^{-}) - \pi_2(t^{-}))(\pi_1(t^{+}) - \pi_2(t^{+})) < 0$ and 
$$
t = \displaystyle \inf_{s \in (t^{-},t^{+})} \{ (\pi_1(t^{-}) - \pi_2(t^{-}))(\pi_1(s) - \pi_2(s)) < 0 \}.
$$

\medskip

Given a set of paths $K \in \mathcal{H}$, denote by $\mathcal{H}_{cross}(K)$ the set of paths obtained by hopping a finite number of times among paths in $K$ at crossing times. The Brownian Net is a $(\mathcal{H},\mathcal{B}_{\mathcal{H}})$-valued random variable that can be constructed by setting $\mathcal{N}_{\lambda,b} =  \overline{\mathcal{H}_{cross}(\mathcal{W}^l_{\lambda,b} \cup \mathcal{W}^r_{\lambda,b})}$.

There are several ways to characterize the Brownian Net: hopping characterization, wedge characterization, and mesh characterization (see \cite{ss1}), but we will focus only on the first two characterizations, which are directly related to the convergence study of this work. We begin with the hopping characterization. The following result is from \cite[Theorem 1.3]{ss1}.

\begin{proposition}
\label{Result:hopping}
\textbf{(\cite[Theorem $1.3$]{ss1})} For every $\lambda > 0$ and $b>0$, there exists an $(\mathcal{H},\mathcal{B}_{\mathcal{H}})$-valued random variable $\mathcal{N}_{\lambda,b}$, called Brownian Net with diffusion coefficient $\lambda^2$ and drift coefficient $b$, whose distribution is uniquely determined by the following properties:
\begin{enumerate}
\item[(a)] For each $z \in \mathbb{R}^2$, $\mathcal{N}_{\lambda,b}$ a.s. contains a unique l-path $l_z$ and r-path $r_z$ that start at $z$;
\item[(b)] For any finite deterministic set of points $z_1, \ldots , z_k, z'_1, \ldots , z'_{k'} \in \mathbb{R}^2$, the random vector of paths $(l_{z_1}, \ldots , l_{z_k}, r_{z'_1}, \ldots r_{z'_{k'}})$ is distributed as a family of left-right coalescing Brownian motions with diffusion coefficient $\lambda^2$ and drift coefficient $b$;  
\item[(c)] For any deterministic countable dense sets $D^l, D^r \subset \mathbb{R}^2$,
$$
\mathcal{N}_{\lambda,b} = \overline{\mathcal{H}_{cross}(\{l_z: z \in D^l \} \cup \{ r_z: z \in D^r \})} \ \ \ \ \ \mbox{a.s.}
$$ 
\end{enumerate}
\end{proposition}

\vspace{0.2cm}

The wedge's construction is based on the existence of certain forbidden regions in the space-time plane where the Brownian Net paths cannot enter. These regions are called wedges. The dual left-right Brownian Web ($\widehat{\mathcal{W}^l_{\lambda,b}},\widehat{\mathcal{W}}^r_{\lambda,b}$)  determines them as follows:

\begin{definition}
\label{Definition:Wedge}
Let $(\mathcal{W}^l_{\lambda,b},\mathcal{W}^r_{\lambda,b},\widehat{\mathcal{W}}^l_{\lambda,b},\widehat{\mathcal{W}}^r_{\lambda,b})$ be a double left-right Brownian Web. For any path $\hat{r} \in \widehat{\mathcal{W}}^r_{\lambda,b}$ and $\hat{l} \in \widehat{\mathcal{W}}^l_{\lambda,b}$ that are ordered with $\hat{r}(s) < \hat{l}(s)$ at the time $s = \hat{\sigma}_{\hat{r}} \wedge \hat{\sigma}_{\hat{l}}$, let $T = \displaystyle \sup_{t<s}\{ \hat{r}(t) = \hat{l}(t)\}$ be the first hitting time of $\hat{r}$ and $\hat{l}$ (possibly equals $-\infty$, in case they never meet). We call the open set:
\begin{equation}
\Psi = \Psi(\hat{r},\hat{l}) = \{(x,u) \in \mathbb{R}^2: T < u < s, \hat{r}(u) < x < \hat{l}(u) \}
\label{eq:wedge} 
\end{equation}  

\vspace{0.2cm}

\noindent a wedge of $(\widehat{\mathcal{W}}^l_{\lambda,b},\widehat{\mathcal{W}}^r_{\lambda,b})$ with left boundary $\hat{r}$, right boundary $\hat{l}$ and bottom point $z = (\hat{r}(T),T) = \hat{l}(T),T)$. A path $\pi \in \Pi$ is said to enter a wedge from outside if there exist times $t,s$ with $t > s \geq \sigma_{\pi}$ such that $(\pi(s),s) \notin \overline{\Psi}$ and $(\pi(t),t) \in \Psi$.

\end{definition}

\vspace{0.3cm}

Now we are ready to present the wedge characterization of the Brownian Net. The following result is from \cite[Theorem 1.10]{ss1}.

\begin{proposition}
\label{Result:wedge}
\textbf{(Theorem $1.10$ of \cite{ss1})} Let $(\mathcal{W}^l_{\lambda,b},\mathcal{W}^r_{\lambda,b},\widehat{\mathcal{W}}^l_{\lambda,b},\widehat{\mathcal{W}}^r_{\lambda,b})$ be a double left-right Brownian Web. Then almost surely,\
$$
\mathcal{N}_{\lambda,b} = \{ \pi \in \Pi: \pi \mbox{ does not enter any wedge of } (\widehat{\mathcal{W}}^l_{\lambda,b},\widehat{\mathcal{W}}^r_{\lambda,b}) \mbox{ from outside} \}
$$
is the Brownian Net associated with $(\widehat{\mathcal{W}}^l_{\lambda,b},\widehat{\mathcal{W}}^r_{\lambda,b})$, i.e., $\mathcal{N}_{\lambda,b} = \overline{\mathcal{H}_{cross}(\mathcal{W}^l_{\lambda,b} \cup \mathcal{W}^r_{\lambda,b})}$.
\end{proposition}

\vspace{0.2cm}

\subsection{Convergence criteria for the Brownian Net}\
\label{subsec:browniannetcritter}

\vspace{0.2cm}

Convergence to Brownian Net has only been established for random sets of paths with the non-crossing property so far. In \cite{survey1}, the authors presented a list of sufficient conditions to prove convergence to the Brownian Net, which we will state below.

Consider a sequence of $(\mathcal{H},\mathcal{B}_{\mathcal{H}})$-valued random variables $(\mathcal{Y}_n)_{n \in \mathbb{N}}$. Based on the existence of subsets of non-crossing paths $W_n^l, \, W_n^r \subset \mathcal{Y}_n$, there are five conditions that, if satisfied, guarantees that $(\mathcal{Y}_n)_{n \in \mathbb{N}}$ converges to the Brownian Net $\mathcal{N}_{\lambda,b}$.

The first condition is related to the non-crossing restriction:

\vspace{0.2cm}

\noindent - \textbf{Condition (C)}: No path $\pi \in \mathcal{Y}_n$ crosses any $l \in W_n^l$ from right to left (i.e. there are no $s < t$ such that $\pi(s) > l(s)$ and $\pi(t) < l(t)$) and no path $\pi \in \mathcal{Y}_n$ crosses any $r \in W_n^r$ from left to right.

The second condition ensures that $(W_n^l, W_n^r)_{n \in \mathbb{N}}$ is a tight family and that any subsequential weak limit contains a copy of the left-right Brownian Web. It is called Condition (I$_{\mathcal{N}}$):

\vspace{0.2cm}

\noindent - \textbf{Condition (I$_{\mathcal{N}}$)}: There exist $l_{n,z} \in W_n^l$ and $r_{n,z} \in W_n^r$ for each $z \in \mathbb{R}^2$, such that for any deterministic $z_1, \ldots , z_k, z'_1, \ldots , z'_{k'} \in \mathbb{R}^2$,   $(l_{n,z_1}, \ldots , l_{n,z_k}, r_{n,z'_1}, \ldots , r_{n,z'_{k'}})$ converges in distribution to a random vector of paths $(l_{z_1}, \ldots , l_{z_k}, r_{z'_1}, \ldots r_{z'_{k'}})$ distributed as a family of left-right coalescing Brownian motions with diffusion coefficient $\lambda^2$ and drift coefficient $b$.

Note that, combining the tightness of $(W_n^l, W_n^r)_{n \in \mathbb{N}}$ with Condition (C) implies that $(\mathcal{Y}_n)_{n \in \mathbb{N}}$ is also a tight family. Indeed Condition (C) ensures that, almost surely, the modulus of continuity of paths in $\mathcal{Y}_n$ can be bounded by the modulus of continuity of paths in $W_n^l \cup W_n^r$, whose starting points become dense in $\mathbb{R}^2$ as $n \rightarrow \infty$ by Condition (I$_{\mathcal{N}}$). So tightness of $(\mathcal{Y}_n)_{n \in \mathbb{N}}$ is a consequence of conditions (C) and (I$_{\mathcal{N}}$).

The third condition ensures that any subsequential weak limit of $(\mathcal{Y}_n)_{n \in \mathbb{N}}$ contains not only a copy of the left-right Brownian Web but also a copy of the Brownian Net constructed by hopping among paths in $\mathcal{W}^l \cup \mathcal{W}^r$ at crossing times. This is Condition (H):

\vspace{0.2cm}

\noindent - \textbf{Condition (H)}: Almost surely, $\mathcal{Y}_n$ contains all paths obtained by hopping among paths in $W_n^l \cup W_n^r$ at crossing times.

Based on the wedge characterization of the Brownian Net, the last two conditions impose a limitation on the number of paths of any subsequential weak limit of $(\mathcal{Y}_n)_{n \in \mathbb{N}}$. These will be called conditions (U$_{\mathcal{N}}^{'}$) and (U$_{\mathcal{N}}^{''}$). 

\vspace{0.2cm}

\noindent - \textbf{Condition (U$_{\mathcal{N}}^{'}$)}: There exist $\widehat{W}_n^l, \widehat{W}_n^r \in \widehat{\mathcal{H}}$, whose starting points become dense in $\mathbb{R}^2$ as $n \rightarrow \infty$, such that a.s. paths in $W_n^l$ and $\widehat{W}_n^l$ (resp. paths in $W_n^r$ and $\widehat{W}_n^r$) do not cross.

\vspace{0.2cm}

\noindent - \textbf{Condition (U$_{\mathcal{N}}^{''}$)}: For any weak limit point $(\mathcal{Y},W^l,W^r,\widehat{W}^l,\widehat{W}^r)$ of $(\mathcal{Y}_n,W_n^l,W_n^r,\widehat{W}_n^l,\widehat{W}_n^r)$ and for any deterministic countable dense set $D \subset \mathbb{R}^2$, a.s. paths in $\mathcal{Y}$ do not enter any wedge of $(\widehat{W}^l(D),\widehat{W}^r(D))$ from outside.

\vspace{0.2cm}

In \cite{survey1} (U$_{\mathcal{N}}^{'}$) and (U$_{\mathcal{N}}^{''}$) are merged into a single condition denoted by (U$_{\mathcal{N}}$). Here we will be able to verify (U$_{\mathcal{N}}^{'}$) but not  (U$_{\mathcal{N}}^{''}$) for the Drainage Network with Branching. This is related to the conjecture mentioned in Section 1, and we will discuss it later. We point out that to verify  (U$_{\mathcal{N}}^{''}$), it is enough to show that paths in $\mathcal{Y}_n$ do not enter wedges of $(\widehat{W}_n^l,\widehat{W}_n^r)$ from outside, and when a pair of paths in $(\widehat{W}_n^l,\widehat{W}_n^r)$ converges to a pair of dual left-right coalescing Brownian motions, the associated first meeting time of the pair in $(\widehat{W}_n^l,\widehat{W}_n^r)$ also converges, which implies that the associated wedge converges.

Now we can state the convergence result:

\begin{proposition}
\label{Result:convergencecritter}
\textbf{( \cite[Theorem $6.11$]{survey1})} Let $(\mathcal{Y}_n)_{n \in \mathbb{N}}$ be a sequence of $(\mathcal{H},\mathcal{B}_{\mathcal{H}})$-valued random variables that satisfy conditions (C),(I$_{\mathcal{N}}$), (H), (U$_{\mathcal{N}}^{'}$) and (U$_{\mathcal{N}}^{''}$) described above. Then $(\mathcal{Y}_n)$ converges in distribution to the Brownian Net $\mathcal{N}_{\lambda,b}$.
\end{proposition}

To conclude this section, we state a weaker result that we will need.

\begin{proposition}
\label{Result:convergencecritter2}
Let $(\mathcal{Y}_n)_{n \in \mathbb{N}}$ be a sequence of $(\mathcal{H},\mathcal{B}_{\mathcal{H}})$-valued random variables which satisfy conditions (C),(I$_{\mathcal{N}}$), (H), and (U$_{\mathcal{N}}^{'}$) described above. Then $(\mathcal{Y}_n)$ is tight and any subsequential weak limit contains a copy of the Brownian Net $\mathcal{N}_{\lambda,b}$.
\end{proposition}

\vspace{0.2cm}

The proof of Proposition \ref{Result:convergencecritter2} follows directly from the proof of Proposition \ref{Result:convergencecritter} as in \cite{survey1}. Indeed condition (U$_{\mathcal{N}}^{''}$) is the condition that guarantees that any weak limit point of $(\mathcal{Y}_n)_{n \in \mathbb{N}}$ does not have more paths than $\mathcal{N}_{\lambda,b}$ and so it must be identically distributed to $\mathcal{N}_{\lambda,b}$.

\vspace{0.2cm}


\section{Drainage Network with branching}
\label{sec:DNB}

\vspace{0.2cm}

\subsection{Model description and main results} \label{model}\

\vspace{0.2cm}

We will begin with an informal description of the model: Consider the lattice $\mathbb{Z}^2$ where the first coordinate represents space and the second one time. Suppose that each vertex can be independently either open with probability $p$ or closed with probability $1-p$, for some fixed $p\in(0,1)$. Now we consider a random graph with vertex set $\mathbb{Z}^2$ such that from each vertex $(x,t)$ departs at most two directed edges according to the following rule: when there is a unique $y$ such that $(y,t+1)$ is the closest open vertex to $(x,t)$ at time $t+1$, then $(x,t) \mapsto (y,t+1)$ is the unique edge departing from $(x,t)$. In case of non-uniqueness of the nearest open vertex, we have two open vertices $(y,t)$ and $(y',t)$, where $y-x=x-y'$, then $(x,t) \mapsto (y,t+1)$ and $(x,t) \mapsto (y',t+1)$ are the edges departing from $(x,t)$. In this way, the collection of directed paths in the graph forms a system of coalescing-branching paths. We call this collection, or equivalently the random graph, the full \emph{Drainage Network with Branching (DNB).}

The full DNB described above is one particular case of a DNB, but the latter can have fewer paths. Specifically, we fix a probability $\epsilon \in [0,1]$ and we change the graph by removing edges in the following way: independently for each point $(x,t)$ which has two departing edges, keep both edges on the graph with probability $\epsilon$. Otherwise, choose uniformly only one of them to be removed. For $\epsilon = 1$, we have the full DNB. The DNB extends the usual Drainage Network, which is the case $\epsilon = 0$, by allowing the paths to branch (see Figure \ref{drainage} for an example). As we mentioned in the introduction, the Drainage Network was introduced in \cite{grs} to represent a drainage system where we have sources of a liquid that flows in a specific direction through empty spaces. It is motivated by studies about drainage patterns into an intramontane trench. There are even verifications of empirical predictions about river network models, see for instance \cite{dual}.  However, we do not know about a similar approach related to the Brownian Net, which is connected with other important theoretical systems, such as the Howitt-Warren flows, see \cite{sss}. From these considerations, we believe that the DNB is a natural model related to the recent convergence studies associated with the universality class of the Brownian Web and Net. 

\begin{figure}[h!]
\centering
\includegraphics[scale=0.40]{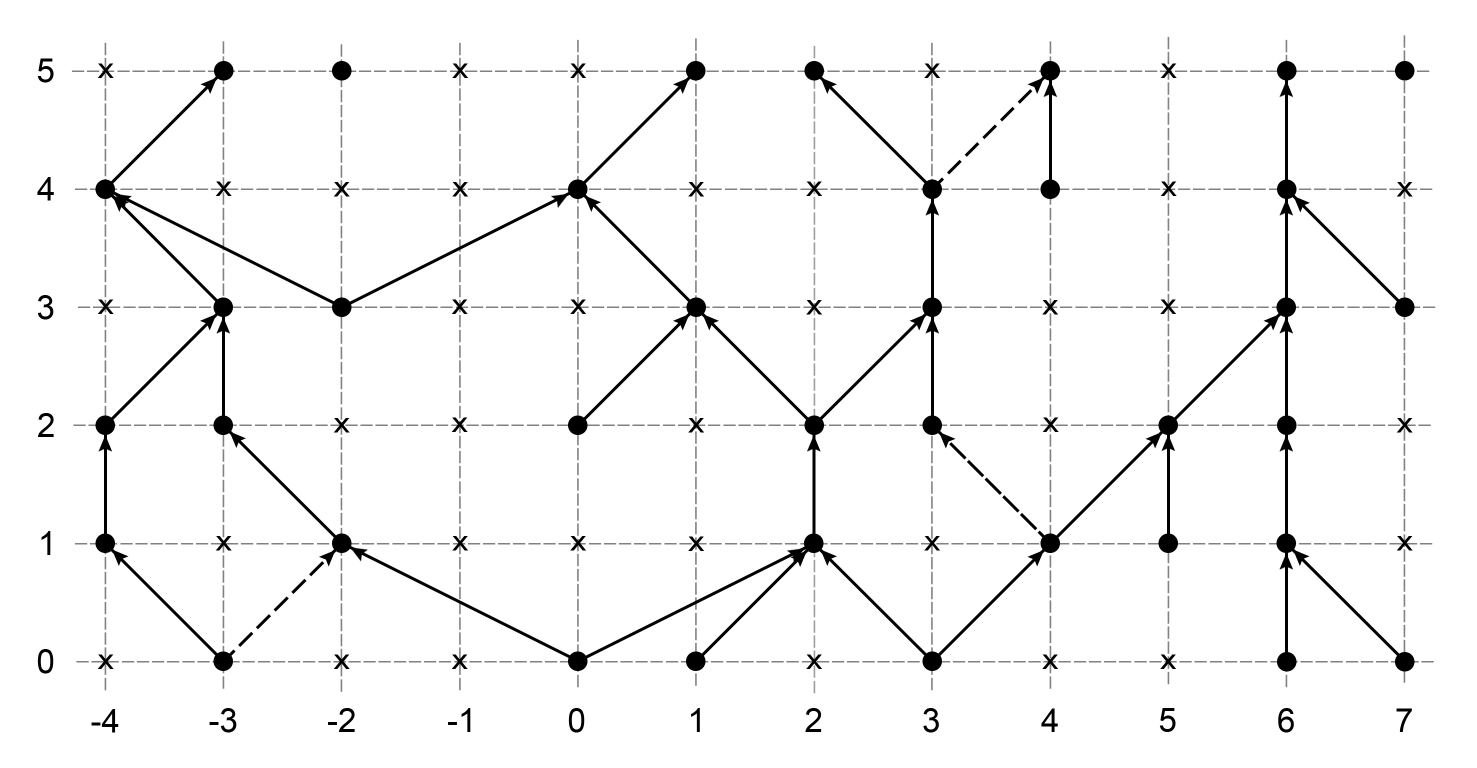}
\caption{Graph representation of a realization of the DNB. Black dots represent open vertices and {\small $\times$} represents closed vertices. Black lines represent paths of the DNB and the $3$ dashed lines represent edges that have been removed from the full DNB.}
\label{drainage}
\end{figure}

In Figure \ref{drainage}, we have the graph illustration for an example of DNB considering parameters $p=0.5$ and $\epsilon = 0.5$. Note that the vertex $(0,0)$ has two edges departing from it, so a branch occurs at this vertex for this realization of the system. Vertex $(4,1)$ also has two edges departing in the full DNB but only one in the DNB.

Now we present a formal description of the DNB with parameters $p$ and $\epsilon$. From now on, we suppose that $p \in (0,1)$ is fixed for the entire paper.

Let $(\omega(z))_{z\in\mathbb{Z}^2}$ 
be a family of IID Bernoulli random variables with parameter $p$ and $(\theta(z))_{z\in\mathbb{Z}^2}$ be a family of IID random variables on $\{-1,0,1\}$  that depends on a parameter $\epsilon \in (0,1)$ and have the following probability function:
$$
P(\theta(z) = 0) = \epsilon, \ \ \ \ \ \ \ \ \ \ P(\theta(z) = -1) = P(\theta(z) = 1) = \frac{1-\epsilon}{2}.
$$
We consider that these two families are independent of each other. Also 
as a reference for future use, we introduced the notation $\mathcal{F}_t = \sigma\{(\omega(z),\theta(z)), z=(z_1,z_2), z_1 \in \mathbb{Z}, z_2 \leq t\}$, $t \in \mathbb{Z}$, which represents the filtration generated by the environment $(\omega,\theta)$.

Every vertex $z \in \mathbb{Z}^2$ will be classified as \emph{open} or \emph{closed} if $\omega(z) = 1$ or $\omega(z) = 0$ respectively. Denote by $V$ the set of open vertices: $V=\{ z \in \mathbb{Z}^2: \omega(z)=1\}$. Moreover, we say that a vertex $z = (z_1,z_2)$ is above $\tilde{z} = (\tilde{z}_1,\tilde{z}_2)$ if $z_2 > \tilde{z}_2$, \emph{immediately above} if $z_2 = \tilde{z}_2 + 1$, \emph{to the right of} $\tilde{z}$ if $z_1 > \tilde{z}_1$ and \emph{to the left} if $z_1 < \tilde{z}_1$.

For all $z \in V$ we consider $h(z,i)$, $i \in \{-1,1\}$ as the nearest open vertex which is immediately above $z$ when it is uniquely defined, otherwise $h(z,-1)$ is the nearest open vertex immediately above and to the left of $z$; and $h(z,1)$ is the nearest open vertex immediately above and to the right of $z$. Using the function $h$ we define $\Gamma^l(z)$ and $\Gamma^r(z)$ as follows: 
$$ \Gamma^l(z) = \begin{cases} \displaystyle{h(z,-1), \mbox{ if } \theta(z) \in \{-1,0\}},  \\ h(z,1), \mbox{ if } \theta(z) = 1; \end{cases}
\quad \textrm{and} \quad 
\Gamma^r(z) = \begin{cases} \displaystyle{h(z,-1), \mbox{ if } \theta(z) = -1},  \\ h(z,1), \mbox{ if } \theta(z) \in \{0,1\}. \end{cases}$$ 
Note that $\Gamma^l(z) \neq \Gamma^r(z)$ only when we have a branching on vertex $z$, i.e., $\theta(z)=0$. 

Now for $z \in V$ and $\tilde{a} = (a_j)_{j \geq 1}$, $a_j \in \{l,r \}$ $\forall \, j\ge 1$, we can associate a continuous random path $(X_t^{z,\tilde{a}})_{t\geq0}$, by making $X_0^{z,\tilde{a}} = z$ and $X_n^{z,\tilde{a}} = \Gamma^{a_n} \circ \Gamma^{a_{n-1}} \circ \ldots \Gamma^{a_1}(z)$, and connecting consecutive vertices $\{X_{j-1}^{z,\tilde{a}},X_{j}^{z,\tilde{a}}\}$, $j\in \mathbb{N}$, through linear interpolation. 

Let $\mathcal{G} = (V,E)$ be the random directed graph with set of vertices $V$ and edges $E=\{\{z,\Gamma^l(z)\},\{z,\Gamma^r(z)\}:z \in V \}$. Also denote by $\mathcal{A} = \{ \tilde{a} = (a_j)_{j \geq 1} : a_j \in \{l,r \} \}$ and put $\mathcal{X}_\epsilon = \{(X^{z,\tilde{a}}_s,s)_{s\geq z_2}:z\in V, \tilde{a} \in \mathcal{A} \}$. The random set of paths $\mathcal{X}_\epsilon$ is called the \emph{Drainage Network with branching parameter $\epsilon$} and  $\mathcal{G} = (V,E)$ is its associated directed graph. Note that two paths in $\mathcal{X}_\epsilon$ will coalesce when they meet each other in the sense that they coincide from that time on.

Fix a sequence $\{\epsilon_n\}_{n\geq1}$ in $(0,1)$ and let $\mathcal{X}_{\epsilon_n}^n = \{ (\frac{z_1}{n},\frac{z_2}{n^2}) \in \mathbb{R}^2: (z_1,z_2) \in \mathcal{X}_{\epsilon_n} \}$, for $n \ge 1$, be the diffusively rescaled Drainage Network with branching parameter $\epsilon_n$. When the explicit value of $\epsilon_n$ is not relevant we will usually omit the index of $\mathcal{X}^n_{\epsilon_n}$. 

\begin{remark} \label{rem:moments}
Note that paths in the DNB have identically distributed increments with finite moments of any order. Indeed the absolute value of the increments is stochastically dominated by a geometric distribution with parameter $p$, denoted here $Geom(p)$. The verification of the last assertion is simple. Once at site $(x,t)$, the size of the next increment is the minimum between the distance at time $t+1$ between x and the nearest open site at its right-hand side  and the nearest open site at its left-hand side. These distances are IID $Geom(p)$, and we can take any of them as an upper bound for the increment size.  
\end{remark}

\begin{remark}
\label{remark:branchprob}
Note that the probability of branching in an open vertex of the DNB depends on the parameter $p$, since the branching can occur only if we have a tie in the nearest open vertices at the next time level. It is straightforward to see that the probability of having this tie for an open vertex is $\frac{p(1-p)}{2-p}$, and then, the probability of having a branching occurring in a vertex is $\frac{p(1-p)}{2-p}\epsilon$.
\end{remark}

Before we state our main results, we need to define some special subsets of paths in $\mathcal{X}^n$: An l-path in $\mathcal{X}^n$ is a path obtained by always choosing the leftmost edge when branching occurs along the path in the DNB. Analogously an r-path in $\mathcal{X}^n$ is a path obtained by always choosing the rightmost edges. Let $W^l_n \subset \mathcal{X}^n$ be its subset of l-paths  and $W^r_n \subset \mathcal{X}^n$ be the subset of r-paths. By this definition, we have that no path $\pi \in \mathcal{X}^n$ crosses any path $l \in W^l_n$ from right to left, i.e., $\pi(s) > l(s)$ and $\pi(t) < l(t)$ for some $s < t$, and no path $\pi \in \mathcal{X}^n$ crosses any path $r \in W^r_n$ from left to right. For example, in Figure \ref{drainage}, the l-path that starts on vertex $(2,1)$ is the one that passes by vertices $(2,2),(1,3),(0,4),(1,5)$ and the r-path that starts from this same vertex is the one that passes by vertices $(2,2),(3,3),(3,4),(2,5)$.

The main objective of our study is to describe the asymptotic behavior of the diffusively rescaled DNB when $\epsilon_n = \fb n^{-\alpha}$ for different values of $\alpha > 0$, where $\mathfrak{b}$ is a positive constant. In more generality, we may consider $(\epsilon_n)_{n\ge 1}$ such that $\lim_{n\rightarrow \infty} \epsilon_n n^\alpha = \mathfrak{b} > 0$, but to simplify our exposition, we will do the proofs assuming $\epsilon_n = \fb n^{-\alpha}$, since we consider that this generalization is straightforward. In this paper we deal solely with the case $\alpha = 1$ and in a further paper we shall consider the case $1< \alpha \le 2$. These cases are clearly distinct, for $\alpha = 1$ the system is expected to converge to the Brownian Net, for $1< \alpha \le 2$ the system is expected to converge to the Brownian Web.

\begin{remark}
If we have $\epsilon_n = \fb n^{-\alpha}$ for $\fb > 0$  and $\alpha \in [0,1)$, then the process $\mathcal{X}^n_{\epsilon_n}$ does not converge in distribution under diffusive scaling. In this case, $n \, \epsilon_n$ diverges to infinity when $n \rightarrow \infty$, which means that the lattice $\mathbb{Z}^2$ is compressed faster than the reduction of the branching probability. So, if $\mathcal{X}^n_{\epsilon_n}$ had a weak limit, it would be concentrated on paths with drift $\infty$ or $-\infty$ which is not possible.   
\end{remark}

\begin{remark}
\label{remark:sigmap}
The Brownian Web and Brownian Net that appear as limit processes in all theorems of this section will have a diffusion coefficient equal to $\lambda^2 = \lambda_p^2$, which is the variance of an increment of a DNB path. This variance is finite since the increments are stochastically bounded by $Geom(p)$. If one wishes to have these results for the standard Brownian Web and Net with diffusion coefficient equal to $1$, it can be achieved by dividing the space coordinate by $\lambda$ in the definition of $\mathcal{X}_\epsilon^n$. 
\end{remark}

\begin{remark}
The case of $\alpha = \infty$ (which means $\epsilon_n$ = 0 for all $n \in \mathbb{N})$ is equivalent to the usual Drainage Network where no branching occurs. For this process, the convergence to the Brownian Web has been established in \cite{cfd}.
\end{remark}

From now on we fix $\alpha =1$. The following two theorems are devoted to the study of the limit behavior of the DNB. We use the standard notation ``$\Longrightarrow$" for convergence in distribution.

\begin{theorem}
\label{Teo:leftright}
Let $W_n^l$ (resp. $W_n^r$) be the set of l-paths (resp. r-paths) of a diffusively rescaled DNB with parameter $\epsilon_n = \fb n^{-1}$ for some $\fb >0$. There exist $\hat{W}_n^l$ and $\hat{W}_n^r$ dual processes of $W_n^l$ and $W_n^r$ respectively such that 
$$
(W_n^l,  W_n^r, \hat{W}_n^l, \hat{W}_n^r) \Longrightarrow (\mathcal{W}_{\lambda,b_p}^l, \mathcal{W}_{\lambda,b_p}^r, \widehat{\mathcal{W}}_{\lambda,b_p}^l, \widehat{\mathcal{W}}_{\lambda,b_p}^r) \ \mbox{in } \mathcal{H}
\ \ \textrm{as } \ n \rightarrow \infty,
$$
where $b_p = \frac{\fb (1-p)}{(2-p)^2}$.
\end{theorem}

\begin{theorem}
\label{Teo:teoprincipal}
If we have $\mathcal{X}^n_{\epsilon_n}$ the diffusively rescaled Drainage Network with branching parameter $\epsilon_n = \fb n^{-1}$ for $\fb >0$, then $\mathcal{X}^n_{\epsilon_n}$ is tight and any subsequential limit contains a copy of the Brownian Net $\mathcal{N}_{\lambda,b_p}$ with $b_p = \frac{\fb (1-p)}{(2-p)^2}$.
\end{theorem}

The statement of Theorem \ref{Teo:teoprincipal} means that if $\mathcal{X}$ is a limit point of $\mathcal{X}^n_{\epsilon_n}$ then there exists a subcollection of paths $\Hat{\mathcal{X}}\subset \mathcal{X}$ such that $\Hat{\mathcal{X}}$ is distributed as $\mathcal{N}_{\lambda,b_p}$.

Following the convergence criteria in Subsection \ref{subsec:browniannetcritter} we will show that conditions (C), (I$_{\mathcal{N}}$), (H) and (U$_{\mathcal{N}}^{'}$) holds for the DNB when $\epsilon_n = \fb n^{-1}$ to prove Theorem \ref{Teo:teoprincipal}. We conjecture that $\mathcal{X}^n_{\epsilon_n}$ converges to $\mathcal{N}_{\lambda,b_p}$. To prove this conjecture we would need to show that condition (U$_{\mathcal{N}}^{''}$) holds, but we were not able to prove it. The Spatial $\Lambda$-Fleming-Viot Process dealt with by Etheridge, Freeman and Straulino in \cite{efs} does not exhibit long range dependence, which imposes a technical difficulty to deal with the condition (U$_{\mathcal{N}}^{''}$) in our case. In Appendix \ref{subsec:conditionUdiscuss}, we talk about this difficulty and which details remain pending to show the convergence to Brownian Net. 

Theorems \ref{Teo:leftright} and \ref{Teo:teoprincipal} are proved in Section \ref{sec:lefttobrownian}, although the main condition to prove Theorem \ref{Teo:teoprincipal} is developed in Section \ref{sec:conditionI}.

In the following subsection, we will define the dual process for the DNB, which  is mentioned in Theorem \ref{Teo:leftright} and have a crucial role in our convergence study.

\vspace{0.2cm}

\subsection{Dual process for Drainage Network with branching}\
\label{subsec:dualdef}

\vspace{0.2cm}

For the process $\mathcal{X}^n$ we can describe a dual process that we denote by $\hat{\mathcal{X}}^n$, with paths that follow the reverse direction in time and satisfy the following restrictions:
\begin{itemize}
	\item The configuration of $\hat{\mathcal{X}}^n$ is unique and determined by the realization of the process $\mathcal{X}^n$;
	\item $\hat{\mathcal{X}}^n$ also has subsets of dual l-paths and dual r-paths that we denote by $\hat{W}^l_n$ and $\hat{W}^r_n$ respectively. Dual l-paths in $\hat{W}^l_n$ cannot cross l-paths in ${W}^l_n$ and dual r-paths in $\hat{W}^r_n$ cannot cross r-paths in ${W}^r_n$.
\end{itemize}

First we will construct $\hat{\mathcal{X}}_{\epsilon}$, the dual process of $\mathcal{X}_{\epsilon}$ and after that, we define $\hat{\mathcal{X}}^n$, the dual of the diffusively rescaled Drainage Network $\mathcal{X}^n$. To construct $\hat{\mathcal{X}}_{\epsilon}$ we will adapt some ideas presented in \cite{dual} for the case without branching. 

Recall from Section \ref{model} that  $\mathcal{G} = (V,E)$ is the random directed graph associated with $\mathcal{X}_{\epsilon}$. The vertices of dual processes will be given by the mid-points between two consecutive vertices in $V$ on each line of time. We denote this random set of vertices by $\hat{V}$. For a more formal description, given $z = (z_1,z_2)\in \mathbb{Z}^2$ let us define:
$$
	K^r(z) = \inf\{k\geq1: (z_1+k,z_2) \in V\},
$$
$$K^l(z) = \inf\{k\geq1: (z_1-k,z_2) \in V\},$$
$$r(z) = (z_1+K^r(z),z_2) \quad \textrm{and} \quad
l(z) = (z_1-K^l(z),z_2).$$

Notice that $r(z)$ and $l(z)$ denote the nearest open vertex respectively at the right-hand side of $z$ and at the left-hand side of $z$. For all $z \in V$, let $\displaystyle \hat{r}(z) = \left(z_1 + \frac{K^r(z)}{2},z_2 \right)$ denote the dual neighbor to the right of $z$ and $\displaystyle \hat{l}(z) = \left(z_1 - \frac{K^l(z)}{2},z_2 \right)$ denote the dual neighbor to the left of $z$. Now we can define the set of vertices of the dual process as:
$$
\hat{V} = \{\hat{r}(z):z \in V\}.
$$    
It is not hard to build an artificial example where $\Hat{V}_1 = \hat{V}_2$ and $V_1 \neq V_2$, but in fact $\hat{V}$ will uniquely determine $V$ with probability one. It happens because whenever we find a point $(x,t) \in \mathbb{Z}^2$ such that $(x-\frac{1}{2},t)$ and $(x+\frac{1}{2},t)$ belong to $\hat{V}$, then every vertex of $V$ at the time level $t$ becomes uniquely determined. Moreover $(x-\frac{1}{2},t)$ and $(x+\frac{1}{2},t)$ belong to $\hat{V}$, if and only if, $(x-1,t)$, $(x,t)$ and $(x+1,t)$ belong to $V$. Thus from basic properties of IID Bernoulli distribution for the open vertices, with probability one for each time level $t$ there exists $x=x(t)$ with $(x,t)$ as above.

\begin{remark}
	Note that the right dual neighbor of a vertex $z \in V$ is equal to the left dual neighbor of the vertex $r(z)$ that is also in $V$, so we can define $\hat{V}$ considering $\hat{l}(z)$ for all $z \in V$ instead of $\hat{r}(z)$ for all $z \in V$.
\end{remark}

\begin{remark}
	Since $K^r(z)$ and $K^l(z)$ can be odd numbers, the first coordinate of dual vertices can be non-integer, so $\hat{V}$ is a subset of the lattice $\frac{1}{2}\mathbb{Z} \times \mathbb{Z}$.
\end{remark}

Now, let us construct the edge set $\hat{E}$ of the dual graph. From each vertex $\hat{z} \in \hat{V}$ departs either one or two edges connecting it to a vertex immediately below (since this process evolves backward in time). If there exists a vertex $\hat{z}' \in \hat{V}$ at time level $\hat{z}_2-1$ such that the line connecting the points $\hat{z}$ and $\hat{z}' $  does not cross any edge in $E$ and $\hat{z}'$ is the nearest vertex in $\hat{V}$ with this property, then $(\hat{z},\hat{z}')$ is the unique edge in $\hat{E}$ departing from $\hat{z}$. But due to branching in the DNB, some vertices in $\hat{V}$ are not able to reach any other vertex immediately below without crossing any path of $\mathcal{X}_\epsilon$, or equivalently an edge of $E$. In these cases, the path of the dual process also branches. Because of this, we obtain a pair of edges departing from $\hat{z}$, one connecting it to the closest vertex to the right immediately below and another connecting it to the closest one to the left immediately below. When such branching occurs, a path from the dual crosses a path from the DNB, but we do not have crossings between l-paths or between r-paths. Moreover, with our definition, branchings of the dual are in one-to-one correspondence with branchings of the DNB. 

For a formal description of $\hat{E}$, define for all $\hat{z} \in \hat{V}$:
$$
a^l(\hat{z}) = \sup\{k \in \mathbb{Z}:(k,\hat{z}_2-1)\in V, \Gamma_1^l((k,\hat{z}_2-1)) < \hat{z}_1 \},
$$
$$
a^r(\hat{z}) = \inf\{k \in \mathbb{Z}:(k,\hat{z}_2-1)\in V, \Gamma_1^r((k,\hat{z}_2-1)) > \hat{z}_1 \}.
$$
If $a^l(\hat{z}) \neq a^r(\hat{z})$, then we set 
$$
\hat{\Gamma}^l(\hat{z}) = \hat{\Gamma}^r(\hat{z}) = \left(\frac{a^l(\hat{z})+a^r(\hat{z})}{2},\hat{z}_2-1 \right) \in \hat{V}.
$$ 
In this case we have a unique edge departing from $\hat{z}$ which is  $(\hat{z},\hat{\Gamma}^r(\hat{z})) \in \hat{E}$. This edge does not cross any edge of $\mathcal{G}$. Note that, by definition, $(a^l(\hat{z}),z_2-1)$, $(a^r(\hat{z}),\hat{z}_2-1)$ are the nearest vertices in $V$ respectively at the left-hand side and right-hand side of the dual vertex $\hat{\Gamma}^r(\hat{z})$.

If $a^l(\hat{z}) = a^r(\hat{z}) = a$, then for every vertex $\hat{z}' = (\hat{x},\hat{z}_2-1)$ in $\hat{V}$, $\hat{x} \in \frac{1}{2} \mathbb{Z}$, the edge $(\hat{z},\hat{z}')$ crosses either  $\big((a,\hat{z}_2-1),\Gamma^l(a,\hat{z}_2-1)\Big) \in E$ or $\big((a,\hat{z}_2-1),\Gamma^r(a,\hat{z}_2-1)\Big) \in E$. In this case we set $\hat{\Gamma}^l(\hat{z}) = \left(\hat{r}(a,\hat{z}_2-1),\hat{z}_2-1 \right) \in \hat{V}$ and $\hat{\Gamma}^r(\hat{z}) = (\hat{l}(a,\hat{z}_2-1),\hat{z}_2-1) \in \hat{V}$. Note that $\hat{\Gamma}^l(\hat{z})$ and $\hat{\Gamma}^r(\hat{z})$ are the nearest vertices in $\hat{V}$ to the right-hand side and left-hand side, respectively, of the vertex $(a,\hat{z}_2-1)$. It means that branching occurs in the dual process at vertex $\hat{z}$. Also note that since the dual paths flow in the reverse direction of time, the sense of left and right is the reverse of the DNB paths. For example, if we look at $\hat{z} = (2,3)$ in Figure \ref{dualfig}, we have that $a^l(\hat{z}) = a^r(\hat{z}) = 2$, $\hat{\Gamma}^l(\hat{z}) = (\frac{5}{2},2)$ and $\hat{\Gamma}^r(\hat{z}) = (1,2)$. Finally, the edge set of the dual graph $\hat{\mathcal{G}} = (\hat{V},\hat{E})$ is given by $\hat{E} = \hat{E}^l \cup \hat{E}^r$ where:
$$
\hat{E}^l = \{\langle \hat{z},\hat{\Gamma}^l(\hat{z})\rangle: \hat{z} \in \hat{V} \}
\quad \textrm{and} \quad 
\hat{E}^r = \{\langle \hat{z},\hat{\Gamma}^r(\hat{z})\rangle: \hat{z} \in \hat{V} \}.
$$

Now we can define a backward in time-continuous random path $(\hat{X}_t^{\hat{z},\tilde{d}})_{t\geq0}$ starting from vertex $\hat{z} \in \hat{V}$ by making $\hat{X}_n^{\hat{z},\tilde{d}} = \hat{\Gamma}^{d_n} \circ \hat{\Gamma}^{d_{n-1}} \circ \ldots \hat{\Gamma}^{d_1}(\hat{z})$ with $\tilde{d} = d_1, \ldots d_n \in \{l,r\}$ for $n \in \mathbb{N}$ and doing linear interpolation between all pairs of consecutive vertex $\{\hat{X}_{j+1}^{\hat{z},\tilde{d}},\hat{X}_j^{z,\tilde{d}}\}$ for $j=0,1\ldots n-1$.

In Figure \ref{dualfig}, we have the graph illustration of the dual process related to the DNB example presented in Figure \ref{drainage}. Note that between two open vertices of the DNB, we always have one vertex of the dual just in the middle and the direction of the edges (arrows), indicating that the dual flows in the reverse direction of time. Also note that a dual path branches at one vertex if, and only if, in front of it, we have an open vertex of DNB where a branching also occurs. So branchings for both systems are in one-to-one correspondence. 
\begin{figure}[h!]
	\centering
	\includegraphics[scale=0.40]{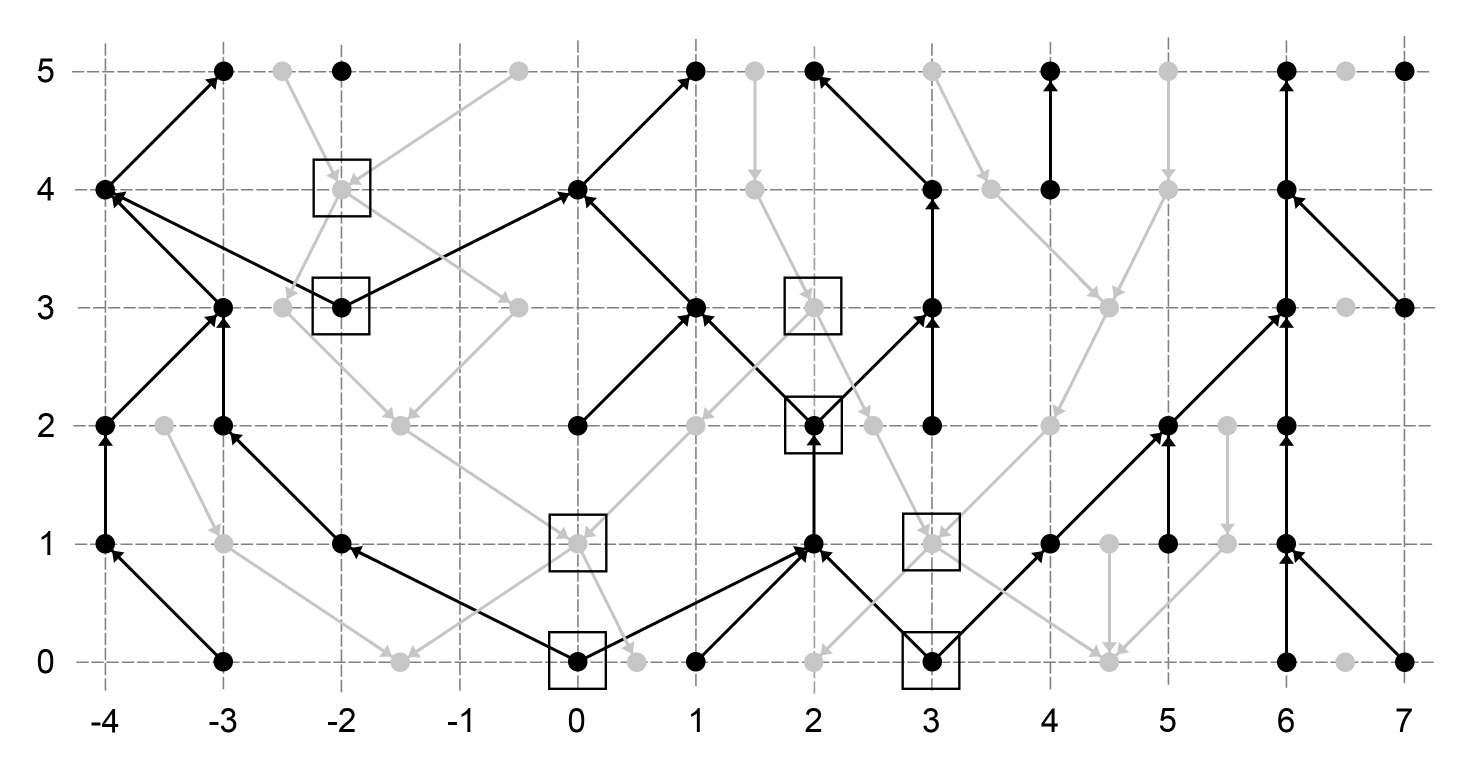}
	\caption{Graph representation of a realization of DNB and its dual. The black points and black arrows are related to the DNB while the gray points and gray arrows are related to the dual process. The squares mark the vertices where branching occurs.}
	\label{dualfig}
\end{figure}

\begin{remark} 
Every path of the dual process also evolves as a Markov process (in the reverse direction of time). It holds because the transition probabilities of a dual path from time $t$ to $t-1$ depend only on the position of this path at time $t$. More specifically, the behavior of these transition probabilities depends on whether the current position of the dual path is integer or not. In the proof of Proposition \ref{proposition:markov}, these transition probabilities are calculated explicitly.
\end{remark}

Now let 
$$
\hat{\mathcal{X}}^n = \hat{\mathcal{X}}_{\epsilon_n}^n = \{ (\frac{z_1}{n},\frac{z_2}{n^2}) \in \mathbb{R}^2: (z_1,z_2) \in \hat{\mathcal{X}}_{\epsilon_n} \},
$$ 
for $n \in \mathbb{N}$ be the dual of the diffusively rescaled Drainage Network with branching parameter $\epsilon_n$.
We denote by $\hat{W}^l_n \subset \hat{\mathcal{X}}^n$ the subset of the dual l-paths, which contains all the paths obtained by always choosing the left option whenever a branch in the dual happens. Similarly, we denote by $\hat{W}^r_n \subset \hat{\mathcal{X}}^n$ the subset of the dual r-paths, which contains all paths obtained by always choose the right option whenever a branch in the dual occurs. Note that we do not have crossings between paths in $W^l_n$ and paths in $\hat{W}^l_n$, nor between paths in $W^r_n$ with paths in $\hat{W}^r_n$. 

\vspace{0.2cm}


\section{Estimates for coalescence times}
\label{sec:coalescence}

\vspace{-0.2cm}

This section is devoted to obtaining an upper bound for the probability that either two l-paths or two r-paths do not coalesce until time $t$. This bound is fundamental to prove all theorems presented in the previous section.

Consider two l-paths $X^u_t, X^v_t \in W^l$ starting from points $u=(u_1,u_2)$ and $v=(v_1,v_2)$; and define the process $(Z^{u,v}_t)_{t \geq (u_2 \vee v_2)}$ as being the distance between the position of these two paths at time $t$, i.e.,
$$
Z^{u,v}_t = | X^u_t - X^v_t |, \ t \geq (u_2 \vee v_2).
$$
Note that due to the translation invariance of the DNB, the distribution of $Z^{u,v}_t$ depends on $u$ and $v$ only through the initial distance between the paths at time $(u_2 \vee v_2)$: $k = | X^u_{(u_2 \vee v_2)} - X^v_{(u_2 \vee v_2)} |$. Then let us consider both paths starting at time zero, i.e., $u_2 = v_2 = 0$, and with initial distance $k$ between them. We write $Z^k_t$ to simplify the notation.

Let us denote by $\tau_k$ the coalescing time for two paths starting at distance k from each other, i.e.
$$
\tau_k = \min \{ t \geq 0: Z^k_t = 0 \}.
$$

\begin{remark}
\label{remark:rightpaths}
Note that we can define an analogous process $Z^k_t$ and a coalescence time $\tau_k$ for a pair of r-paths at initial distance $k$ (we use the same notation). They have the same distribution as the respective variables when defined for l-paths at initial distance $k$.
\end{remark}

Finally, we can present the main result of this section.

\begin{lemma}
\label{lemma:dnbcoalescencetime}
There exists a constant $C_0 > 0$ such that for every $t > 0$ and $k \in \mathbb{N}$ we have:
$$
P(\tau_k > t) \leq \frac{C_0 k}{\sqrt{t}},
$$
where this coalescence time can refer to a pair of either l-paths or r-paths. 
\end{lemma}

\vspace{0.2cm}

For the proof of Lemma \ref{lemma:dnbcoalescencetime} we will make use of the following result from \cite{cssc}:

\begin{proposition}
\label{Result:coalescencetime}
\textbf{(\cite[Theorem $5.2$]{cssc})} Let $\{\mathcal{G}_l:l\geq0\}$ be a filtration and $\{Y_l: l \geq 0\}$ be a $\mathcal{G}_l$-adapted discrete-time stochastic process taking values in $\mathbb{R}_{+}$. Let $\nu^Y = \inf\{l \geq 1: Y_l = 0 \}$ be the first time that the process $Y_l$ reaches zero. Suppose that:  
\begin{enumerate}
\item[(i)] For any $l \geq 0$, a.s. $E[(Y_{l+1} - Y_l)|\mathcal{G}_l] \leq 0$.
\item[(ii)] There exist constants $C_1,C_2 > 0$ such that for any $l \geq 0$, a.s. on the event $\{Y_l > 0\}$, we have
$$
E[(Y_{l+1}-Y_l)^2 \ |\mathcal{G}_l] \geq C_1 \ \mbox{ and } \ E[ \ |Y_{l+1}-Y_l|^3 \ |\mathcal{G}_l] \leq C_2.
$$
\end{enumerate}
Then $\nu^Y < \infty$ almost surely. Further, there exists a constant $C_3 > 0$ such that for any $y > 0$ and any integer $n$,
$$
P(\nu^Y > n | Y_0 = y) \leq \frac{C_3 y}{\sqrt{n}}.
$$
\end{proposition}

\vspace{0.2cm}

We will prove Lemma \ref{lemma:dnbcoalescencetime} using Proposition \ref{Result:coalescencetime}.

\medskip

\noindent \textit{Proof of Lemma \ref{lemma:dnbcoalescencetime}.} For this proof, we can consider that $\tau_k$ refers to coalescence between l-paths, since the proof considering r-paths would be identical. Also, recall that we assume that $u_2=v_2=0$. We also consider $u_1 > v_1$, so $u$ is at the left-hand side of $v$.

We have to verify that $Z^k_t$ and the coalescence time $\tau_k$ satisfy the conditions of Proposition \ref{Result:coalescencetime}. Note that $Z^k_t$ only takes values in $\mathbb{R}_{+}$, because two l-paths cannot cross each other. If the two conditions of Proposition \ref{Result:coalescencetime} hold, we will have proved Lemma \ref{lemma:dnbcoalescencetime}.

We have that $Z^k_t$ is a Markov chain with state space $\mathbb{Z}^{+} \cup \{0\}$. It has $0$ as the only absorbent state, since two paths coalesce when they meet for the first time. The processes $(X^u_{t})_{t \ge 0}$
and $(X^v_{t})_{t \ge 0}$ have IID symmetric increments. Besides that, the increments of $X^u_{t+1} - X^u_{t}$ and $X^v_{t+1} - X^v_{t}$ are also identically distributed (but not independent). Recalling that $\mathcal{F}_t$ is the filtration generated by the environment $(\omega,\theta)$, as defined in Section \ref{model}, we have:
$$
E[Z^k_{t+1}|\mathcal{F}_t] = E[ \ |X^u_{t+1} - X^v_{t+1}| \ | \mathcal{F}_t] =
E[ X^u_{t+1} - X^v_{t+1} | \mathcal{F}_t] = X^u_{t} - X^v_{t} = Z^k_t,
$$
which implies that the process $Z^k_t$ is a non-negative martingale. 
Then, for any $l \geq 0$, $E[(Z^k_{l+1} - Z^k_l) | \mathcal{F}_l] = 0$, so the condition (i) of Proposition \ref{Result:coalescencetime} is satisfied.

About the first inequality on condition (ii) of Proposition \ref{Result:coalescencetime}, note that the increments of $Z^k_t$ are not spatially homogeneous, but given $\mathcal{F}_l$ and that $Z^k_l > 0$, it is always possible to assure that $(Z^k_{l+1}-Z^k_l)^2 \geq 1$ with a convenient choice of open and closed vertices at time $l+1$. To achieve that, it is enough to have an open vertex at position $(X^u_{l},l+1)$ and a closed vertex at position $(X^v_{l},l+1)$. So, by Markov property of $Z^k_t$, we have that $E[(Z^k_{l+1}-Z^k_l)^2 \ |\mathcal{F}_l] \geq p(1-p) = C_1$ for any $l \geq 0$. 

The second inequality on condition (ii) of Proposition \ref{Result:coalescencetime} follows from the fact that both $(X^u_{t})_{t \ge 0}$ and $(X^v_{t})_{t \ge 0}$ have independent increments and additionally their increments have finite moments, see Remark \ref{rem:moments}: 

\vspace{-0.2cm}

\begin{align*}
    E[ \, |Z^k_{l+1} - Z^k_l|^3 \ |\mathcal{F}_l] &\leq 4 \big( E[ \, |X^u_{l+1} - X^u_{l}|^3 \ | \mathcal{F}_l] + E[ \ |X^v_{l+1} - X^v_{l}|^3 \ | \mathcal{F}_l] \big) \\
    &= 4\left( E[ \, |X^u_{l+1} - X^u_{l}|^3|] + E[ \, |X^v_{l+1} - X^v_{l}|^3|]\right) = C_2
\end{align*}
where $C_2$ depends only on $p$. \hfill $\square$

\medskip

The next result follows directly from Lemma \ref{lemma:dnbcoalescencetime}.

\begin{corollary}
\label{corollary:dnbcrossingtime}
Let $L_t^u \in W^l$ starting from $u=(u_1,u_2)$ and $R_t^v \in W^r$ starting from $v=(v_1,v_2)$ with $u_1 - v_1 = k > 0$ and $\tau^c_k = \min \{ t \geq 0: L_t^u \leq R_t^v \}$ the time until we have a crossing between $L_t^u$ and $R_t^v$. There exists a constant $C_0 > 0$ such that for every $t > 0$ and $k \in \mathbb{N}$ we have:
$$
P(\tau^c_k > t) \leq P(\tau_k > t) \leq \frac{C_0 k}{\sqrt{t}}.
$$
\end{corollary}

\begin{proof}
Let $R_t^u \in W^r$ be the r-path that starts from $u=(u_1,u_2)$. Since $L_t^u$ can never be at the right-hand side of $R_t^u$, it will be between $R_t^v$ and $R_t^u$, which implies that it has to cross $R_t^v$ before or at the same time as the coalescence between $R_t^v$ and $R_t^u$, which gives the first inequality. The last inequality follows from Lemma \ref{lemma:dnbcoalescencetime}.
\end{proof}

\begin{remark}
\label{remark:indepcoalesce}
Suppose that the processes $X^u$ and $X^v$ that appear in the definition of $Z^k$ are independent. More precisely, assume that the DNB paths $X^u$ and $X^v$ evolve according to distinct independent environments. Then both Lemma \ref{lemma:dnbcoalescencetime} and Corollary \ref{corollary:dnbcrossingtime} also hold in this case. To check this, we use a result from \cite[Lemma 2.2]{nrs} that obtains the same upper bound presented in Lemma \ref{lemma:dnbcoalescencetime} for the first meeting time between two independent nonsimple random walks. In \cite[Lemma 2.2]{nrs}, the increments of the random walks have zero mean and finite second moment. However it is straightforward to note from the proof that it is enough to have that the difference between the increments of the two random walks has zero mean instead of having  both paths with zero mean increments.
\end{remark}

Now we will state an equivalent result for coalescence times considering paths from the dual, which we prove in Subsection \ref{sec:dualcoalesce}. First let us give some notation for the dual case.

Recalling the notation of Section \ref{subsec:dualdef}, fix $\hat{z} \in \hat{V}$ and set recursively 
\begin{equation*}
\hat{\Gamma}_0^l(\hat{z}) = \hat{z}, \ \hat{\Gamma}_1^l(\hat{z}) = \hat{\Gamma}^l(\hat{z}) \ \textrm{and} \ 
	\hat{\Gamma}_{k+1}^l(\hat{z}) = \hat{\Gamma}^l(\Gamma_{k}^l(\hat{z}))\, , \text{ for } k\ge 1\, .
\end{equation*}

\noindent Let $\hat{Y}_k^l(\hat{z})$ denotes the first coordinate of $\hat{\Gamma}_{k}^l(\hat{z})$.
 
Also define $\hat{\Gamma}_{k}^r(\hat{z})$ and $\hat{Y}_k^r(\hat{z})$ analogously using the $\hat{\Gamma}^r$ instead of $\hat{\Gamma}^l$. So, $\hat{Y}^l_t(\hat{u})$ and $\hat{Y}^l_t(\hat{v})$ denote the position of the dual paths respectively at dual times $\hat{u}_2-t$ and $\hat{v}_2-t$ and starting from the dual vertices $\hat{u}$ and $\hat{v}$.
Define the process $(\hat{Z}^{\hat{u},\hat{v}}_t)_{t \geq 0}$ as being the distance between the position of these two dual paths after $t$ steps, starting to compute it at time $(\hat{u}_2 \wedge \hat{v}_2)$, i.e.
$$
\hat{Z}^{\hat{u},\hat{v}}_t = | \hat{Y}^l_{t+ \hat{u}_2 - (\hat{u}_2 \wedge \hat{v}_2)}(\hat{u}) - \hat{Y}^l_{t+ \hat{v}_2 - (\hat{u}_2 \wedge \hat{v}_2)}(\hat{v}) |, \ t \geq 0. 
$$
Let us denote by $\hat{\tau}_{\hat{u},\hat{v}}$ the time until coalescence between the two dual l-paths starting from $\hat{u}$ and $\hat{v}$, i.e.
$$
\hat{\tau}_{\hat{u},\hat{v}} = \min \{ t \geq 0: \hat{Z}^{\hat{u},\hat{v}}_t = 0 \}.
$$

Note that, unlike what happened with the DNB paths, in this case, the process $\hat{Z}^{\hat{u},\hat{v}}_t$ still depends on the specific position of the dual vertices $\hat{u}$ and $\hat{v}$ even if we know the initial distance between them. It occurs because the distribution of the increments of dual paths depends if they are in an integer position or not (see Figure \ref{dualfig}). Even if we know that the initial distance is a non-integer, the distribution of $\hat{Z}^{\hat{u},\hat{v}}_t$ also depends on which one of either $\hat{u}$ or $\hat{v}$ is a non-integer.

There is an argument that allows us to obtain the bound for the coalescence time considering only the case where $\hat{u}$ and $\hat{v}$ are integers, thus avoiding the problem of dealing with the distinct distributions for $\hat{Z}^{\hat{u},\hat{v}}_t$. To argue that, we will show an environment configuration that assures that  $\hat{Y}^l_{t+1}(\hat{u})$ and $\hat{Y}^l_{t+1}(\hat{v})$ are integers, given that at least one between $\hat{Y}^l_{t}(\hat{u})$ and $\hat{Y}^l_{t}(\hat{v})$ is a non-integer.

To simplify the notation, due to the time homogeneity of the dual process, we can consider that both dual paths start at time zero, i.e., $\hat{u}_2 = \hat{v}_2 = 0$. First, suppose that $\hat{Y}^l_{t}(\hat{u}) = y_u$ is an integer, $\hat{Y}^l_{t}(\hat{v}) = y_v$ is a non-integer and $y_u < y_v$ (besides that, by the construction rule of the dual paths, we need to have that $|y_v - y_u| \geq 1$). Consider the environmental event $(y_u,-t-1)$ is closed, $(y_u-1,-t-1)$ is open, $(y_u+1,-t-1)$ is open, $(y_v-0.5,-t-1)$ is open, $(y_v+0.5,-t-1)$ is closed and $(y_v+1.5,-t-1)$ is open. In this event, $((y_u,-t),(y_u,-t-1))$ and $((y_v,-t),(y_v+0.5,-t-1))$ are in $\hat{E}$, and  $\hat{Y}^l_{t+1}(\hat{u}) = y_u$, $\hat{Y}^l_{t+1}(\hat{u}) = y_v+0.5$. Both dual paths are at integer positions at time $t+1$. The event has probability either $p^4(1-p)^2$ if $y_u+1 < y_v-0.5$ or $p^3(1-p)^2$ if $y_u+1 = y_v-0.5$. 

Now suppose that both $\hat{Y}^l_{t}(\hat{u}) = y_u$ and $\hat{Y}^l_{t}(\hat{v}) = y_v$ are non integers and $y_u < y_v$. Consider the event $(y_u-1.5,-t-1)$ is open, $(y_u-0.5,-t-1)$ is closed, $(y_u+0.5,-t-1)$ is open, $(y_v-0.5,-t-1)$ is open, $(y_v+0.5,-t-1)$ is closed and $(y_v+1.5,-t-1)$ is open. In this event, $((y_u,-t),(y_u-0.5,-t-1))$ and $((y_v,-t),(y_v+0.5,-t-1))$ are in $\hat{E}$, and $\hat{Y}^l_{t+1}(\hat{u}) = y_u-0.5$, $\hat{Y}^l_{t+1}(\hat{u}) = y_v+0.5$, Both dual paths are at integer positions at time $t+1$. The event has probability either $p^4(1-p)^2$ if $y_u+0.5 < y_v-0.5$ or $p^3(1-p)^2$ if $y_u+0.5 = y_v-0.5$.

Because of the previous events, the probability of having both dual l-paths at integer positions at time $t+1$, given that at least one is at a non-integer position at time $t$, is bounded from below by $p^4(1-p)^2$. Denote

\vspace{-0.2cm}

$$
\hat{\nu}_{\hat{u},\hat{v}} = \inf \{ t\ge 0 : 
\hat{Y}^l_{t}(\hat{u}) \in \mathbb{Z} \textrm{ and } \hat{Y}^l_{t}(\hat{v}) \in \mathbb{Z} \}.
$$

Hence, by the Markov property of the dual l-paths (in the reverse direction of time), we have that $\hat{\nu}_{\hat{u},\hat{v}}$ is stochastically dominated by a random variable $\hat{\nu}$ that has geometric distribution with parameter $p^4(1-p)^2$. Then:  

\vspace{-0.2cm}

$$
P(\hat{\tau}_{\hat{u},\hat{v}} > t) = 
P((\hat{\tau}_{\hat{u},\hat{v}} - \hat{\nu}_{\hat{u},\hat{v}}) + \hat{\nu}_{\hat{u},\hat{v}} > t) \le
P(\hat{\tau}_{\hat{u},\hat{v}} - \hat{\nu}_{\hat{u},\hat{v}} > t/2) + 
P({\hat{\nu}_{\hat{u},\hat{v}}} > t/2). 
$$
\vspace{-0.2cm}

By the stochastic domination mentioned before and supposing that a coalescence time bound as the one that appears in Lemma \ref{lemma:dnbcoalescencetime} holds for dual l-paths starting at integer positions (we prove it in Appendix \ref{sec:dualcoalesce}), we have that $P(\hat{\nu}_{\hat{u},\hat{v}} > t/2) \le Ce^{-ct}$. Now, recalling that $k$ denotes the initial distance between the two dual l-paths, we have

\begin{align*}
	P\left(\hat{\tau}_{\hat{u},\hat{v}} - \hat{\nu}_{\hat{u},\hat{v}}> \frac{t}{2}\right) &\le \frac{C}{\sqrt{t}}\mathbb{E}\big[\hat{Z}^{\hat{u},\hat{v}}_{\hat{\nu}_{\hat{u},\hat{v}}}\big]\\
	&=\frac{C}{\sqrt{t}}\Big(k + \sum_{m=1}^{\infty}\mathbb{E}\left[(\hat{Z}^{\hat{u},\hat{v}}_{\hat{\nu}_{\hat{u},\hat{v}}}-k)\big| \hat{\nu}_{\hat{u},\hat{v}} = m \right]P\left(\hat{\nu}_{\hat{u},\hat{v}} = m\right)\Big)\\
	&\le \frac{C}{\sqrt{t}}\Big(k + \sum_{m=1}^{\infty}cm P\left(\hat{\nu}_{\hat{u},\hat{v}} = m\right)\Big)
	= \frac{C}{\sqrt{t}}\left(k + c\mathbb{E}\left[\hat{\nu}_{\hat{u},\hat{v}}\right]\right),
\end{align*}
where $C>0$ is an upper bound on the conditional expectation of an increment of the distance between two dual l-paths, i.e., one can find some $C$ that does not depend on the initial distance between both paths. It is straightforward to compute $C$, and we leave the details to the reader. Thus we only need to consider that both dual paths start from integer positions.

Denote by $\hat{Z}^k_t$ the distance after $t$ steps between the position of two dual paths that start in integer positions at time zero with initial distance $k$, and by $\hat{\tau}_k$ the coalescence time of these two dual paths, i.e:
$$
\hat{Z}^k_t = | \hat{Y}^l_t((0,0)) - \hat{Y}^l_t((k,0)) |, \ t \geq 0, \ \, \textrm{and }
\,
\hat{\tau}_k = \min \{ t \geq 0: \hat{Z}^k_t = 0 \}.
$$

\vspace{0.1cm}

\begin{remark}
Recall that the dual paths flow in the reverse direction of time. We put the notation in such a way to deal only with positive values for the time index of the process $\hat{Z}^{\hat{u},\hat{v}}_t$, but the coalescence between two dual paths starting from the dual vertices $\hat{u}$ and $\hat{v}$ will actually happen at time $(\hat{u}_2 \wedge \hat{v}_2) - \hat{\tau}_{\hat{u},\hat{v}}$.
\end{remark}

Remark \ref{remark:rightpaths} is also applied here. We are considering only dual l-paths, but everything would be analogous for r-paths.

\smallskip

Now we are finally ready to present an equivalent of Lemma \ref{lemma:dnbcoalescencetime} for dual paths.

\begin{lemma}
\label{lemma:dualcoalescencetime}
There exists a constant $\hat{C}_0 > 0$ such that for every $t > 0$ and $k \in \mathbb{N}$ we have:
$$
P(\hat{\tau}_k > t) \leq \frac{\hat{C}_0 k}{\sqrt{t}},
$$
where this coalescence time can refer to two dual l-paths or to two dual r-paths.
\end{lemma}

Since dual paths can be in integer positions or not, we cannot prove Lemma \ref{lemma:dualcoalescencetime} just using Proposition \ref{Result:coalescencetime} as we did for the DNB. To deal with this problem, we will use the following strategy: we define specific regeneration times where both dual paths are in integer positions, then we show that these times appear with enough frequency allowing us to simplify our problem considering the distance between the two dual paths only in these convenient regeneration times. The proof is presented in Section \ref{sec:dualcoalesce}.

\vspace{0.2cm}


\section{Convergence of the finite-dimensional distributions}
\label{sec:conditionI}

\vspace{0.2cm}

This section is devoted to proving that (I$_{\mathcal{N}}$) holds for the DNB. Recall that the branching parameter is $\epsilon_n = \fb n^{-1}$, for $\fb > 0$. So we begin by stating the main result of this section.

\begin{remark}
Whenever we have a path $\pi$ that starts at time $\sigma_{\pi}$, we can consider that $\pi_t =  \pi_{\sigma_{\pi}}$ for $t < \sigma_{\pi}$. It does not affect any convergence result and simplifies the notation when paths start from distinct times.
\end{remark}

\begin{proposition}
\label{lemma:conditionI}
Let $\epsilon_n = \fb n^{-1}$. Consider the sequences of points $z_1^{(n)}, \ldots , z_k^{(n)}$ and $ \tilde{z}_1^{(n)}, \ldots , \tilde{z}_{\tilde{k}}^{(n)}$ in $\mathbb{Z}^2$ where $z_i^{(n)} = (x_i^{(n)},s_i^{(n)})$ with $(\frac{1}{n}x^{(n)}_{i},\frac{1}{n^2}s^{(n)}_{i})\rightarrow (x_{i},s_{i}) = z_i \in \mathbb{R}^2$ for $i=1, \ldots , k$; and $\tilde{z}_j^{(n)} = (\tilde{x}_j^{(n)},\tilde{s}_j^{(n)})$ with $(\frac{1}{n}\tilde{x}^{(n)}_{j},\frac{1}{n^2}\tilde{s}^{(n)}_{j})\rightarrow (\tilde{x}_{j},\tilde{s}_{j}) = \tilde{z}_j \in \mathbb{R}^2$ for $j=1, \ldots , \tilde{k}$. Denote by $l_i^{(n)}$ the l-path of the DNB with branching parameter $\epsilon_n$ starting from $z_i^{(n)}$, and by $r_j^{(n)}$ the r-path of the DNB with branching parameter $\epsilon_n$ starting from $\tilde{z}_j^{(n)}$. Then the following convergence holds:
$$
\left(\frac{l_1^{(n)}\!(tn^2)}{n}, \ldots , \frac{l_k^{(n)}\!(tn^2)}{n},\frac{r_1^{(n)}\!(tn^2)}{n}, \ldots , \frac{r_{\tilde{k}}^{(n)}\!(tn^2)}{n} \right)_{t\in \mathbb{R}} \underset{n \rightarrow \infty}{\Longrightarrow} \left(l_1(t), \ldots l_k(t), r_1(t), \ldots r_{\tilde{k}}(t) \right)_{t\in \mathbb{R}},
$$ 
as random elements of the space of continuous functions $C(\mathbb{R},\mathbb{R}^{k+\tilde{k}})$ endowed with the uniform norm topology, where $(l_1, \ldots l_k, r_1, \ldots r_{\tilde{k}})$ is a collection of left-right coalescing Brownian motions (as defined in Section \ref{subsec:Bnet}), starting from $(z_1, \ldots z_k, \tilde{z}_1, \ldots \tilde{z}_{\tilde{k}})$.
\end{proposition}

\vspace{0.2cm}

As a first step to prove Proposition \ref{lemma:conditionI} (one of the main results in this paper), we state and prove the convergence only for a pair of left-right paths in Proposition \ref{lemma:onepairconverge} just below. Some of our arguments in this first stage are similar to the ones that we find in \cite{ss1} for simple random walks with branching which are independent before coalescence. But the dependence on paths in the DNB implies that we need a new approach to prove convergence. After that, we generalize the proof to more than one pair of paths using ideas presented in \cite{cv} to prove an equivalent condition for the generalized Drainage Network without branching.

\begin{proposition}
\label{lemma:onepairconverge}
Let $\epsilon_n = \fb n^{-1}$. Consider sequences $(x^{(n)})_{n\ge 1}$ and $(y^{(n)})_{n\ge 1}$ in $\mathbb{Z}$ such that $\lim_{n\rightarrow \infty} x^{(n)}/n = x$ and $\lim_{n\rightarrow \infty} y^{(n)}/n = y$, $x,y \in \mathbb{R}$. Denote by $(L_t^{(n)})_{t \geq 0}$ the l-path of $\mathcal{X}^n_{\epsilon_n}$ starting from $(x^{(n)},0)$ and by $(R_t^{(n)})_{t \geq 0}$ the r-path of $\mathcal{X}^n_{\epsilon_n}$ starting from $(y^{(n)},0)$. Then the following convergence holds:
$$\left(\frac{1}{n}L_{tn^2}^{(n)},\frac{1}{n}R_{tn^2}^{(n)}\right)_{t \geq 0} \underset{n \rightarrow \infty}{\Longrightarrow} (L_t,R_t)_{t \geq 0},
$$
where $(L_t,R_t)_{t \geq 0}$ is the unique solution of (\ref{eq:diferentialeq-b}) with initial state $(L_0,R_0) = (x,y)$ under the constraint that $L_t \leq R_t$ for all $t \geq T = \inf\{s \geq 0: L_s \leq R_s\}$.
\end{proposition}

\vspace{0.2cm}

We divided the proofs of Propositions \ref{lemma:conditionI} and \ref{lemma:onepairconverge} into three subsections. Sections \ref{subsec:partI} and \ref{subsec:partII} together prove Proposition \ref{lemma:onepairconverge}, and Section \ref{subsec:partV} is dedicated to the prove Proposition \ref{lemma:conditionI}. 

\vspace{0.2cm}

\subsection{Part I: Only one pair $(L_t^{(n)},R_t^{(n)})_{t\ge 0}$ with $R_0^{(n)}\geq L_0^{(n)}$} \hfill\\ 
\label{subsec:partI}

First we state and prove a result that assures that the family $\{(L_t^{(n)},R_t^{(n)})_{t\ge 0} : n\ge 1\}$ is tight. So here we consider $L^{(n)} =(L_t^{(n)})_{t \geq 0}$ and $R^{(n)} = (R_t^{(n)})_{t \geq 0}$ as defined in Proposition \ref{lemma:onepairconverge}.

\begin{lemma}
\label{lemma:individualconvergence}
The processes $L^{(n)}$ and $R^{(n)}$ converge weakly under diffusive scaling to Brownian motions with diffusion coefficient $\lambda_p^2$ and drifts respectively $-b_p$ and $b_p$, where $\lambda_p^2$ is defined in Remark \ref{remark:sigmap} and $b_p = \frac{\fb (1-p)}{(2-p)^2}$. 
\end{lemma}

\begin{proof}
The convergences are a consequence of Donsker's Theorem. Indeed we will write the increments of $L^{(n)}$ as a sum of two sequences: One with mean zero converges to a Brownian motion under diffusive scaling by a direct application of Donsker's Theorem. The other one converges in probability to the drift function $t \mapsto -b_p t$ under diffusive scaling by the Law of Large Number. An analogous proof holds for $R^{(n)}$ for the drift function  $t \mapsto b_p t$, hence we will only present the proof for $L^{(n)}$.

Note that $L^{(n)}$ is equal in distribution to $L^{a,(n)} + L^{b,(n)}$, where $L^{a,(n)} = (L_t^{a,(n)})_{t \ge 0}$ evolves as a linearly interpolated path from the Drainage Network without branching and $L^{b,(n)} = (L_t^{b,(n)})_{t \ge 0}$ is a linearly interpolated jump process that only jumps in times where $L^{a,(n)}$ jumps to the right. At those jump times, $L^{b,(n)}$ will have probability $\frac{\fb p(1-p)}{2n(2-p)}$ to do a jump to the left with a length equal to twice of the jump length of $L^{a,(n)}$ and probability $1 - \frac{\fb p(1-p)}{2n(2-p)}$ to stay in the same position. Indeed $\frac{\fb p(1-p)}{2n(2-p)}$ is the probability that branching occurs in the DNB from a given vertex $(x,t)$ and, through a uniform choice, the right-hand side is chosen as the destination from that vertex. We can compute it as:
$$\textrm{(branching parameter)} \times \textrm{(side choice probability)} \times \textrm{(branching probability in full-DNB)}
		$$
		$$
		= \epsilon_n \times \frac{1}{2}  \times P((x,t+1) \textrm{ is closed}) \times \textrm{(probability of equality between two IID Geom(p))}
		$$
		$$
		= \frac{\fb}{2n} \times (1-p) \times \frac{p}{2-p} = \frac{\fb p(1-p)}{2n(2-p)}.
$$
So $L^{b,(n)}$ is a drift correction for paths in the Drainage Network without branching which generates paths identically distributed to l-paths in the DNB.

Both process $L^{a,(n)}$ and $L^{b,(n)}$ have independent increments with finite variance. We have that $L^{a,(n)}$ converges to a Brownian motion under diffusive scaling by Donsker's Theorem. Moreover $(L_t^{b,(n)} + \frac{b_p t}{n})_{t\ge 0}$ is a martingale for any fixed $n \in \mathbb{N}$, where $b_p = \frac{\fb (1-p)}{(2-p)^2}$.  Indeed we have that $-\frac{b_p}{n}$ is the mean value of the increments of $L^{b,(n)}$. The mean value is computed as follows:
$$
-\sum_{k=1}^\infty 2 k \, \textrm{(probability that } ((x,t),(x\pm k,t+1)) \in E \textrm{ and edge } ((x,t),(x + k,t+1)) \textrm{ is chosen)}
$$
$$
= - \sum_{k=1}^\infty 2 k \, \epsilon_n \times  \frac{1}{2}  \times P((x,t+1) \textrm{ is closed}) \times  \textrm{(probability that two IID Geom(p) are equal to k)}
$$
$$
= - \frac{\fb (1-p)}{n} \sum_{k=1}^\infty k (1-p)^{2(k-1)} p^2 = - \frac{\fb (1-p)}{n(2-p)^2} = - \frac{b_p}{n}.
$$
Now denote by $(Z_j^{b,(n)})_{j\ge 1}$ the increments of $L^{b,(n)}$ at integer times. Note that 
$Z_j^{b,(n)} \left| Z_j^{b,(n)} \neq 0\right.$ is stochastically dominated by $2 W$ with $W \sim Geom(p)$,
which implies that 
$$
Var\left(Z_j^{b,(n)}\right) \le E\big[ (Z_j^{b,(n)})^2 \big] = E\big[ (Z_j^{b,(n)})^2 \big| Z_j^{b,(n)} \neq 0 \big] P\big( Z_j^{b,(n)} \neq 0 \big) \leq \frac{2 \fb (1-p)}{np} \rightarrow 0
$$ 
as $n \rightarrow \infty$.

Apply Doob maximal inequality and we have for any fixed $t > 0$, $\delta >0$ and $n$ sufficiently large depending on $\delta$ that

\begin{align*}
	P \left( \sup_{0 \leq s \leq t} \left[ \frac{1}{n}L_{\left\lfloor sn^2 \right\rfloor}^{b,(n)} + b_p s\right] \geq \delta \right) &\le P\left(\max_{0 \leq k \leq \left\lfloor tn^2 \right\rfloor} \left[ \frac{1}{n} \sum_{j=1}^k \left( Z_j^{b,(n)} + \frac{b_p}{n}\right) \right] \geq \frac{\delta}{2} \right) 
\end{align*}
\begin{align*}
	&\leq  \frac{C}{\delta^2 n^2} Var\Big( \sum_{j=1}^{\left\lfloor tn^2 \right\rfloor} Z_j^{b,(n)}  \Big)
	\le \frac{C \, t}{\delta^2}\left(1 + \frac{1}{n^2}\right) Var\left(Z_j^{b,(n)}\right) \rightarrow 0, \, \textrm{ as } n\rightarrow \infty.
\end{align*}
Hence, 
\begin{equation}
L^{b,(n)}
\textrm{ converges under diffusive scaling to the deterministic path }
t \mapsto -b_p t
\label{eq:drift}
\end{equation}\
and the convergence holds in probability on every bounded interval $[0,T]$, $T\ge 0$.
\end{proof}

\vspace{0.1cm}

Now we present a result that gives some control on the probability of occurrence of long jumps in the paths of the DNB. It will be required further ahead in this section.

\begin{lemma}
\label{lemma:longjump}
Let $\pi_z(t)$, for $t > \sigma_\pi$ be a random path in DNB that starts from an open vertex $z$ at time $\sigma_\pi$. We have that, for any fixed $s>0$, $g \ge 1$ and $n\ge 1$,
$$P\left( \sup_{\sigma_\pi < t \leq \sigma_\pi + sn^2} | \pi_z(t) - \pi_z(t-1)| \geq g \right) \le 2 s n^2 e^{-c(g-1)},$$
for some $c = c(p) > 0$. In particular, for any $\gamma >0$, the probability that a path in the DNB makes a jump of size greater than $g(n) = n^\gamma$ in a time interval of order $O(n^2)$ goes to zero as $n \rightarrow \infty$. 
\end{lemma}

\begin{proof}
Let $Y_t = | \pi_z(t) - \pi_z(t-1)|$, for $t \in \mathbb{N}$ and $t \geq \sigma_\pi$. First, since the distribution of the path increments is invariant in time and space, we can assume that $\pi_z(t)$ starts at time 0 (i.e. $\sigma_\pi = 0$). Now, since $Y_t \preceq G \sim Geom(p)$ and using Markov property, for any fixed $s>0$
$$
P\left( \sup_{0 < t \leq sn^2} Y_t \geq g \right) \leq (\lfloor s \rfloor + 1) n^2 P(G \ge \lceil g \rceil) \le  2sn^2(1-p)^{g-1} .
$$  
To conclude the proof just take $c = -\log(1-p)$.
\end{proof}

\begin{corollary}
\label{cor:longjump}
For any $K > -2/\log(1-p)$, the probability that a path in the DNB makes a jump of size greater than $K \log(n)/2$ in a time interval of order $O(n^2)$ goes to zero as $n \rightarrow \infty$.
\end{corollary}

\vspace{0.2cm}

For the remainder of this section, $\kappa > -2/\log(1-p)$ is a fixed positive real number. For any fixed $s>0$, we have by Corollary \ref{cor:longjump} that the probability that $L^{(n)}$ or $R^{(n)}$ makes any jump of size greater than $\kappa \log(n)/2$ in the time interval $[0,sn^2]$ converges to zero as $n$ goes to infinity. So, we can construct our arguments conditioning on the event that those long jumps do not happen, and we will assume that from now on.

Our initial strategy would be to represent $(L^{(n)},R^{(n)})$ as the solution of a difference equation, which in the limit yields an SDE with a unique solution that also solves the following equations: 

\begin{eqnarray}
&&(i) \ \
dL_t = d\tilde{B}_{T_t}^l + d\tilde{B}_{S_t}^s - b_p dt, \nonumber \\
&&(ii) \ \
dR_t = d\tilde{B}_{T_t}^r + d\tilde{B}_{S_t}^s + b_p dt, \nonumber \\
&&(iii) \ \ T_t + S_t = t \ \forall \, t>0, \nonumber \\
&&(iv) \ \
\displaystyle{\int_{0}^{t} \mbox{I}_{\{L_s < R_s\}} d S_s } = 0, \ \forall \, t>0,
\label{eq:LrRtalternative}
\end{eqnarray}
where $(\tilde{B}_{t}^l)_{t\ge 0}$, $(\tilde{B}_{t}^r)_{t\ge 0}$ and $(\tilde{B}_{t}^s)_{t\ge 0}$ are independent Brownian motions. Under the restriction that  $L_t \leq R_t$ for all $t \geq 0$, \eqref{eq:LrRtalternative} is equivalent to (\ref{eq:diferentialeq-b}) due to \cite[Lemma 2.2]{ss1} that we present just below. We point out that \cite[Lemma 2.2]{ss1} was stated assuming $b_p = 1$, but the extension to non unitary drift case is straightforward.

\begin{proposition}
\label{Result:equationequivalence}
\textbf{(\cite[Lemma 2.2 ]{ss1})} Denote $R^2_{\leq} = \{(x,y) \in \mathbb{R}^2: x \leq y\}$.
\begin{enumerate}
\item[(a)] There is a one-to-one correspondence in law between weak $R^2_{\leq}$-valued solutions of (\ref{eq:diferentialeq-b}) and weak $R^2_{\leq}$-valued solutions of (\ref{eq:LrRtalternative}).
\item[(b)] For each initial state $(L_0,R_0) \in R^2_{\leq}$, equation (\ref{eq:LrRtalternative}) has a unique pathwise solution.\
\item[(c)] Solutions to (\ref{eq:LrRtalternative}) satisfy $\displaystyle{S_t = \int_0^t I_{\{L_s = R_s\}}ds}$,
$$S_t = 0 \vee \sup_{0 \leq s \leq T_t} \left(\frac{1}{2}(L_0 + \tilde{B}_s^l - R_0 - \tilde{B}_s^r) - b_p s \right) \quad a.s.,$$
and $\displaystyle{\lim_{t\rightarrow \infty} T_t = \infty}$.
\end{enumerate}
\end{proposition}

\begin{remark}
A characterization of a solution of equation \eqref{eq:LrRtalternative} as a pair of reflected left-right Brownian paths is given in \cite[Proposition 3.2]{ss}. We do not state it here, as it will not simplify our exposition, but it may be useful to the reader to understand our proof.
\end{remark}

Comparing \cite[Lemma 2.2]{ss1} with our case, due to the dependence between $L^{(n)}$ and $R^{(n)}$, it is harder to represent $(L^{(n)},R^{(n)})$ as the solution of a convenient difference equation. The problem is that when $L^{(n)}$ and $R^{(n)}$ are sufficiently close to each other, it is hard to couple with a pair of independent left and right paths reflected at crossing attempts. So we will create an auxiliary process $\tilde{R}^{(n)}$ that will evolve in the same way as $R^{(n)}$ when it is far from $L^{(n)}$ and will evolve according to a different rule when it is near to $L^{(n)}$. The process $\tilde{R}^{(n)}$ will be created in such a way that we can represent $(L^{(n)},\tilde{R}^{(n)})$ as the solution of a difference equation that converges to the solution of \eqref{eq:LrRtalternative} and that the distance between $R^{(n)}$ and $\tilde{R}^{(n)}$ becomes negligible under diffusive scaling. The definition of $\tilde{R}^{(n)}$ is rather complex, since we should keep it evolving independently of $L^{(n)}$ as much as possible, blocking crossings between them.

The process $\tilde{R}^{(n)}$ will be constructed through a coupling that uses a convenient alternation between independent environments. 
Since we are conditioning on the event that $L^{(n)}$ and $R^{(n)}$  do not make jumps of size $\kappa \log(n)/2$ in some time interval $[0,sn^2]$, we essentially want to use the fact that, if two paths of the DNB are far from each other (by a distance of $n^{\frac{3}{4}} >> \kappa \log(n)$ at least), then conditionally these two paths will only need to observe two disjoint sets of vertices in the next time level to make their jump decisions, then their choices are made independently from each other.

Recall the definition of the environment in Section \ref{model}. We say that $(\omega(z))_{z \in \mathbb{Z}^2}$ and  $(\theta(z))_{z \in \mathbb{Z}^2}$ define the {\emph{$(\omega,\theta)$-environment}}.  By definition, $(L^{(n)},R^{(n)})$ evolves according to the $(\omega,\theta)$-environment. Let $\tilde{\omega}=(\tilde{\omega}(z))_{z\in\mathbb{Z}^2}$ and $\tilde{\theta}=(\tilde{\theta}(z))_{z\in\mathbb{Z}^2}$ be independent families of random variables identically distributed as $\omega$ and $\theta$ respectively, and also independent of them. Therefore, $(\tilde{\omega},\tilde{\theta})$ is a new environment. 
The process $\tilde{R}^{(n)}$ will evolve almost as a path of a DNB, and its jump choices will be obtained from an alternation between the $(\omega,\theta)$-environment and the $(\tilde{\omega},\tilde{\theta})$-environment. 

Let us denote by $\Delta_t^{(n)}$ the distance between $\tilde{R}_t^{(n)}$ and $L_t^{(n)}$ for $t\ge 0$. We construct $\tilde{R}^{(n)}$ setting $\tilde{R}^{(n)}_0 = R^{(n)}_0$ and then following the steps ($\tilde{R}$.1), ($\tilde{R}$.2) and ($\tilde{R}$.3) described below (recall that $\kappa$ was fixed before):
\begin{enumerate}
\item[($\tilde{R}$.1)] 
Let us define
\begin{equation}\label{tau}
    \varrho = \inf\{t \ge 0: |R_t^{(n)} - L_t^{(n)}| < n^{\frac{3}{4}}\}.
\end{equation}
Now let us consider $\{\Hat{R}_t^{(n)}: t\ge \varrho\}$ a DNB r-path, starting from $R_{\varrho}^{(n)}$, but considering the $(\tilde{\omega},\tilde{\theta})$-environment and conditioning that it does not make jumps of size greater than $\kappa \log(n)/2$ in some time interval. Also let us define the following stopping times,
\begin{align*}
    \nu &= \inf\{t \ge \varrho: |\Hat{R}_t^{(n)} - L_t^{(n)}| > n^{\frac{7}{8}}\},\\
    \psi &= \inf\{t \ge \varrho: |\Hat{R}_t^{(n)} - L_t^{(n)}| < \kappa \, n^{\frac{1}{4}}\log(n)^2\},
\end{align*}
and $\phi = \nu \wedge \psi$. Then we define
\begin{equation}
	\tilde{R}_t^{(n)}=\left\lbrace\begin{array}{lc}R_t^{(n)} & \text{, for } 0\le t \le \varrho,\\
	\Hat{R}_t^{(n)} & \text{, for } \varrho \le t\le \phi.
	\end{array}\right.
\end{equation}
If $\phi = \nu$ then the construction of $\tilde{R}^{(n)}$ continues in the step ($\tilde{R}$.2), otherwise it continues in the step ($\tilde{R}$.3). 

\noindent {\bf Remark about ($\tilde{R}$.1)}: $\tilde{R}^{(n)}$ and $R^{(n)}$ evolve together until $\Delta^{(n)}$ becomes smaller than $n^{\frac{3}{4}}$. Hence,  $\tilde{R}^{(n)}$ jumps according to $(\omega,\theta)$-environment during this time. After that, $\tilde{R}^{(n)}$ jumps according to $(\tilde{\omega},\tilde{\theta})$-environment. If at time zero, $\tilde{R}^{(n)}$ and $R^{(n)}$ are already at a distance smaller than $n^{\frac{3}{4}}$, $\tilde{R}^{(n)}$ immediately starts to evolve according to $(\tilde{\omega},\tilde{\theta})$-environment. Then, from the time  $\Delta^{(n)}$ is smaller than $n^{\frac{3}{4}}$, we wait until either $\Delta^{(n)}$ becomes greater than $n^{\frac{7}{8}}$, and we go to step ($\tilde{R}$.2), or $\Delta^{(n)}$ becomes smaller than $\kappa \, n^{\frac{1}{4}}\log(n)^2$, and we go to step ($\tilde{R}$.3).

\item[($\tilde{R}$.2)] 
Let us consider the following stopping times,
\begin{align*}
    \Hat \varrho &= \inf\{t \ge \phi: |\Hat{R}_t^{(n)} - L_t^{(n)}| < n^{\frac{3}{4}}\}\\
    \Hat \nu &= \inf\{t \ge \phi: |\Hat{R}_t^{(n)} - L_t^{(n)}| > n^{\frac{7}{8}} \text{ and } \Hat{R}_t^{(n)} = R_t^{(n)}\}.
\end{align*}
Then $\tilde{R}^{(n)}$ continues evolving as $\Hat{R}_t^{(n)}$ up to time $\Hat \varrho \wedge \Hat \nu$, i.e
\begin{equation*}
    \tilde{R}_t^{(n)} = \Hat{R}_t^{(n)} \text{, for } \phi \le t \le \Hat \varrho \wedge \Hat \nu .
\end{equation*}
In case $\Hat \nu < \Hat \varrho$, the definition of $\tilde{R}^{(n)}$ continue as follows. Consider
\begin{equation*}
     \mathring{\nu} = \inf\{t \ge \Hat \nu: |R_t^{(n)} - L_t^{(n)}| < n^{\frac{3}{4}}\},
\end{equation*}
and 
\begin{equation*}
    \tilde{R}_t^{(n)} = R_t^{(n)} \ \, \text{, for } \Hat \nu < t\le \mathring{\nu}.
\end{equation*}
So, we have defined $\tilde{R}^{(n)}$ up to the random time,
\begin{equation*}
	\mathring{\varrho} = \left\lbrace
 \begin{array}{ll} 
 \Hat \varrho & \textrm{, if } \Hat \varrho < \Hat \nu,\\
\mathring{\nu} & \textrm{, if } \Hat \nu < \Hat \varrho.
\end{array}
 \right.
\end{equation*}
To define $\tilde{R}^{(n)}$ for $t > \mathring{\varrho}$, note that, at time at $\mathring{\varrho}$, the process $\tilde{R}^{(n)}$ is already at a distance smaller than $n^{\frac{3}{4}}$ from $L^{(n)}$. Then, we come back to ($\tilde{R}$.1) considering $\varrho$ in \eqref{tau} as $\mathring{\varrho}$.

\noindent {\bf Remark about ($\tilde{R}$.2):} After step ($\tilde{R}$.1), if $\Delta^{(n)}$ is greater than $n^{\frac{7}{8}}$, we proceed with the construction of $\tilde{R}^{(n)}$ as follows: while $\tilde{R}^{(n)}$ evolves using the independent $(\tilde{\omega},\tilde{\theta})$-environment, either $\Delta^{(n)}$ becomes smaller than $n^{\frac{3}{4}}$ again, then the construction returns to step ($\tilde{R}$.1), or $\tilde{R}^{(n)}$ meets $R^{(n)}$ in a position at distance greater than $n^{\frac{7}{8}}$ from $L^{(n)}$, and then $\tilde{R}^{(n)}$ returns to jump according to $(\omega,\theta)$-environment. In the latter case $\tilde{R}^{(n)}$ and $R^{(n)}$ return to move together until $\Delta^{(n)}$ becomes smaller than $n^{\frac{3}{4}}$ again, and then, the construction returns to step ($\tilde{R}$.1).

\item[($\tilde{R}$.3)] 

\medskip

Note that at this stage of the construction we have defined the process $\tilde{R}^{(n)}$ up to time $\phi$ and $(\tilde{R}_{\phi}^{(n)} - L_{\phi}^{(n)}) \le \kappa \, n^{\frac{1}{4}}\log(n)^2$. To avoid crossing, if $\tilde{R}_{\phi}^{(n)} < L_{\phi}^{(n)}$ we redefine $\tilde{R}_\phi^{(n)}$ as $\tilde{R}_{\phi}^{(n)}  = L_{\phi}^{(n)}$. However, note that conditional to the event that $L^{(n)}$, $R^{(n)}$ and $\Hat R^{(n)}$ do not make jumps of size $\kappa \log(n)/2$, the crossing attempt does not happen at time $\phi$ and $\tilde{R}_\phi^{(n)}$ does not need to be redefined. 

Let us set $\mathring{\zeta}_0 = \phi$ and define recursively the sequence $\{\zeta_k, k\ge 1\}$ and $\{\mathring{\zeta}_k, k\ge 1\}$ of stopping times, as well as the evolution of $\tilde{R}^{(n)}$ between those times, as follows:
\begin{itemize}
\item[1.] Having defined $\mathring{\zeta}_{k - 1}$, set 
$$\zeta_k = \inf\{ t \ge \mathring{\zeta}_{k-1} : \text{ the } (\omega,\theta) \text{-DNB path branches at time } t \text{ at site } L_{t}^{(n)}\}.
$$

\item[2.] For $\mathring{\zeta}_{k-1} \le t < \zeta_k$, $\tilde{R}^{(n)}$ have the same increments of the process $L^{(n)}$, i.e.,
$$
\tilde{R}_{t+1}^{(n)} - \tilde{R}_t^{(n)} = L_{t+1}^{(n)} - L_{t}^{(n)}.
$$
During those time intervals $\Delta^{(n)}$ remains constant. We write  $\tilde{R}^{(n)} \approx L^{(n)}$ to represent the union of the random sets $\{t:\mathring{\zeta}_{k-1} \le t < \zeta_k\}$, $k\ge 1$, i.e. the set of times where $\tilde{R}^{(n)}$ is being forced to have the same increments of $L^{(n)}$. Note that if $\mathring{\zeta}_{k-1} = \zeta_k$ nothing happens in this stage. 

\item[3.] At time $\zeta_k$, $L^{(n)}$ jumps to the left and we set $\tilde{R}^{(n)}$ by imposing that it jumps to the right by the same amount of the jump of $L^{(n)}$, i.e.,
$$
\tilde{R}_{\zeta_k+1}^{(n)} - \tilde{R}_{\zeta_k}^{(n)} = |L_{\zeta_k+1}^{(n)} - L_{\zeta_k}^{(n)}|.
$$
\item[4.] If 
$$
\Delta_{\zeta_k+1}^{(n)} > \kappa \, n^{\frac{1}{4}}\log(n)^2,
$$
we go back to step ($\tilde{R}$.1) considering $\varrho$ in \eqref{tau} as $\zeta_{k}+1$, otherwise we continue.
\item[5.] From time $\zeta_k+1$, $\tilde{R}^{(n)}$ returns to evolve using the $(\tilde{\omega},\tilde{\theta})$-environment (independently of $L^{(n)}$). 
\item[6.] set
$$
\mathring{\zeta}_k = \inf\{ t > \zeta_{k}+1 : \text{ either } \Delta_{t}^{(n)} < \Delta_{t-1}^{(n)} \text{ or } \Delta_{t}^{(n)} > \kappa \, n^{\frac{1}{4}}\log(n)^2 \} .
$$
\item[7.] If 
$$
\Delta_{\mathring{\zeta}_k}^{(n)} > \kappa \, n^{\frac{1}{4}}\log(n)^2,
$$
we go back to step ($\tilde{R}$.1) considering $\varrho$ in \eqref{tau} as $\mathring{\zeta}_k$, otherwise we continue.
\item[8.] If 
$$
\Delta_{\mathring{\zeta}_k}^{(n)} < \Delta_{\mathring{\zeta}_k - 1}^{(n)} 
$$
we go back to stage 1 of step ($\tilde{R}$.3) to define $\zeta_{k+1}$. 
If 
$$
\Delta_{\mathring{\zeta}_k}^{(n)} < 0
$$ 
we redefine $\tilde{R}_{\mathring{\zeta}_k}$ as $\tilde{R}_{\mathring{\zeta}_k}^{(n)}  = L_{\mathring{\zeta}_k}^{(n)}$ to avoid crossing (which implies that  $\Delta_{\mathring{\zeta}_k}^{(n)}$ becomes equals to zero) and then go back to stage 1 of step ($\tilde{R}$.3) to define $\zeta_{k+1}$. 
\end{itemize}
Note that the interaction rule in this step blocks crossing attempts (that are possible due to independence).

\noindent {\bf Remark about ($\tilde{R}$.3):} Here we impose a non-crossing rule between $\tilde{R}^{(n)}$ and $L^{(n)}$. After step ($\tilde{R}$.1) or the event described at the end of this step ($\tilde{R}$.3) itself, $\Delta^{(n)}$ is smaller than $\kappa \, n^{\frac{1}{4}}\log(n)^2$. If we have an attempt of crossing between $\tilde{R}^{(n)}$ and $L^{(n)}$, we will force $\tilde{R}^{(n)}$ to stop in the same position of $L^{(n)}$. Whenever $\Delta^{(n)}$ becomes lesser than $\kappa \, n^{\frac{1}{4}}\log(n)^2$ (including an attempt of crossing),  $\tilde{R}^{(n)}$ will begin to make the same movements as those of $L^{(n)}$, keeping $\Delta^{(n)}$ constant. They remain evolving like this until the $(\omega,\theta)$-DNB path decides to branch in a position that $L^{(n)}$ is occupying. At this moment, $\Delta^{(n)}$ will increase by twice the jump length of $L^{(n)}$ (note that it is the same behavior that $R^{(n)}$ and $L^{(n)}$ have when they meet each other, but with $\tilde{R}^{(n)}$ we enforce this behavior before they meet each other). After that, $\tilde{R}^{(n)}$ returns to jump according to the $(\tilde{\omega},\tilde{\theta})$-environment (independently of $L^{(n)}$), even if $\Delta^{(n)}$ continues lesser than $\kappa \, n^{\frac{1}{4}} \log(n)^2$. We write  $\tilde{R}^{(n)} \approx L^{(n)}$ to represent the random set of times $t$ such that $\tilde{R}_t^{(n)}$ is being forced to have the same increments of $L_t^{(n)}$. From the moment that $\tilde{R}^{(n)}$ and $L^{(n)}$ start to jump independently again, either $\Delta^{(n)}$ decreases by any amount and then we restart step ($\tilde{R}$.3) (we will have $\tilde{R}^{(n)} \approx L^{(n)}$ again), or $\Delta^{(n)}$ becomes greater than $\kappa \, n^{\frac{1}{4}}\log(n)^2$ and then we return to step ($\tilde{R}$.1).
\end{enumerate}

By the construction of $\tilde{R}^{(n)}$ described in steps ($\tilde{R}$.1) - ($\tilde{R}$.3), it always jumps according to $R^{(n)}$, $L^{(n)}$ or the process $\Hat{R}^{(n)}$ introduced in step ($\tilde{R}$.1), which evolves according to the $(\tilde{\omega},\tilde{\theta})$-environment. Then we can also use Corollary \ref{cor:longjump} to check that $\tilde{R}^{(n)}$ does not make a jump of size greater than $\kappa \log(n)/2$ in time interval $[0,sn^2]$ with a probability that goes to one as $n$ goes to infinity. Anyway, since we are conditioning on the event that neither $R^{(n)}$ nor $L^{(n)}$ nor $\hat{R}^{(n)}$ makes jumps of size greater than $\kappa \, n^{\frac{1}{4}} \log(n)^2$, we have that $\tilde{R}^{(n)}$ will not make long jumps too.

Since $R^{(n)}$ and $L^{(n)}$ have drifts in opposite directions (see Lemma \ref{lemma:individualconvergence}), we only go back from step ($\tilde{R}$.2) to ($\tilde{R}$.1) in the construction of $\tilde{R}^{(n)}$ a finite number of times almost surely. This fact together with the estimates for coalescence times between $\tilde{R}^{(n)}$ and $R^{(n)}$ when they are near to each other will imply that the pairs $(L^{(n)},R^{(n)})$ and $(L^{(n)},\tilde{R}^{(n)})$ converge weakly to the same limit under diffusive scaling. It is proved in Lemma \ref{proposition:RtilandR}.

\begin{remark}
1. The above construction is fundamental throughout the rest of this section. For instance the choice of the distance $\kappa \, n^{\frac{1}{4}}\log(n)^2$ in step ($\tilde{R}$.3) is only justified in the proof of Lemma \ref{proposition:lastterm} at the end of the section. 2. Intuitively $(L^{(n)},\tilde{R}^{(n)})$ will converge to the solution of \eqref{eq:LrRtalternative} because, neglecting long jumps, it evolves as a pair of independent left-right paths reflected to avoid crossings, except for an amount of time that is negligible under diffusive scaling. This is the amount of time where $\tilde{R}^{(n)} \approx L^{(n)}$, and most of our work will be to show the claimed assertion that it is negligible under diffusive scaling.
\end{remark}

To illustrate the relation between $\tilde{R}^{(n)}$ and $L^{(n)}$, we have Figure \ref{tildeRtLt}, where we considered two different scenarios that can happen when the construction of $\tilde{R}^{(n)}$ reaches step ($\tilde{R}$.3).

In the scenario (a) of Figure \ref{tildeRtLt}, we
start at step ($\tilde{R}$.1). At time $\tau_1$, the distance $\Delta^{(n)}$ becomes smaller than $\kappa \, n^{\frac{1}{4}} \log(n)^2$ and we pass to step ($\tilde{R}$.3). At time $\tau_2$, the $(\omega,\theta)$-DNB branches at site $L_{\tau_2}^{(n)}$, thereafter $\Delta^{(n)}$ increases to $\kappa \, n^{\frac{1}{4}} \log(n)^2$ and then we go back to step ($\tilde{R}$.1). At time $\tau_3$, $\Delta^{(n)}$ becomes lesser than $\kappa \, n^{\frac{1}{4}} \log(n)^2$ a second time and then we return to step ($\tilde{R}$.3), having $\tilde{R}^{(n)} \approx L^{(n)}$ again. 

In the scenario (b) of Figure \ref{tildeRtLt}, we also start at step ($\tilde{R}$.1). At time $\tau_1$, the distance $\Delta^{(n)}$ becomes smaller than $\kappa \, n^{\frac{1}{4}} \log(n)^2$ and we pass to step ($\tilde{R}$.3), starting to have $\tilde{R}^{(n)} \approx L^{(n)}$. At time $\tau_2$, the $(\omega,\theta)$-DNB branches at site $L_{\tau_2}^{(n)}$, and then, $\tilde{R}^{(n)}$ and $L^{(n)}$ return to move independently. Although, $\Delta^{(n)}$ continues lesser than $\kappa \, n^{\frac{1}{4}} \log(n)^2$, so we are still in step ($\tilde{R}$.3). At time $\tau_3$, $\Delta^{(n)}$ decreases and then we return to have $\tilde{R}^{(n)} \approx L^{(n)}$. What has occurred at times $\tau_2$ and $\tau_3$ occurs again respectively at times $\tau_4$ and $\tau_5$.

\begin{figure}[!ht]
\centering
\includegraphics[scale=0.36]{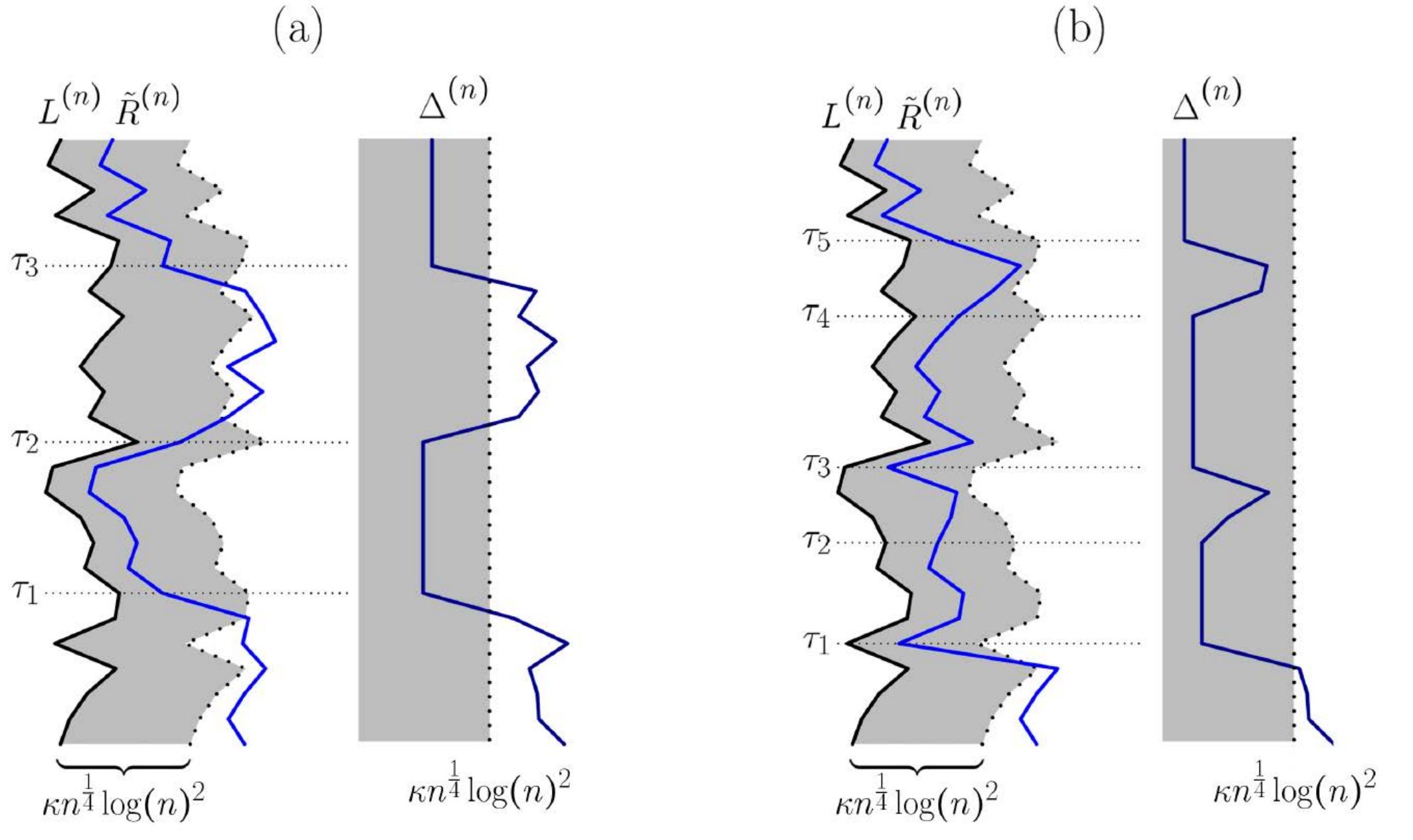}
\caption{Example of the interaction between $\tilde{R}^{(n)}$ and $L^{(n)}$ considering two different scenarios after the construction of $\tilde{R}^{(n)}$ reaches step ($\tilde{R}$.3). The pictures on the right in each scenario show the evolution of $\tilde{R}^{(n)}$ and $L^{(n)}$, where the dotted line indicates the distance $\kappa \, n^{\frac{1}{4}}\log(n)^2$ at right-hand side of $L^{(n)}$. The pictures on the left in each scenario show the evolution of $\Delta^{(n)} = \tilde{R}^{(n)} - L^{(n)}$ (which is constant while $\tilde{R}^{(n)} \approx L^{(n)}$). Whenever $\Delta^{(n)}$ enters the gray area or makes a jump to the left while it is inside the gray area, we begin to have $\tilde{R}^{(n)} \approx L^{(n)}$ until the $(\omega,\theta)$-DNB branches at a site that $L^{(n)}$ is occupying. The random times $\tau_i$, for i odd, indicate the times when we begin to have $\tilde{R}^{(n)} \approx L^{(n)}$ and the random times $\tau_i$, for i even, indicate the times where $\tilde{R}^{(n)}$ returns to move independently of $L^{(n)}$.}
\label{tildeRtLt}
\end{figure}

To prepare ourselves to write the system of difference equations for the pair $(L_t,\tilde{R}_t)$, we need to define some auxiliary processes. Let $V^l = (V_t^l)_{t \in \mathbb{N}}$, $V^r = (V_t^r)_{t \in \mathbb{N}}$ and $V^s = (V_t^s)_{t \in \mathbb{N}}$ be independent discrete-time symmetric Markov chains defined in $\mathbb{Z}$, starting at the origin at time zero and with the following transition probabilities:

\vspace{0.2cm} 

$P_v(x,x) = p$, and 
$P_v(x,y) = p(1-p)^{2|y-x|} + \frac{p^2}{2} (1-p)^{2|y-x|-1}$ , \; \ $\forall \, y \neq x,$

\vspace{0.2cm} 

\noindent where $p$ is the probability parameter of the DNB. The transition probabilities of these processes coincide with those of a path from the Drainage Network without branching.

For $\alpha = l,s$, let $D^{(n),\alpha,-} = (D_t^{(n),\alpha,-})_{t \in \mathbb{N}}$ be a process that can move only at times where $V^{\alpha}$ jumps to a position at its right, and at these times, $D^{(n),\alpha,-}$ will have probability $\frac{\fb p(1-p)}{2n(2-p)}$ to do a jump to the left with a length equal to twice the length of the jump of $V^{\alpha}$ and probability $1 - \frac{\fb p(1-p)}{2n(2-p)}$ to remain in the same position (see the proof of Lemma \ref{lemma:individualconvergence}). Likewise, for $\alpha = r,s$, let $D^{(n),\alpha,+} = (D_t^{(n),\alpha,+})_{t \in \mathbb{N}}$ be a process that can move only in the times where $V^{\alpha}$ jumps to a position at its left, and at these times, $D^{(n),\alpha,+}$ will have probability $\frac{\fb p(1-p)}{2n(2-p)}$ to do a jump to the right with a length equal to twice the length of the jump of $V^{\alpha}$ and probability $1 - \frac{\fb p(1-p)}{2n(2-p)}$ to remain in the same position. The choices are independent for $\alpha = r,l,s$, therefore $D^{(n),l,-}$,  $D^{(n),r,+}$ and $(D^{(n),s,-},D^{(n),s,+})$ are independent processes.

Denote by $J^{(n)}$ the set of integer times in which $\tilde{R}_t^{(n)} \approx L_t^{(n)}$ and recall that we are conditioning on the event that $L_t^{(n)}$ and $\tilde{R}_t^{(n)}$ do not make jumps of size greater than $\kappa \log(n)/2$. Then the pair $(L_t^{(n)},\tilde{R}_t^{(n)})_{t \geq 0}$ can be described at integer times as the solution of:

\begin{eqnarray}
&&(i) \ \ L_t^{(n)} = V_{T_t^{(n)}}^l + D_{T_t^{(n)}}^{(n),l,-} + V_{S_t^{(n)}}^s + D_{S_t^{(n)}}^{(n),s,-},\nonumber\\
&&(ii) \ \ \tilde{R}_t^{(n)} = V_{T_t^{(n)}}^r + D_{T_t^{(n)}}^{(n),r,+} + V_{S_t^{(n)}}^s + D_{S_t^{(n)}}^{(n),s,+} + \displaystyle{\sum_{j=1}^t U_j^{(n)}},\nonumber\\
&&(iii) \ \ T_t^{(n)} = \sum_{s=0}^{t-1} \mbox{I}_{\{s \notin J^{(n)}\}},\nonumber\\
&&(iv) \ \ S_t^{(n)} = \sum_{s=0}^{t-1} \mbox{I}_{\{s \in J^{(n)}\}},
\label{eq:LrRtdifequation}
\end{eqnarray}

\noindent where $U_j^{(n)}$, $j \in \mathbb{N}$, are random variables that assume the value of
\begin{equation}\label{Us}
\Big( \Big[ L_j^{(n)} - \big( \tilde{R}_{j-1}^{(n)} + \big( V_{T_j^{(n)}}^r + D_{T_j^{(n)}}^{(n),r,+} \big) - \big( V_{T_{j-1}^{(n)}}^r + D_{T_{j-1}^{(n)}}^{(n),r,+} \big) \big) \Big] \ \mbox{I}_{\big\{ \tilde{R}_{j-1}^{(n)} > L_{j-1}^{(n)}, (j-1) \notin J^{(n)} \big\} } \Big)_{+},
\end{equation}
where $(x)_{+}$ denotes the positive part of $x$. The random variables $U_j^{(n)}$ and the last term in the equation (ii) are here to deal with the corrections in $\tilde{R}_j^{(n)}$ that we have to make whenever an attempts to cross $L_j^{(n)}$ is blocked. At these crossing attempts, the correction puts $\tilde{R}^{(n)}$ on the same position of $L^{(n)}$.

Now define $L^{(n)}$, $\tilde{R}^{(n)}$, $V^{\alpha}$, $D^{(n),l,-}$, $D^{(n),r,+}$, $D^{(n),s,\pm}$,  $ T^{(n)}$, $S^{(n)}$ and $U^{(n)}$ at non integer times by linear interpolation. The rescaled processes then satisfy the following equations:

\vspace{-0.1cm}

$$(i) \ \ \frac{1}{n} L_{tn^2}^{(n)} = \frac{1}{n} V_{T_{tn^2}^{(n)}}^l + \frac{1}{n} D_{T_{tn^2}^{(n)}}^{(n),l,-} + \frac{1}{n} V_{S_{tn^2}^{(n)}}^s + \frac{1}{n} D_{S_{tn^2}^{(n)}}^{(n),s,-},$$

\vspace{0.4cm}

$$(ii) \ \ \frac{1}{n} \tilde{R}_{tn^2}^{(n)} = \frac{1}{n} V_{T_{tn^2}^{(n)}}^r + \frac{1}{n} D_{T_{tn^2}^{(n)}}^{(n),r,+} + \frac{1}{n} V_{S_{tn^2}^{(n)}}^s + \frac{1}{n} D_{S_{tn^2}^{(n)}}^{(n),s,+} + \frac{1}{n} \displaystyle{\sum_{j=1}^{tn^2} U_j^{(n)}},$$

\vspace{0.4cm}

$$(iii) \ \ \frac{1}{n^2} T_{tn^2}^{(n)} + \frac{1}{n^2} S_{tn^2}^{(n)} = t,$$

\begin{equation}
(iv) \displaystyle{\int_{0}^{t} \mbox{I}_{\left\{\frac{1}{n} \tilde{R}_{sn^2}^{(n)} - \frac{1}{n} L_{sn^2}^{(n)} \geq \kappa \log(n)^2 n^{-\frac{3}{4}} \right\}} d\left( \frac{1}{n^2} S_{sn^2}^{(n)} \right)} = 0.
\label{eq:LrRtequation}
\end{equation}

\vspace{0.2cm}

Applying Donsker's Theorem and the same arguments used in the proof of Lemma \ref{lemma:individualconvergence}, we have that
\begin{equation}
\left( \frac{1}{n} V_{tn^2}^l, \frac{1}{n} V_{tn^2}^r, \frac{1}{n} V_{tn^2}^s, -\frac{1}{n} D_{tn^2}^{(n),l,-}, \frac{1}{n} D_{tn^2}^{(n),r,+}, -\frac{1}{n} D_{tn^2}^{(n),s,-}, \frac{1}{n} D_{tn^2}^{(n),s,+} \right)_{t\geq0} 
\label{eq:7tuple}
\end{equation}
converges in distribution when $n \rightarrow \infty$ to $(\tilde{B}_t^l,\tilde{B}_t^r,\tilde{B}_t^s, b_p t, b_p t, b_p t, b_p t)_{t \geq 0}$, where $\tilde{B}_t^l$, $\tilde{B}_t^r$ and $\tilde{B}_t^s$ are independent Brownian motions with diffusion coefficient $\lambda_p^2$, $b_p = \frac{\fb (1-p)}{(2-p)^2}$, as verified in Lemma \ref{lemma:individualconvergence}. This convergence is the same obtained in \cite{ss1} to prove Proposition \ref{lemma:onepairconverge} for independent paths.

We still have to analyze the limit behavior of the last term in (ii) of \eqref{eq:LrRtequation}, but our next lemma gives us that this term must converge to zero in probability and consequently does not affect the limit behavior of $\frac{1}{n}\tilde{R}_{tn^2}^{(n)}$.

\begin{lemma}
\label{proposition:lastterm}
Considering $(U_j^{(n)})_{j\in \mathbb{N}}$ as defined in \eqref{Us}, we have that:
$$\frac{1}{n} \displaystyle{\sum_{j=1}^{tn^2} U_j^{(n)}} {\underset{n \rightarrow \infty}\longrightarrow} 0 ~\text{in probability} .$$
\end{lemma}

\vspace{0.1cm}

The proof of Lemma \ref{proposition:lastterm} will be postponed. It will be presented immediately before Section \ref{subsec:partII}. 

Now note that, since $t \mapsto \displaystyle{\frac{1}{n^2} T_{tn^2}^{(n)}}$ and $t \mapsto \displaystyle{\frac{1}{n^2} S_{tn^2}^{(n)}}$ increases with slope at most $1$, the laws of $\displaystyle{\Big\{\Big( \frac{1}{n^2} T_{tn^2}^{(n)}\Big)_{t \geq 0} \Big\}_{n \in \mathbb{N}}}$ and $\displaystyle{ \Big\{\Big( \frac{1}{n^2} S_{tn^2}^{(n)}\Big)_{t \geq 0} \Big\}_{n \in \mathbb{N}}}$ are tight. By the convergence of the 7-tuple in \eqref{eq:7tuple} and Lemma \ref{proposition:lastterm}, we have that $\displaystyle{\Big\{ \Big( \frac{1}{n} L_{tn^2}^{(n)},  \frac{1}{n} \tilde{R}_{tn^2}^{(n)}\Big)_{t \geq 0} \Big\}_{n \in \mathbb{N}}}$ is also tight and consequently, for $n \in \mathbb{N}$, the laws of the $11$-tuple which consists of the $7$-tuple in \eqref{eq:7tuple} joint with
$$
\displaystyle{\left( \displaystyle{\frac{1}{n} L_{tn^2}^{(n)}, \frac{1}{n} \tilde{R}_{tn^2}^{(n)}, \frac{1}{n^2} T_{tn^2}^{(n)}, \frac{1}{n^2} S_{tn^2}^{(n)}} \right)_{t \geq 0}}
$$
is tight. Let us prove that all limit point process of subsequences have the same distribution and to achieve that, let us consider a subsequence that converge in distribution to
\begin{equation}
\left(\tilde{B}_t^l,\tilde{B}_t^r,\tilde{B}_t^s, b_p t, b_p t, b_p t, b_p t, L_t, R_t, T_t, S_t \right)_{t \geq 0}.  
\label{eq:11tuple}
\end{equation}
For this subsequence, by Skorohod's Representation Theorem, we can couple this $11$-tuple for $n \in \mathbb{N}$ and the limiting process in \eqref{eq:11tuple}, such that the convergence is almost surely.

Assuming this coupling, we claim that $(L_t, R_t, T_t, S_t)_{t \geq 0}$ solves equation \eqref{eq:LrRtalternative} and consequently is uniquely determined in law by Proposition \ref{Result:equationequivalence}. In fact (i), (ii) and (iii) of \eqref{eq:LrRtalternative} follows immediately from making $n$ go to infinity in (i), (ii) and (iii) of \eqref{eq:LrRtequation} respectively. Now to verify that (iv) of \eqref{eq:LrRtalternative} also holds, let us choose a continuous non-decreasing function $\rho_{\delta}: [0,\infty) \mapsto \mathbb{R}$, for each $\delta >0$, such that $\rho_{\delta}(u) = 0$ for $u \leq \delta$ and $\rho_{\delta}(u) = 1$ for $u \geq 2\delta$. Then, since $\rho_\delta$ is bounded, by \eqref{eq:LrRtequation} and our convergence assumption, we have that
\begin{align*}
	0 &= \lim_{n \rightarrow \infty} \displaystyle{\int_{0}^{t} \mbox{I}_{\left\{\frac{1}{n} \tilde{R}_{sn^2}^{(n)} - \frac{1}{n} L_{sn^2}^{(n)} \geq \kappa \log(n)^2 n^{-\frac{3}{4}}\right\}} d\left( \frac{1}{n^2} S_{sn^2}^{(n)} \right)} \\
	&\geq \lim_{n \rightarrow \infty} \displaystyle{\int_{0}^{t} \rho_{\delta}\left( \frac{1}{n}\tilde{R}_{sn^2}^{(n)} - \frac{1}{n} L_{sn^2}^{(n)} \right) d\left( \frac{1}{n^2} S_{sn^2}^{(n)} \right)} 
	= \int_{0}^t \rho_{\delta}(\tilde{R}_s - L_s)dS_s \, , \forall \delta>0\, .
\end{align*}
Letting $\delta \downarrow 0$, by Dominated Convergence Theorem, we have (iv) of \eqref{eq:LrRtalternative}. It gives us that the limit in distribution of any subsequence of $\left( \frac{1}{n} L_{tn^2}^{(n)}, \frac{1}{n} \tilde{R}_{tn^2}^{(n)} \right)$ is a left-right Brownian motion, which implies that $(L^{(n)},R^{(n)})$ under diffusive scaling converges to a left-right Brownian motion.

It remains to verify that under diffusive scaling $(L^{(n)},R^{(n)})$
and $(L^{(n)},\tilde{R}^{(n)})$
have the same weak limit. This is the content of the next lemma which completes the proof of Part I.

\begin{lemma}
\label{proposition:RtilandR}
The pairs $(L^{(n)}, R^{(n)})$ and $(L^{(n)}, \tilde{R}^{(n)})$ converge weakly to the same limit under diffusive scaling. 
\end{lemma}

\vspace{0.1cm}

\noindent \textit{Proof.} First note that $R^{(n)}$ and $\tilde{R}^{(n)}$ are equal except in some specific time windows (depending on the distances between $R^{(n)}$ and $L^{(n)}$, and between $\tilde{R}^{(n)}$ and $L^{(n)}$). Also note that for any $\beta<1$, paths that are at a distance smaller than $n^{\beta}$ from each other are shrunken to the same path under diffusive scaling, thus we do not need to worry with time windows where $\tilde{R}^{(n)}$ is at a distance smaller than $n^{\beta}$ from $R^{(n)}$. Therefore our strategy here will be to show that we can fix an exponent $\beta < 1$ such that the time windows where $\tilde{R}^{(n)}$ can be at a distance greater than $n^{\beta}$ from  $R^{(n)}$ almost surely shrink to zero under diffusive scaling. This is enough to ensure that $(L^{(n)}, R^{(n)})$ and $(L^{(n)}, \tilde{R}^{(n)})$ have the same limit.

Recall steps ($\tilde{R}$.1), ($\tilde{R}$.2) and ($\tilde{R}$.3), in the definition of $\tilde{R}^{(n)}$. Define the sequence $\{ \mathfrak{T}_k^{(n)}, \mathfrak{S}_k^{(n)}\}_{k \ge 1}$ as follows,
\begin{align*}
     \mathfrak{T}_1^{(n)} &= \inf\{ t \ge 0: |L^{(n)}_t - \tilde{R}^{(n)}_t| \le n^{\frac{3}{4}}\},\\
     \mathfrak{S}_1^{(n)} &= \inf\{ t > \mathfrak{T}_1^{(n)}: |L^{(n)}_t - \tilde{R}^{(n)}_t| > n^{\frac{7}{8}} \text{ and } R_t^{(n)} = \tilde{R}_t^{(n)}\},
\end{align*}
and having defined $ \mathfrak{T}_1^{(n)},\mathfrak{S}_1^{(n)},\dots,  \mathfrak{T}_k^{(n)}, \mathfrak{S}_k^{(n)}$ for $k \ge 1$, consider
\begin{align*}
     \mathfrak{T}_{k + 1}^{(n)} &= \inf\{ t > \mathfrak{S}_k^{(n)}: |L^{(n)}_t - \tilde{R}^{(n)}_t| \le n^{\frac{3}{4}}\},\\
     \mathfrak{S}_{k + 1}^{(n)} &= \inf\{ t \ge \mathfrak{T}_{k + 1}^{(n)}: |L^{(n)}_t - \tilde{R}^{(n)}_t| > n^{\frac{7}{8}} \text{ and } R_t^{(n)} = \tilde{R}_t^{(n)}\},
\end{align*}
Now fix $s>0$ and $\frac{7}{8} < \beta < 1$. We will be interested in time intervals $[\mathfrak{T}_k^{(n)},\mathfrak{S}_k^{(n)}]$ where $\mathfrak{T}_k^{(n)}< sn^2$ and either $R^{(n)}$ or $\tilde{R}^{(n)}$ reaches a distance greater than $n^\beta$ from $L^{(n)}$. Thus put 
$$
\begin{array}{rcl}
\Upsilon_{n,s,\beta} & := & \Big\{k\ge 1: \mathfrak{T}_k^{(n)}< sn^2 \textrm{ and } R^{(n)} \textrm{ or } \tilde{R}^{(n)} \\
& & \qquad \quad \textrm{ reaches a distance greater than } n^\beta \textrm{ from } L^{(n)} \textrm{ in } \big[\mathfrak{T}_k^{(n)},\mathfrak{S}_k^{(n)}\big]\Big\}. 
\end{array}
$$
To simplify notation we will simply write $\Upsilon_n$ suppressing the indication of $s$ and $\beta$.
Therefore we have that
$$
\mathfrak{T}_1^{(n)}  < \mathfrak{S}_1^{(n)} < ... < \mathfrak{T}_k^{(n)} < \mathfrak{S}_k^{(n)} < ...
< \mathfrak{T}_{\max \Upsilon_n}^{(n)} < \mathfrak{S}_{\max \Upsilon_n}^{(n)},
$$
with 
$\mathfrak{T}_{\max \Upsilon_n}^{(n)} \le \lfloor sn^2 \rfloor \le \mathfrak{S}_{\max \Upsilon_n}^{(n)}$ or $\mathfrak{S}_{\max \Upsilon_n}^{(n)} < \lfloor sn^2 \rfloor$. Moreover
\begin{itemize}
\item $\mathfrak{T}_k^{(n)}<sn^2$ is a time where $R^{(n)}$ and $\tilde{R}^{(n)}$ stop evolving together (that is, $\tilde{R}^{(n)}$ begins to evolve according to $(\tilde{\omega},\tilde{\theta})$-environment);
\item $\mathfrak{S}_k^{(n)}$ is a time where $R^{(n)}$ and $\tilde{R}^{(n)}$ meet each other at a distance greater than $n^{\frac{7}{8}}$ from $L^{(n)}$, returning to evolve together;
\item  In each time interval $[\mathfrak{T}_k^{(n)},\mathfrak{S}_k^{(n)}]$ with $k\in \Upsilon_n$, either $R^{(n)}$ or $\tilde{R}^{(n)}$ reaches a distance greater than $n^\beta$ from $L^{(n)}$.
\end{itemize} 

Denote by $\widetilde{T}_k^{\beta,(n)}$ the amount of time inside $[\mathfrak{T}_k^{(n)},\mathfrak{S}_k^{(n)}]$ where $\tilde{R}^{(n)}$ is at a distance greater than $n^{\beta}$ from $R^{(n)}$. Since $\min\{\tilde{R}^{(n)},R^{(n)}\} \ge L^{(n)}$, $\widetilde{T}_k^{\beta,(n)}$ can be positive only when $k \in \Upsilon_n$. Then
it is enough to show that, for some $\gamma < 2$,
\begin{equation}
\lim_{n\rightarrow\infty} P \Big( \sum_{k\in \Upsilon_n} \widetilde{T}_k^{\beta,(n)} > n^\gamma\Big) = 0.
\label{eq:distanttime}
\end{equation}
Note that equation \eqref{eq:distanttime} means that when $n$ is large enough, the amount of time in $[0,sn^2]$ where $\tilde{R}^{(n)}$ is at a distance of order $n$ (or greater) from $R^{(n)}$ become of order smaller than $n^2$. It implies that this amount of time is negligible under diffusive scaling as $n$ goes to infinity.

To prove \eqref{eq:distanttime}, first we need some control on $\# \Upsilon_{n}$, when $n$ goes to infinity. To do that, consider
$$
\begin{array}{rcl}
\widetilde{\Upsilon}_{n,\beta} := \Big\{k\ge 1: R^{(n)} \textrm{ or } \tilde{R}^{(n)} \textrm{ reaches a distance greater than } n^\beta \textrm{ from } L^{(n)} \textrm{ in } \big[\mathfrak{T}_k^{(n)},\mathfrak{S}_k^{(n)}\big]\Big\}, 
\end{array}
$$
and $\tilde{N}^{(n)}_\beta = \# \widetilde{\Upsilon}_{n,\beta}$. Clearly $\# \Upsilon_{n} \leq \tilde{N}^{(n)}_\beta$ for every $n \in \mathbb{N}$ and $s>0$. 
By the Strong Markov property we have that $\tilde{N}^{(n)}_\beta$ is stochastically bounded above by a geometric random variable with some parameter that we will denote by $\tilde{p}^{(n)}_\beta$. We cannot claim that the distribution is a geometric random variable, because the probability that $k \in \widetilde{\Upsilon}_{n,\beta}$ depends on the overshoot of $R^{(n)}$ or $\tilde{R}^{(n)}$ when they become greater than $n^{\frac{7}{8}}$. However the overshoot distribution is stochastically bounded, which implies the stochastic domination for $\tilde{N}^{(n)}_\beta$.

Now we will obtain a lower bound to $\tilde{p}^{(n)}_\beta$. To do that, we assume without loss of generality that $\tilde{R}^{(n)}$ reaches the distance $n^{\beta}$ from $L^{(n)}$ before $R^{(n)}$, but the argument would be the same in the other scenario. Given that the distance between $\tilde{R}^{(n)}$ and $L^{(n)}$ is greater than $n^{\beta}$, we have that $\Delta^{(n)} = \tilde{R}^{(n)} - L^{(n)}$  evolve as a random walk with drift $\displaystyle{\frac{2b_p}{n}}$ (twice of the drift of each path) until $\Delta^{(n)} \leq n^{\frac{3}{4}}$. Now let us define the events 

\vspace{0.3cm}

\noindent $\tilde{E}_n = \{ \Delta^{(n)}$ reaches a position greater than or equal to $n$  before it reaches a position smaller than or equal to $n^{\frac{3}{4}}$ given that $\Delta_0^{(n)} \ge n^{\beta} \},$ 

\vspace{0.2cm}

\noindent $\tilde{E}_{2,n} = \{ \Delta^{(n)}$ never reaches a position smaller than or equal to $n^{\frac{3}{4}}$ given that $\Delta_0^{(n)} \ge n\},$

\vspace{0.2cm}

\noindent $\tilde{F}_n = \{ \Delta^{(n)}$ never reaches a position smaller than or equal to $n^{\frac{3}{4}}$ after it reaches a position greater than or equal to $n^{\beta}\}$.

\vspace{0.2cm}

First we will find a lower bond for $P(\tilde{E}_n)$, so let's assume here that $\Delta_0^{(n)} = n^\beta$. For $n$ fixed let $\tau_{E}^{(n)} = \inf_{t > 0} \{ \Delta_t^{(n)} \notin (n^{\frac{3}{4}}, n)\}$. Now we denote by $\xi_{-}^{(n)}$ and $\xi_{+}^{(n)}$ the overshoot distribution of $\Delta^{(n)}$ when it crosses the position $n^{\frac{3}{4}}$ from right to left and when it crosses the position $n$ from left to right, respectively. Note that $\tau_{E}^{(n)}$ is an almost surely finite stopping time and $\big(\Delta_t^{(n)}\big)_{t < \tau_{E}^{(n)}}$ is a $\widetilde{\mathcal{F}}_t$-submartingale, where $\widetilde{\mathcal{F}}_t = \sigma\{(\omega(z),\theta(z),\tilde{\omega}(z),\tilde{\theta}(z)), z=(z_1,z_2), z_1 \in \mathbb{Z}, z_2 \leq t\}$, $t \in \mathbb{Z}$, which represents the filtration generated by the environments $(\omega,\theta)$ and $(\tilde{\omega},\tilde{\theta})$. 
Thus $\big(\Delta_{\tau_{E}^{(n)} \wedge k}^{(n)}\big)_{k\ge 1}$ is bonded in $L^2$, then uniformly integrable. The Optional Stopping Theorem gives us that
$$
E[\Delta_0^{(n)}] \leq E\Big[\Delta_{\tau_{E}^{(n)}}^{(n)}\Big]
$$
which implies 
$$
  n^\beta \le (n^{\frac{3}{4}}-\tilde{C}_{-}^{(n)}) (1-P(\tilde{E}_n)) + (n+\tilde{C}_{+}^{(n)})P(\tilde{E}_n) \le  n^{\frac{3}{4}} (1- P(\tilde{E}_n)) + (n+\tilde{C}_{+}^{(n)})P(\tilde{E}_n) ,
$$
where $\tilde{C}_{-}^{(n)} = E[\xi_{-}^{(n)}] > 0$ for all $n \in \mathbb{N}$ and 
$\tilde{C}_{+}^{(n)} = E[\xi_{+}^{(n)}] \leq \frac{5}{p}$ for $n$ large enough by Claim \ref{claim:overshoot} in Appendix \ref{subsec:conditionIresults}. Thus
$$ P(\tilde{E}_n) \geq \frac{n^\beta-n^{\frac{3}{4}}}{n-n^{\frac{3}{4}}+\tilde{C}_{+}^{(n)}} \ge \frac{n^\beta-n^{\frac{3}{4}}}{n-n^{\frac{3}{4}}+\frac{5}{p}} = \frac{1}{n^{1-\beta}}  \left(\frac{1 - n^{\frac{3}{4} - \beta} }{ 1 -\frac{n^{\frac{3}{4}} - \frac{5}{p}}{n} } \right) \ge \frac{1}{2n^{1-\beta}},
$$ 
for large values of $n$. 
Concerning the event $\tilde{E}_{2,n}$, by Donsker's and Portmanteau Theorems, we have that
\begin{align}\label{eq:browniansup} 
	\liminf_{n \rightarrow \infty} P(\tilde{E}_{2,n}) &\geq \liminf_{n \rightarrow \infty} P\Big(\inf_{t > 0}\Delta_{t}^{(n)} > n^{\frac{3}{4}} \big| \Delta_0^{(n)} = n\Big)\nonumber \\
	&> P\Big(\inf_{t > 0} B_t^{(0,0),2b_p,2\lambda_p^2} > -\frac{1}{2} \Big)
	= 1 - \exp\Big\{ -\frac{\sqrt{2}b_p}{\lambda_p}\Big\}
	= \frac{1}{\tilde{C}_1} > 0,
\end{align}
where $B_s^{z,\mu,\sigma^2}$ is a Brownian motion with drift $\mu$ and diffusion $\sigma^2$, starting at point $z$. The equality \eqref{eq:browniansup} occurs because $- (\inf_{t \geq 0} B_t^{(0,0),\mu,\sigma^2})$ when $\mu > 0$ has an exponential distribution with parameter $\frac{2\mu}{\sigma}$ (see e.g. Section 6.8 of \cite{resnick}). The result in \eqref{eq:browniansup} gives that $P(\tilde{E}_{2,n}) \geq (\tilde{C}_1)^{-1}$ for large values of $n$.

So, by the Strong Markov property, $\tilde{p}^{(n)}_\beta = P(\tilde{F}_n) \geq P(\tilde{E}_n)P(\tilde{E}_{2,n}) \geq \frac{1}{2\tilde{C}_1 n^{1-\beta}}$ for large values of $n$ and then, $\tilde{N}^{(n)}_\beta$ is stochastically bounded from above by a geometric random variable $G_\beta^{(n)}$ with parameter $\frac{1}{2\tilde{C}_1 n^{1-\beta}}$. Since $\{\widetilde{T}_k^{\beta,(n)}, k \in \widetilde{\Upsilon}_{n,\beta}\}$ is independent of $\tilde{N}^{(n)}_\beta$ and denoting $\widetilde{\Upsilon}_{n,\beta} = \{ k_j : j\le \tilde{N}^{(n)}_\beta \}$, we can write 
$$
P\Big( \sum_{k\in \Upsilon_n} \widetilde{T}_k^{\beta,(n)} > n^\gamma\Big)  \leq   P\Big(  \sum_{j = 1}^{\tilde{N}^{(n)}_\beta} \widetilde{T}_{k_j}^{\beta,(n)} > n^\gamma \Big)
$$ 
which can be decomposed as
\begin{align} 
	&\sum_{m=1}^{\infty} P\Big(  \sum_{j=1}^{\tilde{N}^{(n)}_\beta} \widetilde{T}_{j}^{\beta,(n)} > n^\gamma \Big| \tilde{N}^{(n)}_\beta = m\Big) P(\tilde{N}^{(n)}_\beta = m) \leq \sum_{m=1}^{\infty} P\Big(  \bigcup_{j=1}^{m} \left\{\widetilde{T}_{k_j}^{\beta,(n)} > \frac{n^\gamma}{m} \right\}\Big) P(\tilde{N}^{(n)}_\beta = m) \nonumber\\
	& \leq \sum_{m=1}^{\infty} \sum_{j=1}^{m} P\Big(\widetilde{T}_{k_j}^{\beta,(n)} > \frac{n^\gamma}{m}\Big) P(\tilde{N}^{(n)}_\beta = m) \leq \sum_{m=1}^{\infty} m \max_j P\left( \widetilde{T}_{k_j}^{\beta,(n)} > \frac{n^\gamma}{m} \right) P(\tilde{N}^{(n)}_\beta = m) \nonumber
\end{align}
\begin{align} 
	& \leq \sum_{m=1}^{\left\lfloor n^{1-\beta}\log(n)^2 \right\rfloor} m P(\tilde{N}^{(n)}_\beta = m) \max_j P\left( \widetilde{T}_{k_j}^{\beta,(n)} > \frac{n^\gamma}{n^{1-\beta}\log(n)^2} \right)  + \sum_{m=\left\lceil n^{1-\beta}\log(n)^2 \right\rceil}^{\infty}  m P(\tilde{N}^{(n)}_\beta = m) \nonumber\\
	&\leq E[\tilde{N}^{(n)}_\beta] \max_j P\left( \widetilde{T}_{k_j}^{\beta,(n)} > \frac{n^\gamma}{n^{1-\beta}\log(n)^2} \right) + \sum_{m=\left\lceil n^{1-\beta}\log(n)^2 \right\rceil}^{\infty}  m P(\tilde{N}^{(n)}_\beta \geq m), \nonumber
\end{align}
and using the stochastic domination, the rightmost side of the above inequality is bounded above by
\begin{align} \label{eq:tildeN}
	&E[G_{\beta}^{(n)}] \max_j P\left( \widetilde{T}_{k_j}^{\beta,(n)} > \frac{n^\gamma}{n^{1-\beta}\log(n)^2} \right) + \sum_{m=\left\lceil n^{1-\beta}\log(n)^2 \right\rceil}^{\infty}  m P(G_{\beta}^{(n)} \geq m)\nonumber\\
	&\leq 2\tilde{C}_1 n^{1-\beta} \max_j P\left( \widetilde{T}_{k_j}^{\beta,(n)} > \frac{n^{\gamma+\beta-1}}{\log(n)^2} \right) + \ 2\tilde{C}_1n^{1-\beta} \left( 1 - \frac{1}{2\tilde{C}_1n^{1-\beta}} \right)^{n^{1-\beta}\log(n)^2-1}.
\end{align}

\begin{remark}
    Note that we have not claimed that $\{\widetilde{T}_k^{\beta,(n)}, k \ge 1\}$ are independent or identically distributed, and we don't need these properties in the calculations above.
\end{remark}

Assume without loss of generality that $1 \in \widetilde{\Upsilon}_{n,\beta}$. Note that the rightmost term in \eqref{eq:tildeN} converges to zero as $n$ goes to infinity and then, to finish the proof, we will show that
$$P
\Big(\widetilde{T}_1^{\beta,(n)} > \frac{n^{\gamma+\beta-1}}{\log(n)^2}\Big) < \frac{1}{n^{1-\beta}\log(n)},
$$
for large values of $n$ and some $\gamma < 2$.  It will be clear that this last bound holds uniformly for any index $k \in \widetilde{\Upsilon}_{n,\beta}$. To prove it, we will introduce convenient subintervals of $[\mathfrak{T}_1^{(n)},\mathfrak{S}_1^{(n)}]$, thus we need to introduce more notation. Define
$$
\mathfrak{T}_{1,1}^{(n)} = \inf \big\{ t > \mathfrak{T}_1^{(n)} \, : \,
\max \big(|\tilde{R}_{t}^{(n)} - L_{t}^{(n)} | , |R_{t}^{(n)} - L_{t}^{(n)}| \big) > 2n^{\frac{7}{8}} \big\},
$$
$$
\mathfrak{S}_{1,1}^{(n)} = \inf \big\{ t \ge \mathfrak{T}_{1,1}^{(n)} \, : \,
R_{t}^{(n)} = \tilde{R}_{t}^{(n)} \big\}, 
$$
and having defined $\mathfrak{T}_{1,1}^{(n)}$ and $\mathfrak{S}_{1,1}^{(n)}$, continue defining $\mathfrak{T}_{1,j}^{(n)}$ and $\mathfrak{S}_{1,j}^{(n)}$ for $j \geq 2$ in the following way
$$
\mathfrak{T}_{1,j}^{(n)} = \inf \big\{ t > \mathfrak{S}_{1,j-1}^{(n)} \, : \,
\max \big(|\tilde{R}_{t}^{(n)} - L_{t}^{(n)} | , |R_{t}^{(n)} - L_{t}^{(n)}| \big) > 2n^{\frac{7}{8}} \big\},
$$
$$
\mathfrak{S}_{1,j}^{(n)} = \inf \big\{ t \ge \mathfrak{T}_{1,j}^{(n)} \, : \,
R_{t}^{(n)} = \tilde{R}_{t}^{(n)} \big\}, 
$$
until $\mathfrak{S}_{1,j}^{(n)} = \mathfrak{S}_{1}^{(n)}$, i.e, until the meeting between $R^{(n)}$ and $\tilde{R}^{(n)}$ occur at a distance equal or greater than $n^{\frac{7}{8}}$ from $L^{(n)}$. Then consider
\[
\nu^{(n)} = \max \{ j \geq 1 \, : \mathfrak{S}_{1,j}^{(n)} \le \mathfrak{S}_{1}^{(n)}\, \},
\]
\[
\widetilde{S}_j^{(n)} = \mathfrak{S}_{1,j}^{(n)} - \mathfrak{T}_{1,j}^{(n)}.
\]
Note that given $1 \in \widetilde{\Upsilon}_{n,\beta}$, $\mathfrak{S}_{1,1}^{(n)} \le \mathfrak{S}_1^{(n)}$, thus $\nu^{(n)}\ge 1$. Moreover for any $t \in [\mathfrak{T}_1^{(n)},\mathfrak{S}_1^{(n)}]$, we have that  $|R_t^{(n)} - \tilde{R}_t^{(n)}| < 2n^{\frac{7}{8}} < n^{\beta}$ (without loss of generality we are supposing $n$ sufficiently large such that $2n^{\frac{7}{8}} < n^{\beta}$) if $t \notin [\mathfrak{T}_{1,j}^{(n)},\mathfrak{S}_{1,j}^{(n)}]$ for any $j = 1, \ldots , \nu^{(n)}$ (Clearly the union of $[\mathfrak{T}_{1,j}^{(n)},\mathfrak{S}_{1,j}^{(n)}]$, $1\le j \le \nu^{(n)}$, does not cover $[\mathfrak{T}_{1}^{(n)},\mathfrak{S}_{1}^{(n)}]$). Then 
\[
\widetilde{T}_1^{\beta,(n)} \leq \sum_{j=1}^{\nu^{(n)}} \widetilde{S}_j^{(n)}.
\]

\begin{remark}
During the time interval $[\mathfrak{T}_1^{(n)},\mathfrak{S}_1^{(n)}]$ it is possible that $\tilde{R}^{(n)}$ gets near to $L^{(n)}$, starting to copy the increments of $L^{(n)}$. But even during these periods, the increments of $\tilde{R}^{(n)}$ continue to be distributed as the increments of an r-path, since it will go in the opposite direction of $L^{(n)}$ as soon as the $(\omega,\theta)$-DNB branches at a position that $L^{(n)}$ is occupying. Besides that, the imposition of the non-crossing rule between $\tilde{R}^{(n)}$ and $L^{(n)}$ is favorable to have the meeting between $\tilde{R}^{(n)}$ and $R^{(n)}$ happening at a distance greater than $n^{\frac{7}{8}}$ from $L^{(n)}$.
\label{tildeRtnearLt}
\end{remark}

Now we will obtain an upper bound on the distribution of $\nu^{(n)}$. For any $1 \leq j \leq \nu^{(n)}$, note that either $|\tilde{R}_{\mathfrak{S}_{1,j}^{(n)}}^{(n)} - L_{\mathfrak{S}_{1,j}^{(n)}}^{(n)}| \le n^{\frac{7}{8}}$ or $|\tilde{R}_{\mathfrak{S}_{1,j}^{(n)}}^{(n)} - L_{\mathfrak{S}_{1,j}^{(n)}}^{(n)}| > n^{\frac{7}{8}}$ and, in the former case, $\mathfrak{S}_{1,j}^{(n)} = \mathfrak{S}_{1}^{(n)}$ and $j=\nu^{(n)}$. Hence, by the independence of increments of the paths in DNB, we claim that $\nu^{(n)}$ is stochastically bounded from above by a geometric random variable with parameter 0.5. To prove it, let $L$, $R$ and $R'$ be respectively an l-path starting at position $0$ at time $0$, an r-path starting at position $0$ at time $0$ and an r-path starting at position $2n^{\frac{7}{8}}$ also at time $0$, and let $\tau$ be the coalescing time between $R$ and $R'$. It is equivalent to saying that $\tau$ is the coalescing time between $R-L$ and $R'-L$, which are positively drifted random walks starting respectively at positions $0$ and $2n^{\frac{7}{8}}$. By stationarity of the environment and symmetry properties, we have that the event
$$
\{ R-L \textrm{ and } R'-L \textrm{ meet each other for the first time at a distance greater than } n^{\frac{7}{8}} \textrm{ from } 0 \}
$$
has a probability greater than or equal to 0.5. This last assertion together with Remark \ref{tildeRtnearLt} proves the previous claim.

Denoting by $\tilde{G}$ a geometric random variable with parameter $0.5$, and since $\widetilde{S}_j^{(n)}$, $j \ge 1$, are independent of $\nu^{(n)}$, we have that 
\begin{eqnarray}\label{eq:indexesI2} 
	P \Big(\sum_{j=1}^{\nu^{(n)}}\widetilde{S}_j^{(n)} > \frac{n^{\gamma+\beta-1}}{\log(n)^2} \Big) &=& \sum_{k=1}^{\infty} P \Big( \sum_{j=1}^{\nu^{(n)}}\widetilde{S}_j^{(n)} > \frac{n^{\gamma+\beta-1}}{\log(n)^2} \Big| \nu^{(n)} = k\Big) P(\nu^{(n)} = k) \nonumber \\
&  & \!\!\!\!\!\!\!\!\!\!\!\!\!\!\!\!\!\!\!\!\!\!\!\!\!
\le \ \sum_{k=1}^{\infty} k \, \max_{j} P \left( \widetilde{S}_j^{(n)} > \frac{n^{\gamma+\beta-1}}{k\log(n)^2} \right) P(\nu^{(n)} = k) \nonumber \\
&  & \!\!\!\!\!\!\!\!\!\!\!\!\!\!\!\!\!\!\!\!\!\!\!\!\!
\le E[\nu^{(n)}] \, \max_j P \left( \widetilde{S}_j^{(n)} > \frac{n^{\gamma+\beta-1}}{\log(n)^3} \right) + \sum_{k=\left\lfloor\log(n)\right\rfloor}^{\infty} k P(\nu^{(n)} \geq k) \nonumber \\
&  & \!\!\!\!\!\!\!\!\!\!\!\!\!\!\!\!\!\!\!\!\!\!\!\!\!
\le E[\tilde{G}] \, \max_j P \left( \widetilde{S}_j^{(n)} > \frac{n^{\gamma+\beta-1}}{\log(n)^3} \right) + \sum_{k=\left\lfloor\log(n)\right\rfloor}^{\infty} k P(\tilde{G} \geq k) \nonumber \\
&  & \!\!\!\!\!\!\!\!\!\!\!\!\!\!\!\!\!\!\!\!\!\!\!\!\!
	 \leq 2 \, \max_j P \left( \widetilde{S}_j^{(n)} > \frac{n^{\gamma+\beta-1}}{\log(n)^3} \right) + 2^{1-\log(n)} \left( \log(n)+1 \right).
\end{eqnarray}
Since
$$
	\lim_{n \to \infty} n^{1-\beta}2^{1-\log(n)} \left( \log(n)+1 \right)  = 0 ,
$$
we do not need to worry about the rightmost term in \eqref{eq:indexesI2} and to conclude we will show that
$$
\max_j P \left(\widetilde{S}_j^{(n)} > \frac{n^{\gamma+\beta-1}}{\log(n)^3} \right) < \frac{1}{2n^{1-\beta}\log(n)},
$$
for large values of $n$. For the sake of simplicity, consider $j=1$. Note that $\widetilde{S}_1^{(n)}$ is a meeting time and then we will use estimates like the one in Lemma \ref{lemma:dnbcoalescencetime} to control the tail of its distribution. The time window $\widetilde{S}_1^{(n)}$ starts as soon as $\tilde{R}^{(n)}$ or $R^{(n)}$ reach a distance greater than $2n^{\frac{7}{8}}$ from $L^{(n)}$, which means that at this moment, the distance between $\tilde{R}^{(n)}$ and $R^{(n)}$ is at most $2n^{\frac{7}{8}}$ plus a random variable $\delta^{(n)}$ that represents the overshoot of $\tilde{R}^{(n)}$ or $R^{(n)}$ when one of them surpasses the distance $2n^{\frac{7}{8}}$ from $L^{(n)}$.
\begin{remark}
During the time window $\widetilde{S}_1^{(n)}$ it is possible that $\tilde{R}^{(n)}$ gets near to $L^{(n)}$, starting to copy the increments of $L^{(n)}$. But as explained in Remark \ref{tildeRtnearLt}, during these periods the increments of $\tilde{R}^{(n)}$ continue to be distributed as the increments of an r-path. So, the conditions of Proposition \ref{Result:coalescencetime} are also valid when we set $Y_l^{(n)} = |R_l^{(n)} - \tilde{R}_l^{(n)}|$ and consider the filtration $\{\tilde{\mathcal{F}}_l\}_{l \geq 0}$. Therefore, we can apply the bound obtained in Lemma \ref{lemma:dnbcoalescencetime} to the time until the first meeting between $\tilde{R}^{(n)}$ and $R^{(n)}$.
\label{remark:tildeRtcoalesce}
\end{remark}
By a rough estimate, the distribution of $\delta^{(n)}$ is bounded by the sum of two random variables, one representing the overshoot distribution of $R^{(n)} - L^{(n)}$ and the other representing the overshoot distribution of $\tilde R^{(n)} - L^{(n)}$. Both $R^{(n)} - L^{(n)}$ and $\tilde R^{(n)} - L^{(n)}$ have independent increments that can be stochastically bounded by geometric random variables, then follows from Lemma 2.6 of  \cite{nrs} that exists $C' < \infty$ such that $\sup_{n\ge 1} E[\delta^{(n)}] \le C'$. So, by Remark \ref{remark:tildeRtcoalesce} (Lemma \ref{lemma:dnbcoalescencetime}) and Chebyshev inequality, we have that, for large values of $n$,
\begin{align}\label{eq:betachoose2} 
	&P\left(\widetilde{S}_1^{(n)} >  \frac{n^{\gamma+\beta-1}}{\log(n)^3} \right) \nonumber\\
	&=P \hspace{-0.11cm} \left( \hspace{-0.11cm} \left.\widetilde{S}_1^{(n)} \hspace{-0.08cm} > \frac{n^{\gamma+\beta-1}}{\log(n)^3} \right|  \delta^{(n)}  <  n^{\frac{7}{8}} \hspace{-0.11cm}  \right) \hspace{-0.08cm}  P \hspace{-0.09cm}  \left( \delta^{(n)}  <  n^{\frac{7}{8}}   \right) + P \hspace{-0.11cm}  \left( \hspace{-0.11cm}  \left.\widetilde{S}_1^{(n)} \hspace{-0.08cm}  > \frac{n^{\gamma+\beta-1}}{\log(n)^3} \right|  \delta^{(n)}  \hspace{-0.06cm} \geq \hspace{-0.06cm} n^{\frac{7}{8}} \hspace{-0.09cm}  \right) \hspace{-0.08cm} P \hspace{-0.08cm} \left( \hspace{-0.06cm} \delta^{(n)} \hspace{-0.06cm}  \geq \hspace{-0.06cm} n^{\frac{7}{8}} \hspace{-0.06cm} \right) \hspace{-0.06cm} \nonumber \\
	&\leq \hspace{-0.08cm} \frac{C_0 \hspace{-0.06cm} \left( \hspace{-0.06cm} 2n^{\frac{7}{8}}+ n^{\frac{7}{8}} \hspace{-0.06cm} \right) \hspace{-0.06cm} \log(n)^{\frac{3}{2}}}{n^{\frac{\gamma+\beta-1}{2}}} + \frac{E[\delta^{(n)}]}{n^{\frac{7}{8}}} \leq \frac{3C_0 n^{\frac{7}{8}}\log(n)^{\frac{3}{2}}}{n^{\frac{\gamma+\beta-1}{2}}} + \frac{C'}{n^{\frac{7}{8}}}. 
\end{align}
To conclude, we just need to choose $\gamma$ and $\beta$ in such a way that the first term on the right hand side of \eqref{eq:betachoose2} becomes smaller than $(n^{\beta-1}\log(n)^{-1})/2$ for large values of $n$ (recalling that $\gamma < 2$ and $\frac{7}{8} < \beta < 1$). It can be achieved for example if we take $\gamma = \frac{47}{24}$ and $\beta = \frac{23}{24}$, completing the proof. \hfill $\square$ 

\vspace{0.2cm}

From the end of the proof of Lemma \ref{proposition:RtilandR} just above, we can state the following claim, which we use later.

\begin{claim}
\label{claim:rtRt}

For any $s>0$, the amount of time where $|R_t^{(n)} - \tilde{R}_t^{(n)}| > n^{\frac{23}{24}}$ in the time interval $[0,sn^2]$ is almost surely $o(n^{\frac{47}{24}})$.
\end{claim}

\medskip

\noindent \textit{Proof of Lemma \ref{proposition:lastterm}.} We have a sum of non-negative terms, thus it is enough to show that for any fixed $t>0$ and $\epsilon > 0$:
$$\limsup_{n\rightarrow\infty} P\Big(\frac{1}{n} \displaystyle{\sum_{j=1}^{tn^2} U_j^{(n)}} > \epsilon\Big) = 0,$$
where $(U_j^{(n)})_{j\in \mathbb{N}}$ was defined in \eqref{Us}.

Set $\Gamma_{n,t}$ as the event that neither $L^{(n)}$ nor $\tilde{R}^{(n)}$ makes a jump of size greater than $\kappa \log(n)/2$ in the time interval $[0,tn^2]$. By Corollary \ref{cor:longjump}, it is enough to prove that
$$\limsup_{n\rightarrow\infty} P\Big(\frac{1}{n} \displaystyle{\sum_{j=1}^{tn^2} U_j^{(n)}} > \epsilon \, ; \, \Gamma_{n,t} \Big) = 0.$$

\vspace{0.2cm}

Note that the distribution of $U_j^{(n)}$ on the event $\{U_j^{(n)} > 0\}$ depends on how far the paths $L_j^{(n)}$ and $\tilde{R}_j^{(n)}$ are from each other before the attempt of crossing. However, when $L_j^{(n)}$ and $\tilde{R}_j^{(n)}$ are evolving independently, we have that $\Delta^{(n)} = \tilde{R}^{(n)} - L^{(n)}$ is a random walk with increments that have a finite absolute third moment, which follows from Remark \ref{rem:moments}. Then, by Lemma 2.6 of \cite{nrs}, there exists $C' = C'(p)$ such that

$$
E[U_j^{(n)} |U_j^{(n)} > 0,L_{j-1}^{(n)},\tilde{R}_{j-1}^{(n)}] \le C'.
$$

\noindent Conditioning on the event that suppresses long jumps, the conditional expectation of $U_j^{(n)}$ only decreases. Thus the previous claims and the non-crossing rule of steps ($\tilde{R}$.1), ($\tilde{R}$.2) and ($\tilde{R}$.3) guarantee that (enlarging $C'$ if necessary)

$$
E[U_j^{(n)} |U_j^{(n)} > 0, \Gamma_{n,t}] \le C'.
$$

\vspace{0.2cm}

Applying Markov's inequality:
\begin{equation}
    P\Big(\frac{1}{n} \displaystyle{\sum_{j=1}^{tn^2} U_j^{(n)}} > \epsilon \, ; \, \Gamma_{n,t} \Big) \leq P\Big(\frac{1}{n}\sum_{j=1}^{tn^2}U_j^{(n)} > \epsilon\Big|\Gamma_{n,t}\Big) \le  \frac{1}{n\epsilon}E\Big(\sum_{j=1}^{tn^2}U_j^{(n)} \Big|\Gamma_{n,t}\Big). 
\label{eq:pU} 
\end{equation}
Now, denote by $N_u^{t,(n)}$ the random variable that counts the number of attempts of crossings between $L_t^{(n)}$ and $\tilde{R}_t^{(n)}$ before time $tn^2$ (so, the only cases where $U_j^{(n)}$ have a positive value). Given $N_u^{t,(n)}$ and $i_1 < ... <i_{N_u^{t,(n)}} \le tn^2$, define the event
$$
J(i_1,...,i_{N_u^{t,(n)}}) = \{ U_{i_1} > 0, ..., U_{i_{N_u^{t,(n)}}} > 0\}. 
$$
Then write the expectation in \eqref{eq:pU} as
$$
E \Big( \sum_{i_1 < ... <i_{N_u^{t,(n)}}} \hspace{-0.3cm} P\Big( J(i_1,...,i_{N_u^{t,(n)}}) \Big| N_u^{t,(n)} , \Gamma_{n,t} \Big)  E\Big( \sum_{j=1}^{tn^2} U_j^{(n)} \Big| J(i_1,...,i_{N_u^{t,(n)}}) , N_u^{t,(n)} , \Gamma_{n,t} \Big) \, \Big| \, \Gamma_{n,t} \Big) \hspace{-0.05cm},
$$
where
$$ 
E\Big( \sum_{j=1}^{tn^2} U_j^{(n)} \Big| J(i_1,...,i_{N_u^{t,(n)}}) , N_u^{t,(n)} , \Gamma_{n,t} \Big)
\le
\sum_{k=1}^{N_u^{t,(n)}}
E\Big( U_{i_k}^{(n)} \Big| U_{i_k}^{(n)} > 0 , N_u^{t,(n)} , \Gamma_{n,t} \Big) \le C'  N_u^{t,(n)} 
$$

Returning to \eqref{eq:pU}, we obtain that
$$
P\Big(\frac{1}{n} \displaystyle{\sum_{j=1}^{tn^2} U_j^{(n)}} > \epsilon \, ; \, \Gamma_{n,t} \Big) \le \frac{C'}{n \epsilon} E(N_u^{t,(n)} | \Gamma_{n,t}).
$$

Therefore we need to show that $E(N_u^{t,(n)}| \Gamma_{n,t})$ is $o(n)$. We are going to prove a stronger result, that $E(N_u^{t,(n)}| \Gamma_{n,t})$ converges to zero as $n$ goes to infinity (or equivalently $E(N_u^{t,(n)})$, since $P(\Gamma_{n,t})$ converges to one as n goes to $\infty$). 

Given $\Gamma_{n,t}$, the process $\Delta^{(n)}$ does not make any jump of size greater than $\kappa \log(n)$ in the interval $[0,tn^2]$. It means that as soon as $\Delta^{(n)}$ becomes smaller than $\kappa \, n^{\frac{1}{4}}\log(n)^2$, it needs to cross towards $0$, one by one, each interval of size $\kappa \log(n)$, until it could reach $0$. So, $\Delta^{(n)}$ has to visit at least $n^{\frac{1}{4}}\log(n)$ positions at its left before it can reach zero.

If $\Delta^{(n)} \le \kappa \, n^{\frac{1}{4}}\log(n)^2$, whenever $\Delta^{(n)}$ jumps to the left, it remains in the same position for some time and, after that, it necessarily jumps to the right. Now we are going to define an auxiliary process $\tilde{\Delta}^{(n)}$ based on $\Delta^{(n)}$ with three aims: 1. discard times where $\Delta^{(n)}$ remains constant because of $\tilde{R}^{(n)} \approx L^{(n)}$. 2. discard times where $\Delta^{(n)} \ge \kappa \, n^{\frac{1}{4}}\log(n)^2$. 3. help us to control crossing attempts between $\tilde{R}^{(n)}$ and $L^{(n)}$. It starts as soon as $\Delta^{(n)}$ becomes smaller than $\kappa \, n^{\frac{1}{4}}\log(n)^2$, at the same position as $\Delta^{(n)}$, and evolves in the following way: 
\begin{itemize}
\item[(i)] We only consider increments of time for $\tilde{\Delta}^{(n)}$ at moments when $\Delta^{(n)}$ can jump, i.e., when $\Delta^{(n)} \le \kappa \, n^{\frac{1}{4}}\log(n)^2$ we do not consider times where $\tilde{R}^{(n)} \approx L^{(n)}$ in step ($\tilde{R}$.3).
\item[(ii)] Whenever $\Delta^{(n)}$ jumps to zero, $\tilde{\Delta}^{(n)}$ also jumps to zero and stops. It does not move anymore. All the following steps should only be considered until $\tilde{\Delta}^{(n)} = 0$. 
\item[(iii)] Whenever $\Delta^{(n)}$ jumps from a position smaller than or equal to $\kappa \, n^{\frac{1}{4}}\log(n)^2$ to a position greater than $\kappa \, n^{\frac{1}{4}}\log(n)^2$, then  $\tilde{\Delta}^{(n)}$ jumps to $\kappa \, n^{\frac{1}{4}}\log(n)^2$ and time freezes for $\tilde{\Delta}^{(n)}$ until (iv). 
\item[(iv)] Except for (ii), i.e. $\Delta^{(n)}$ has not stopped at zero yet, whenever $\Delta^{(n)}$ jumps from a position greater than $\kappa \, n^{\frac{1}{4}}\log(n)^2$ to a position less than or equal to $\kappa \, n^{\frac{1}{4}}\log(n)^2$, then $\tilde{\Delta}^{(n)}$ jumps to the same destination as $\Delta^{(n)}$.
\item[(v)] Except for (ii), if $\Delta^{(n)} \leq \kappa \, n^{\frac{1}{4}}\log(n)^2$, whenever $\Delta^{(n)}$ jumps to the left (its value decreases), the next jump that changes the value of $\Delta^{(n)}$ is necessarily a jump to the right, after a period with $\tilde{R}^{(n)} \approx L^{(n)}$, that occurs when the $(\omega,\theta)$-DNB branches at a position that $L^{(n)}$ is occupying. In this case, $\tilde{\Delta}^{(n)}$ makes a single jump equal to the sum of these two jumps of $\Delta^{(n)}$. Recall that $\Delta^{(n)} \ge 0$ by the non-crossing rule and equally $\tilde{\Delta}^{(n)} \ge 0$. For instance, suppose that time $t$ for $\Delta^{(n)}$ corresponds to time $\tilde{t}$ for $\tilde{\Delta}^{(n)}$ and $\Delta^{(n)}_t = \tilde{\Delta}^{(n)}_{\tilde{t}} = x \leq \kappa \, n^{\frac{1}{4}}\log(n)^2$. If $\Delta^{(n)}_{t+j} = x-5$, $1\le j \le 4$, and $\Delta^{(n)}_{t+5} = x+1$, then we will have $\tilde{\Delta}^{(n)}_{\tilde{t}+1} = x+1$. 
\item[(vi)] Except for (iii), whenever $\Delta^{(n)}$ makes jumps to the right that are not related to a branching of the $(\omega,\theta)$-DNB or stay in the same position without having $\tilde{R}^{(n)} \approx L^{(n)}$, $\tilde{\Delta}^{(n)}$ simply jumps equally to $\Delta^{(n)}$.
\end{itemize}
Note that $L^{(n)}$ and $\tilde{R}^{(n)}$ can only visit integer positions and consequently $\Delta^{(n)}$ is always an integer. So, if any position mentioned in the construction above is non integer, we simply round up the value to the nearest integer and it will make no difference in our arguments. Besides that, by construction, we can verify the following properties:
\begin{itemize}
\item If $\tilde{\Delta}^{(n)}$ starts at a position greater than zero, then the probability that $\Delta^{(n)}$ reaches zero in the time interval $[0,tn^2]$ is bounded above by the probability that $\tilde{\Delta}^{(n)}$ reaches zero in this time interval. 
\item The process $\tilde{\Delta}^{(n)}$  is a Markov process while evolving inside the interval $[0, \kappa \, n^{\frac{1}{4}}\log(n)^2)$. 
\item Considering (v) and (vi), while $\tilde{\Delta}^{(n)}$ is evolving inside the open interval $(0, \kappa \, n^{\frac{1}{4}}\log(n)^2)$, any of its increments is composed of two terms: The first one is a typical increment of the difference between two independent DNB paths, an l-path and an r-path. The second one is zero, if the first is non-negative, otherwise it is equal to twice the absolute value of an increment of an l-path of the DNB given that it branches. Note that the probability that the first increment is negative is bounded from below by $(p(1-p))^2$, which is a lower bound for the probability that $L^{(n)}$ remains constant and $\tilde R^{(n)}$ jumps one unit to the left, thus $\Delta^{(n)}$ also decreases by one unit. Also, the first term has a mean of order $o(1/n)$, but positive. The second term has mean $2/p(2-p)$ (twice the mean of a geometric random variable with parameter $p(2-p)$, see Remark \ref{ramificacao}). Then, we have that $\tilde{\Delta}^{(n)}$ has a positive drift that is bounded from below by $2p(1-p)^2/(2-p)$ for $n \in \mathbb{N}$.
\end{itemize}

\begin{remark}\label{ramificacao}Denoting by $Y^{(n)}$ a random variable with the same distribution of an increment of $L^{(n)}$ and by $B$ the event that the DNB branches at the position of $L^{(n)}$, we have that $P([Y^{(n)}=k] \cap B) = \epsilon_n p^2(1-p)^{(2k-1)}$. Since
\[
P(B) = \sum_{k=1}^{\infty} \epsilon_n p^2(1-p)^{(2k-1)} = \epsilon_n p^2[(1-p)/(1-(1-p)^2)],
\] 
it follows that $P(Y=k|B) = (1-p)^{2(k-1)}[1-(1-p)^2]$ and consequently the increment of $L^{(n)}$ given $B$ is distributed as a geometric random variable with parameter $p(2-p)$.
\end{remark}

By the above properties, when $\tilde{\Delta}^{(n)}$ is at a position greater than or equal to $j \kappa \log(n)$ for $j = 1, \ldots , n^{\frac{1}{4}}\log(n)$, it has probability $p_u^{(n)} > 0$ of never visiting $(j-1) \kappa \log(n)$ (due to translation invariance, the probability does not depend on $j$). Besides that, since $\tilde{\Delta}^{(n)}$ has a positive drift that does not converge to zero as $n$ goes to infinity, there exists $p_u > 0$ such that $p_u^{(n)} \geq p_u$ for any value of $n$. Thus, for $\tilde{\Delta}^{(n)}$ starting at $\kappa n^{\frac{1}{4}}\log(n)^2$, it follows from the Gambler's ruin problem that
$$P \left( \left. \inf_{0 \leq s \leq tn^2}  \tilde{\Delta}_s^{(n)} = 0 \right| \Gamma_{n,t} \right) \leq (1-p_u)^{n^{\frac{1}{4}}\log(n)}.$$

Hence, we can write
$$E(N_u^{t,(n)} | \Gamma_{n,t}) \leq tn^2 P \left( \left. \inf_{0 \leq s \leq tn^2}  \tilde{\Delta}^{(n)} = 0 \right| \Gamma_{n,t} \right) \leq tn^2 (1-p_u)^{n^{\frac{1}{4}}\log(n)},$$
which converges to zero as $n \rightarrow \infty$ for any fixed $t>0$, concluding the proof. \hfill $\square$

\subsection{Part II: Only one pair $(L^{(n)},R^{(n)})_{n \in \mathbb{N}}$ with $R_0^{(n)} < L_0^{(n)}$} 
\label{subsec:partII}

Before we prove the convergence in this case, we will state Lemma \ref{lemma:pathsstayclose} and Lemma \ref{lemma:meetingtime} that will be useful here and in the following scenarios. Both lemmas are proved in Appendix \ref{subsec:conditionIresults}.

\vspace{0.2cm}

\begin{lemma}
\label{lemma:pathsstayclose}
Fix $(x,s) \in \mathbb{Z}^2$ and $\gamma_1 >0$. Let $l^{(n)}_x$ and $\overline{l}^{(n)}_{x}$ be two l-paths evolving in the same environment, starting from $(x,s)$ and $(x+\lfloor n^{\gamma_1} \rfloor ,s)$ respectively.  Denoting by $\tau^{(n)}_x$ the coalescence time between $l^{(n)}_x$ and $\overline{l}^{(n)}_{x}$, we have that the following limit holds: 
\begin{equation*}
\lim_{n \to \infty} P\left(\sup_{s \leq t \leq s + \tau^{(n)}_x} \left| \overline{l}^{(n)}_{x}(t) -  l^{(n)}_x(t) \right| > n^{\gamma} \right) = 0, \, \mbox{ for any } \gamma > \gamma_1\, .
\end{equation*}
This statement also holds if $l^{(n)}_x$ and $\overline{l}^{(n)}_{x}$ are evolving according to independent environments (in this case, we can think of $\tau^{(n)}_x$ as the first meeting time). Besides that, this statement also holds if we replace both l-paths by r-paths.
\end{lemma}

\begin{lemma}
\label{lemma:meetingtime}
Fix $(x,s) \in \mathbb{Z}^2$ and $\gamma_1 >0$. Let $r^{(n)}_x$ be an r-path that starts from $(x,s) \in \mathbb{Z}^2$ and $l^{(n)}_{x}$ be an l-path that starts from $(x+\lfloor n^{\gamma_1} \rfloor,s)$,  both evolving according to the same environment. Now let $\tilde{r}^{(n)}_x$ be an r-path that also starts from the same point $(x,s)$, but evolves in an independent environment until it crosses $l^{(n)}_{x}$ for the first time. After that, $\tilde{r}^{(n)}_x$ starts to evolve in the same environment of $r^{(n)}_x$ and $l^{(n)}_{x}$. Also define $\mathcal{T}_{x,s}^{(n)}$ as the amount of time after $s$ that we have to wait until $r^{(n)}_x$ and $\tilde{r}^{(n)}_x$ meet each other at the right hand side of $l^{(n)}_{x}$ for the first time. Therefore the distribution of $\mathcal{T}^{(n)} = \mathcal{T}_{x,s}^{(n)}$ does not depend on $x$ or $s$, and moreover 
$$
\lim_{n \to \infty} P(\mathcal{T}^{(n)} > 2n^{2\gamma}) = 0, \text{ for any } \gamma > \gamma_1\, .
$$
\end{lemma}

\vspace{0.2cm}

\begin{remark}By Remark \ref{remark:indepcoalesce}, we have that Lemma \ref{lemma:meetingtime} also holds if $\tilde{r}^{(n)}_x$ continue evolving considering an independent environment after it crosses $l^{(n)}_{x}$. 
\end{remark}

Now note that since $R^{(n)}$ has positive drift and $L^{(n)}$ has negative drift, they will cross each other in an almost surely finite time for every $n \geq 1$. Besides that, as verified in Part I, $\Big\{ \left( \frac{1}{n} L_{tn^2}^{(n)},  \frac{1}{n} R_{tn^2}^{(n)}\right)_{t \geq 0} \Big\}_{n \in \mathbb{N}}$ are tight and then, we can take a subsequence that converges weakly to some limiting process $(L_t,R_t)_{t \geq 0}$. By Skorohod's Representation Theorem, we can couple these processes for $n \in \mathbb{N}$ and the limit process such that the convergence holds almost surely. 

Let $\overline{S}^{(n)}$ be the random time when  $R^{(n)}$ and $L^{(n)}$ cross each other. Assuming the coupling, we have that $\frac{1}{n^2}\overline{S}^{(n)} \rightarrow \overline{S}$, $\frac{1}{n} R_{\overline{S}^{(n)}}^{(n)} \rightarrow R_{\overline{S}}$ and $\frac{1}{n} L_{\overline{S}^{(n)}}^{(n)} \rightarrow L_{\overline{S}}$ almost surely as $n$ goes to infinity, where $\overline{S}$ is the crossing time between the limit paths $L$ and $R$. This is a consequence of the uniform convergence of the subsequence of $\Big\{ \left( \frac{1}{n} L_{tn^2}^{(n)},  \frac{1}{n} R_{tn^2}^{(n)}\right)_{t \geq 0} \Big\}_{n \in \mathbb{N}}$ that we got from Skorohood's Representation Theorem. By the interaction rule between $R^{(n)}$ and $L^{(n)}$, we have that once they meet each other, $R^{(n)}$ can not be at the left-hand side of $L^{(n)}$ anymore. 

So, we can divide our proof into two stages: one considering $R_t^{(n)}$ and $L_t^{(n)}$ for $t \in [0,\overline{S}^{(n)})$ and the other when $t \geq \overline{S}^{(n)}$. About this second scenario, we  already have verified in Subsection \ref{subsec:partI} that the convergence holds when $R^{(n)}$ and $L^{(n)}$ start from fixed vertices such that $R_0^{(n)} \geq L_0^{(n)}$. The difference here is that these paths start at a random time $\overline{S}^{(n)}$, but it is a $\mathcal{F}_t$-stopping time and by the Strong Markov Property of $(L^{(n)},R^{(n)})$ the weak convergence holds in time interval $[\overline{S},\infty)$.

For $t < \overline{S}^{(n)}$, first note that $\overline{S}^{(n)}$ is almost surely finite, as a consequence of  Corollary \ref{corollary:dnbcrossingtime}. Indeed, the convergence that we mentioned above implies that $\overline{S}^{(n)}$ is of order $n^2$, when $R_t^{(n)}$ and $L_t^{(n)}$ start at a distance of order $n$ from each other. Besides that, given any initial distance between $R^{(n)}$ and $L^{(n)}$, Corollary \ref{corollary:dnbcrossingtime} gives us an estimate on the tail of the distribution of $\overline{S}^{(n)}$. Together with Lemma \ref{lemma:longjump}, we have that $R_t^{(n)}$ and $L_t^{(n)}$ will not make jumps of size greater than $n^{\frac{3}{4}}$ in the time interval $[0,\overline{S}^{(n)})$ with a probability that converges to one as $n \rightarrow \infty$. 

Now denote by $\overline{T}^{(n)}$ the first time that $R^{(n)}$ and $L^{(n)}$ reach a distance smaller than $n^{\frac{3}{4}}$ from each other and let us couple $R^{(n)}$ with another path $\overline{R}^{(n)}$. The path $\overline{R}^{(n)}$ is equal to $R^{(n)}$ until $\overline{T}^{(n)}$. Note that, at this moment, and conditioning on the event that $R^{(n)}$ and $L^{(n)}$ do not make jumps of size greater than $\kappa \, log(n)/2$, $R^{(n)}$ is still on the left-hand side of $L^{(n)}$. As in Section \ref{subsec:partI}, let $\tilde{\omega}=(\tilde{\omega}(z))_{z\in\mathbb{Z}^2}$ and $\tilde{\theta}=(\tilde{\theta}(z))_{z\in\mathbb{Z}^2}$ be independent families of random variables identically distributed as $\omega$ and $\theta$ respectively, and also independent of them. After $\overline{T}^{(n)}$, $\overline{R}^{(n)}$ begins to evolve according to the $(\tilde{\omega}(z),\tilde{\theta}(z))$-environment, until $\overline{R}^{(n)}$ meets $R^{(n)}$ at the right-hand side of $L^{(n)}$ for the first time. At this time, $\overline{R}^{(n)}$ returns to be equal to $R^{(n)}$. Denote by $\overline{S}_c^{(n)}$ the first time where $\overline{R}^{(n)}$ is at the right-hand side of $L^{(n)}$.

\begin{remark} At this point it is natural that the reader asks himself why we are introducing $\overline{R}^{(n)}$. Indeed we introduce $\overline{R}^{(n)}$ to impose an independence between $\overline{R}^{(n)}$ and $L^{(n)}$ in a time interval $[\overline{T}^{(n)} , \overline{S}^{(n)} ]$ which is negligible under diffusive scaling. To see that it is negligible, we just have to start another r-path from the position of $L^{(n)}$ at time $\overline{T}^{(n)}$ and use coalescing time estimates for this r-path and $\overline{R}^{(n)}$. The coalescence occurs after time $\overline{S}^{(n)}$, and after it we only need to use the Markov property and the result from Section \ref{subsec:partI}. Thus we do not need to introduce $\overline{R}^{(n)}$, and we would have a slightly shorter proof. However in the next section we use a more complex construction based on the definition $\overline{R}^{(n)}$ to couple multiple r-paths and l-paths to independent versions with reflection between left-right pairs.     
\end{remark}

Since $\Big\{ \left( \frac{1}{n} \overline{R}^{(n)}_{tn^2},  \frac{1}{n} R_{tn^2}^{(n)}\right)_{t \geq 0} \Big\}_{n \in \mathbb{N}}$ is tight, we can take a subsequence that converges weakly to some limiting process $\displaystyle{\left(\overline{R}_t,R_t\right)_{t \geq 0}}$. By Skorohod's Representation Theorem, we can couple these processes for $n \in \mathbb{N}$ and the limiting process such that the convergence holds almost surely. Also, denote by $\overline{S}_c$ the time where $\overline{R}$ crosses $L$.

Again conditioning on the event that $R^{(n)}$ and $L^{(n)}$ do not make jumps of size greater than $\kappa \, log(n)$, we have that $\overline{R}^{(n)}$ and $L^{(n)}$ are independent paths until $\overline{R}^{(n)}$ meets $R^{(n)}$ at the right-hand side of $L^{(n)}$. Then, by Lemma \ref{lemma:individualconvergence}, $(\overline{R}^{(n)}_t,{L}_t^{(n)})_{t<\overline{S}_c^{(n)}}$ converges weakly to $(B_t^r,B_t^l)_{t < \overline{S}_c}$ under diffusive scaling, where $B_t^r$ and $B_t^l$ are independent Brownian motions with diffusion coefficient $\lambda_p$ and drifts $b_p$ and $-b_p$ respectively.

To conclude, we have to show that $(R^{(n)}$,$L^{(n)})$ and $(\overline{R}^{(n)},L^{(n)})$ converge weakly under diffusive scale to the same limit. Let $\mathcal{T}^{(n)}$ be as in the statement of Lemma \ref{lemma:meetingtime} for the triple $(R^{(n)},\overline{R}^{(n)},L^{(n)})$, which here is equivalent to the length of time where $R^{(n)}$ and $\overline{R}^{(n)}$ are different. Then
$$
P\Big( \sup_{\overline{T}^{(n)} \leq t \leq \overline{T}^{(n)} + \mathcal{T}^{(n)}} \big|R^{(n)}_t - \overline{R}^{(n)}_t\big| > n^{\frac{5}{6}} \Big)
$$
is bounded from above by
\begin{align}\label{eq:partII}
	& P\Big(\mathcal{T}^{(n)} > 2n^{\frac{8}{5}}\Big) + P \Big( \sup_{\overline{T}^{(n)} \leq t \leq \overline{T}^{(n)} + 2n^{\frac{8}{5}}} \big|R^{(n)}_t - \overline{R}^{(n)}_t\big| > n^{\frac{5}{6}} \Big) \nonumber\\
	& \leq P\left(\mathcal{T}^{(n)} > 2n^{\frac{8}{5}}\right) + P \Big( \sup_{\overline{T}^{(n)} \leq t \leq \overline{T}^{(n)} + 2n^{\frac{8}{5}}} \big( \big|R^{(n)}_t - R^{(n)}_{\overline{T}^{(n)}}\big| + \big|\overline{R}^{(n)}_t - \overline{R}^{(n)}_{\overline{T}^{(n)}}\big| \big) > n^{\frac{5}{6}} \Big).
\end{align}
The first term in \eqref{eq:partII} converges to zero as $n$ goes to infinity by Lemma \ref{lemma:meetingtime} (choosing $\gamma_1 = \frac{3}{4}$ and $\gamma = \frac{4}{5}$).  About the second term in \eqref{eq:partII}, by Lemma \ref{lemma:individualconvergence} we have that both $R^{(n)}$ and $\overline{R}^{(n)}$ converge under diffusive scaling to Brownian motions with drift $b_p$. Since $5/6 > 4/5$, by the Strong Markov property, this second term also converges to zero as $n$ goes to infinity, which let us conclude that $(R^{(n)}$,$L^{(n)})$ and $(\overline{R}^{(n)},L^{(n)})$ converge weakly under diffusive scaling to the same limit.

Finally, the proofs presented in Sections \ref{subsec:partI} and \ref{subsec:partII} together give us the proof of Proposition \ref{lemma:onepairconverge} \hfill $\square$ 

\vspace{0.2cm}

\subsection{Part III: $k$ l-paths and $\tilde{k}$ r-paths}
\label{subsec:partV}

Now we will assume that we have $k$ sequences of l-paths of the DNB, $l_i^{(n)}$ starting from $z_i^{(n)} = (x_i^{(n)},0)$ with $n^{-1} z_i^{(n)}$ converging to $z_i = (x_i,0)$ for $1\le i \le k$, and $\tilde{k}$ sequences of r-paths, $r^{(n)}_{\tilde{z}_j}$ starting from $\tilde{z}_j^{(n)} = (\tilde{x}_j^{(n)},0)$, with $n^{-1}\tilde{z}_j^{(n)}$ converging to $\tilde{z}_j = (\tilde{x}_j,0)$ for $j = 1, \ldots, \tilde{k}$. Based on the non-crossing property, without loss of generality we assume the worst scenery, which is all r-paths starting at the left of all l-paths, i.e, $\tilde{x}_{\tilde{k}}^{(n)} < \ldots < \tilde{x}_{1}^{(n)} < x_1^{(n)} < \ldots < x_k^{(n)}$.

The main strategy is to couple our collection of l-paths and r-paths with a collection of paths that are conditionally independent given the non-occurrence of long jumps. This coupling should have the property that the time windows where the collections differ are negligible under diffusive scaling. So we start considering slight modifications of paths $(l_i^{(n)})_{i=1}^{k-1}$. The modified paths will be denoted by $(\overline{l}_i^{(n)})_{i=1}^{k-1}$ and they are constructed in the following way: $\overline{l}_i^{(n)}$ is equal to $l_i^{(n)}$ until it reaches a distance smaller than $n^{\frac{3}{4}}$ from another l-path at its right-hand side. When it happens, $\overline{l}_i^{(n)}$  begins to evolve according to a different environment $(\omega_i(z),\theta_i(z))_{z \in \mathbb{Z}^2}$ (which is an independent copy of $(\omega,\theta)$) until it attempts to cross another l-path, at this moment the paths will coalesce at the position of the l-path with the greater index. Note that in a system with independent paths until coalescence, they do not coalesce when they cross each other (coalescence only occurs when the paths meet each other at the same position). This is not the case here, however both the systems will have the same asymptotic behavior due to the  estimates on coalescence times. Additionally, we set $\overline{l}_k^{(n)} = l_k^{(n)}$ and suppose that all the environments involved in the construction are mutually independent.

We will also replace r-paths with modified versions denoted by $(\overline{r}_j^{(n)})_{j=1}^{\tilde{k}}$, which we will construct in the following. Let $(\overline{\omega}_j(z),\overline{\theta}_j(z))_{z\in \mathbb{Z}^2}$, for $j=1,...,\tilde{k}$, be copies of $(\omega,\theta)$, still supposing that all the environments involved in the construction are mutually independent. Now define
$$
\overline{\mathcal{F}}_t = \sigma\big( ({\omega}(z),{\theta}(z)), \ ({\omega_i}(z),{\theta_i}(z))_{i=1}^{k-1}, \ (\overline{\omega}_j(z),\overline{\theta}_j(z))_{j=1}^{\tilde{k}}, \ z=(z_1,z_2), \ z_1 \in \mathbb{Z}, \ z_2 \leq t \big),
$$ 
for $t\ge 0$, and consider the $\overline{\mathcal{F}}_t$-stopping times
$$
\overline{T}_{i,j}^{(n)} = \inf \big\{ t\ge 0 : \,  (\overline{l}_i^{(n)}(t) - r_j^{(n)}(t)) \le n^{\frac{3}{4}} \big\} , \ i=1,...,k, \ \mbox{and} \ j=1,...,\tilde{k},
$$  

i.e. $\overline{T}_{i,j}^{(n)}$ is the first time that $r_j^{(n)}$ reaches a distance smaller than $n^{\frac{3}{4}}$ to the left of $\overline{l}_i^{(n)}$. Note that for any fixed $n \in \mathbb{N}$ and $i = 1,2, \ldots, k$, by the assumption in the order of the starting points, $\overline{T}_{i,1}^{(n)} \le \overline{T}_{i,2}^{(n)} \le ... \le \overline{T}_{i,\tilde{k}}^{(n)}$. Also, for any fixed $n \in \mathbb{N}$ and $j = 1,2, \ldots, \tilde{k}$,  $\overline{T}_{1,j}^{(n)} \le \overline{T}_{2,j}^{(n)} \le ... \le \overline{T}_{k,j}^{(n)}$. Some of these stopping times can be equal, which happens, for example, if some l-paths coalesce before they get close to the same r-path.\

Each of the paths $\overline{r}_1^{(n)}, \ldots, \overline{r}_{\tilde{k}}^{(n)}$ is built according to the following steps. 
\begin{itemize}
\item[(i)] Each $\overline{r}_j^{(n)}$ are equal to $r_j^{(n)}$ until it reaches time 
$$
\varrho = \min_{1\le i\le k} \overline{T}_{i,1}^{(n)},
$$
if $j=1$, or 
$$
\varrho = \min_{1\le j'\le j-1} \overline{\mathfrak{T}}_{j,j'} \wedge \min_{1\le i\le k} \overline{T}_{i,j}^{(n)}
$$

if $2\le j \le \tilde{k}$, where
$$
\overline{\mathfrak{T}}_{j,j'} = \inf \big\{ t\ge 0 : \,  (\overline{r}_j^{(n)}(t) - r_{j'}^{(n)}(t)) \le n^{\frac{3}{4}} \big\}, \ \, 1 \leq j' \leq j-1.
$$
Thus $\min_{1\le j'\le j-1} \overline{\mathfrak{T}}_{j,j'}$ is the $\overline{\mathcal{F}}_t$-stopping time where $\overline{r}_j^{(n)}$ reaches a distance smaller or equal to $n^{\frac{3}{4}}$ from another r-path initially at its left. At time $\varrho$, $\overline{r}_j^{(n)}$ begins to evolve according  to the $(\overline{\omega}_j,\overline{\theta}_j)$-environment. If $\varrho = \min_{1\le j'\le j-1} \overline{\mathfrak{T}}_{j,j'}$ go to step (ii), otherwise go to step (iii).
\item[(ii)] After time $\varrho$, when either $\overline{r}_j^{(n)}$ meets or crosses $\overline{r}_{j'}^{(n)}$, for $1\le j'\le j-1$, for the first time, they coalesce at the position of $\overline{r}_{j'}^{(n)}$ and continue to evolve according to the environment that $\overline{r}_{j'}^{(n)}$ is using.
\item[(iii)] Suppose that $\varrho = \overline{T}_{i,j}^{(n)}$. After time $\varrho$,  $\overline{r}_j^{(n)}$ evolves according to the  $(\overline{\omega}_j,\overline{\theta}_j)$-environment until it reaches $\overline{S}_{i,j}^{(n)}$, where
$$
\overline{S}_{i,j}^{(n)} = \inf \big\{ t\ge 0 : \, \overline{r}_j^{(n)}(t) \ge \overline{l}_i^{(n)}(t) \big\} , \ i=1,...,k, \ \mbox{and} \ j=1,...,\tilde{k},
$$
i.e. $\overline{S}_{i,j}^{(n)}$ is the first time that either $\overline{r}_j^{(n)}$ meets or crosses $\overline{l}_i^{(n)}$. Then go to step (iv). If $\varrho = \overline{S}_{i,j}^{(n)}$, then go immediately to step (iv). Note that for any fixed $n \in \mathbb{N}$ and $i = 1,2, \ldots, k$, by the assumption in the order of the starting points, $\overline{S}_{i,1}^{(n)} \le \overline{S}_{i,2}^{(n)} \le ... \le \overline{S}_{i,\tilde{k}}^{(n)}$. Also, for any fixed $n \in \mathbb{N}$ and $j = 1,2, \ldots, \tilde{k}$, $\overline{S}_{1,j}^{(n)} \le \overline{S}_{2,j}^{(n)} \le ... \le \overline{S}_{k,j}^{(n)}$. Moreover, $\overline{T}_{i,j}^{(n)} \le \overline{S}_{i,j}^{(n)}$, for any pair $(i,j)$ with $i=1,...,k$ and $j=1, \ldots, \tilde{k}$.
\item[(iv)] From time $\overline{S}_{i,j}^{(n)}$, $\overline{r}_j^{(n)}$ begins to evolve according to the same environment that $\overline{l}_i^{(n)}$ is using until one of the three following possibilities happens:
\begin{enumerate}
\item Another r-path $\overline{r}_{j'}^{(n)}$ crosses $\overline{l}_i$ and we have $\overline{l}_i^{(n)} < \overline{r}_{j'}^{(n)} < \overline{r}_j^{(n)}$ with $\overline{r}_j^{(n)} - \overline{r}_{j'}^{(n)} \le n^{\frac{3}{4}}$. In this scenario, $\overline{r}_{j'}^{(n)}$ begins to evolve according to environment that $\overline{l}_i^{(n)}$ is considering and $\overline{r}_j^{(n)}$ automatically begin to evolve according to the $(\overline{\omega}_j,\overline{\theta}_j)$-environment, because $\overline{r}_j^{(n)}$ enters the situation described in step (ii).
\item $\overline{r}_j^{(n)}$ meets $r_j^{(n)}$ while is evolving according to the $(\omega,\theta)$-environment. In this case we return to have $\overline{r}_j^{(n)}$ equal to $r_j^{(n)}$. 
\item $\overline{r}_j^{(n)}$ reaches a distance smaller than $n^{\frac{3}{4}}$ to the left of another l-path $\overline{l}_m^{(n)}, m\ge i+1$ and is at a distance greater than $\frac{n^{\frac{3}{4}}}{2}$ from $\overline{l}_i^{(n)}$. In this case $\overline{r}_j^{(n)}$ begins to evolve according to the $(\overline{\omega}_j,\overline{\theta}_j)$-environment. From here, if $\overline{r}_j^{(n)}$ returns to a distance smaller than $\frac{n^{\frac{3}{4}}}{2}$ from the nearest l-path $\overline{l}_i^{(n)}$ at its left, go to step (v).
\end{enumerate}
\item[(v)] Whenever $\overline{r}_j^{(n)}$ returns to a distance smaller than $\frac{n^{\frac{3}{4}}}{2}$ from the nearest l-path $\overline{l}_i^{(n)}$ at its left and there are not any other r-path $\overline{r}_l^{(n)}$ between $\overline{r}_j^{(n)}$ and $\overline{l}_i^{(n)}$, we have that $\overline{r}_j^{(n)}$ will return to evolve according to the same environment that $\overline{l}_i^{(n)}$ is using, like in the beginning of step (iv), if they were not already in the same environment.
\end{itemize}

\begin{remark}
In the construction of the auxiliary paths $\overline{l}_i^{(n)}$ and $\overline{r}_j^{(n)}$ for $i = 1,...,k$ and $j = 1,...,\tilde{k}$ is fundamental to assure that any $\overline{r}_j^{(n)}$ will never cross any $\overline{l}_i^{(n)}$ from right to left. Note that the steps (iv) and (v) in the construction above, together with the imposition that the coalescence in step (ii) occurs in the position of the $r$-path that is occupying the rightmost position, are there to guarantee that this kind of crossing will never happen if these paths do not make jumps of size equal to $\frac{n^{\frac{3}{4}}}{4}$ or greater. We will suppress these long jumps later.
\label{rem:partIII}
\end{remark}

\begin{figure}[!ht]
\centering
\includegraphics[scale=0.54]{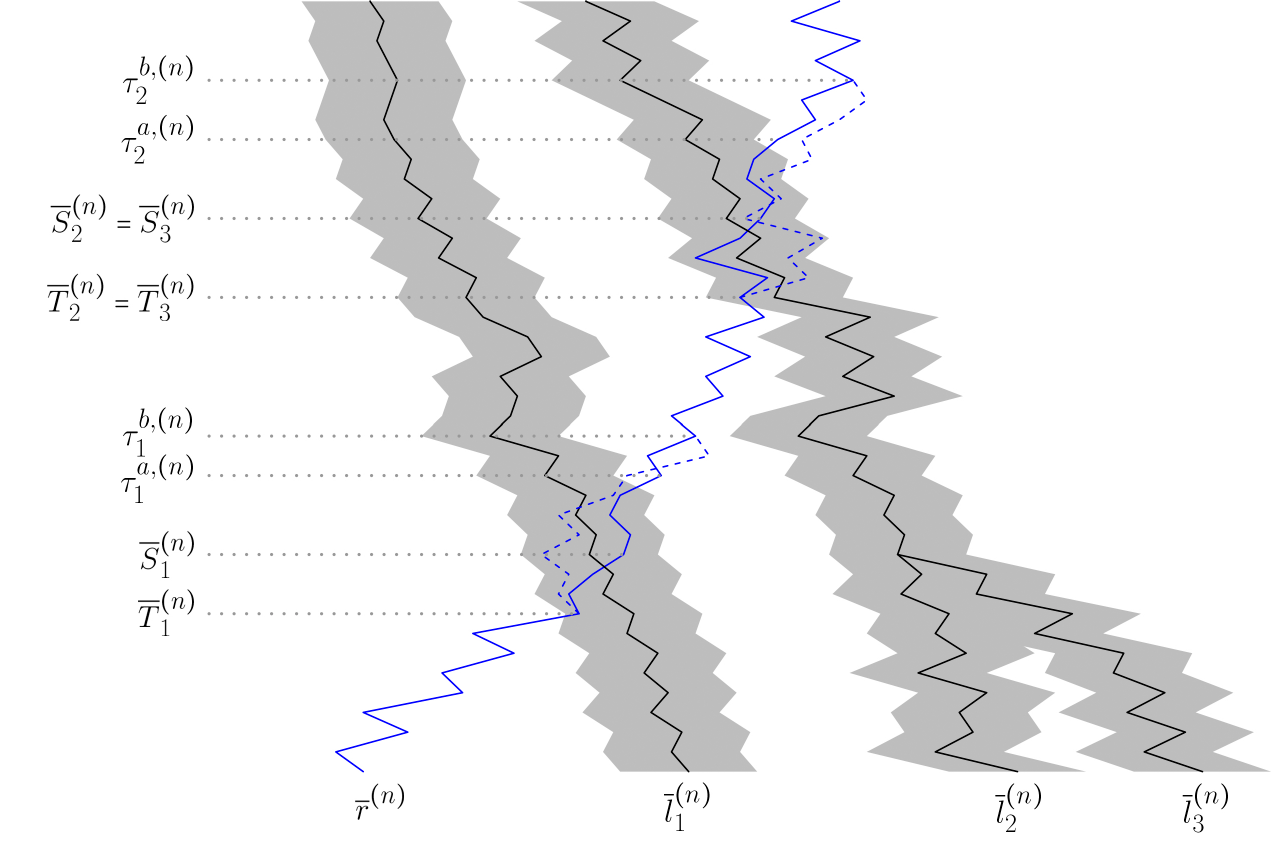}
\caption{Example of the evolution of $\overline{r}^{(n)}$, $\overline{l}^{(n)}_1, \overline{l}^{(n)}_2, \overline{l}^{(n)}_3$. The gray areas indicate the regions near to any $\overline{l}_j^{(n)}$, $j=1,2,3$ such that if $\overline{r}^{(n)}$ enters there, the environment that $\overline{r}^{(n)}$ is considering can change. The dashed lines indicate the evolution of $r^{(n)}$ when it is not equal to $\overline{r}^{(n)}$. 
}
\label{fig_partIV}
\end{figure}

To illustrate an example of this construction, we have Figure \ref{fig_partIV} considering only one alternative r-path $\overline{r}^{(n)}$ and three alternative l-paths $\overline{l}^{(n)}_i$ for $i=1,2,3$. Note that when $\overline{r}^{(n)}$ hits the stopping time $\overline{T}_1^{(n)} = \overline{T}_{1,1}^{(n)}$, it begins to evolve according to the $(\overline{\omega},\overline{\theta})$-environment (independently of $r^{(n)}$) and this event initiates step (iii). At time $\overline{S}_1^{(n)} = \overline{S}_{1,1}^{(n)}$, $\overline{r}^{(n)}$ crosses $\overline{l}_1^{(n)}$ and starts to evolve considering the same environment of $\overline{l}_1^{(n)}$ (which at this point, is the $(\omega,\theta)$-environment, since $\overline{l}_1^{(n)}$ is far from any other $\overline{l}_i^{(n)}$), initiating the step (iv). After that, we will wait for either Scenario 2 or Scenario 3 of step (iv) to happen. At time $\tau_1^{b,(n)}$, $\overline{r}^{(n)}$ meets $r^{(n)}$ in a distance greater than $n^{\frac{3}{4}}$ from any l-path at their right-hand side, so $r^{(n)}$ and $\overline{r}^{(n)}$ are evolving considering the same environment and then $\overline{r}^{(n)}$ returns to be equal to $r^{(n)}$ (it means that the event described in Scenario 2 of step (iv) happened first). At time $\overline{T}_2^{(n)} = \overline{T}_3^{(n)} = \overline{T}_{2,1}^{(n)} = \overline{T}_{3,1}^{(n)}$, $\overline{r}^{(n)}$ reaches a distance smaller than $n^{\frac{3}{4}}$ from $\overline{l}_2^{(n)}$ and $\overline{l}_3^{(n)}$, which already coalesced at this point, returning to the step (iii). From now on the example repeat the same behavior that we described.

From $(\overline{l}_1^{(n)}\hspace{-0.05cm}, \ldots, \overline{l}_k^{(n)}\hspace{-0.05cm}, \overline{r}_1^{(n)}\hspace{-0.05cm}, \ldots, \overline{r}_{\tilde{k}}^{(n)})$, we will construct the process $(\check{l}_1^{(n)}\hspace{-0.05cm}, \hspace{-0.05cm}\ldots, \check{l}_k^{(n)}\hspace{-0.05cm}, \check{r}_1^{(n)}\hspace{-0.05cm}, \ldots, \check{r}_{\tilde{k}}^{(n)})$. They will evolve equally, except that the paths $(\check{l}_1^{(n)}, \ldots, \check{l}_k^{(n)}, \check{r}_1^{(n)}, \ldots, \check{r}_{\tilde{k}}^{(n)})$ can not make jumps of size greater than $n^{\frac{3}{4}}/4$. So both processes will be equal until the first time when one path among $(\overline{l}_1^{(n)}, \ldots, \overline{l}_k^{(n)}, \overline{r}_1^{(n)}, \ldots, \overline{r}_{\tilde{k}}^{(n)})$ makes a jump of size greater than $n^{\frac{3}{4}}/4$, if it happens. At this moment, the associated path in $(\check{l}_1^{(n)}, \ldots, \check{l}_k^{(n)}, \check{r}_1^{(n)}, \ldots, \check{r}_{\tilde{k}}^{(n)})$ will jump only $n^{\frac{3}{4}}/4$ and both process will no longer be equal, despite that, both processes continue to evolve according to the steps described for $(\overline{l}_1^{(n)}, \ldots, \overline{l}_k^{(n)}, \overline{r}_1^{(n)}, \ldots, \overline{r}_{\tilde{k}}^{(n)})$. Now from $(\check{l}_1^{(n)}, \ldots, \check{l}_k^{(n)}, \check{r}_1^{(n)}, \ldots, \check{r}_{\tilde{k}}^{(n)})$ we define, $\check{T}_{i,j}^{(n)}$ and $\check{S}_{i,j}^{(n)}$ in the same way that we defined, $\overline{T}_{i,j}^{(n)}$ and $\overline{S}_{i,j}^{(n)}$ from $(\overline{l}_1^{(n)}, \ldots, \overline{l}_k^{(n)}, \overline{r}_1^{(n)}, \ldots, \overline{r}_{\tilde{k}}^{(n)})$, which are also $\overline{\mathcal{F}}_t$-stopping times.

Fix $s > 0$ and let us define the following events:
\begin{align*}
	\mathcal{A}_{n,s} &= \Big\{ \max_{1 \leq i \leq k} \sup_{0 \leq t \leq sn^2} \left|l_{i}^{(n)}(t) - \overline{l}_{i}^{(n)}(t)\right| \leq n^{\frac{7}{8}} \ \mbox{and} \max_{1 \leq j \leq \tilde{k}} \sup_{0 \leq t \leq sn^2} \left|r^{(n)}_j(t) - \overline{r}^{(n)}_j(t)\right| \leq n^{\frac{14}{15}}\Big\}\, \\
	\mathcal{B}_{n,s} & = \bigcap_{i=1}^{k} \bigcap_{j=1}^{\tilde{k}} \Big\{\overline{l}_{i}^{(n)} \text{ and } \overline{r}_{j}^{(n)} \text{ do not make any jump of size greater than }\frac{n^\frac{3}{4}}{4} \text{in } [0,sn^2] \Big\}
\end{align*}
Our approach here is similar to the one used in the proof of \cite[Theorem 8]{cfd}. We show that conditioned to $\mathcal{A}_{n,s}\cap\mathcal{B}_{n,s}$, $(l_1^{(n)}, \ldots, l_k^{(n)}, r_1^{(n)}, \ldots, r_{\tilde{k}}^{(n)})$ and $(\check{l}_1^{(n)}, \ldots, \check{l}_k^{(n)},  \check{r}_1^{(n)},  \ldots, \check{l}_{\tilde{k}}^{(n)})$  have the same weak limit under diffusive scaling and that $P((\mathcal{A}_{n,s}\cap\mathcal{B}_{n,s})^c)$ goes to zero as $n$ goes to infinity, which will allow us to replace the original paths by $(\check{l}_1^{(n)}, \ldots, \check{l}_k^{(n)}, \check{r}_1^{(n)},  \ldots, \check{l}_{\tilde{k}}^{(n)})$  in our proof.

Let us begin with $\mathcal{B}_{n,s}^c$. Since the probability that $\overline{l}_i^{(n)}$ makes a jump of size greater than $n^{\frac{3}{4}}/4$ does not depend on $i$ and is equal to the probability of $\overline{r}_j^{(n)}$ makes a jump of size greater than $n^{\frac{3}{4}}/4$ for any $j$, we can apply Lemma \ref{lemma:longjump} and have that
\begin{align}\label{eq:eventB}
	P(\mathcal{B}_{n,s}^c) &\leq (k+\tilde{k})P\Big(\overline{r}_1^{(n)} \mbox{ makes a jump of size greater than } n^\frac{3}{4}/4 \mbox{ in } sn^2 \mbox{ attempts}\Big) \nonumber\\
	& \leq 2(k+\tilde{k})sn^2\exp\left\{-\frac{c}{4}\left(n^{\frac{3}{4}}-4\right)\right\}, 
\end{align}
which converges to zero as n goes to infinity.

To deal with  $\mathcal{A}_{n,s}^c$, write $\mathcal{A}_{n,s}^c = \overline{\mathcal{A}}_{n,s}^c  \cup \ddot{\mathcal{A}}_{n,s}^c$ where:
\begin{align*}
	\overline{\mathcal{A}}_{n,s}^c &= \Big\{ \max_{1 \leq i \leq k} \sup_{0 \leq t \leq sn^2} \left|l_{i}^{(n)}(t) - \overline{l}_{i}^{(n)}(t)\right| > n^{\frac{7}{8}} \Big\}\, , \\
	\ddot{\mathcal{A}}_{n,s}^c &= \Big\{ \max_{1 \leq j \leq \tilde{k}} \sup_{0 \leq t \leq sn^2} \left|r^{(n)}_j(t) - \overline{r}_j^{(n)}(t)\right| > n^{\frac{14}{15}}\Big\}.
\end{align*}
Concerning $\overline{\mathcal{A}}_{n,s}^c$, write $P\big(\overline{\mathcal{A}}_{n,s}^c\big)$ as
\begin{equation}
 P\Big(\bigcup_{i=1}^k \Big\{ \sup_{0 \leq t \leq sn^2} \left|l_{i}^{(n)}(t) - \overline{l}_{i}^{(n)}(t)\right| > n^{\frac{7}{8}}\Big\}\hspace{-0.05cm}\Big) \hspace{-0.05cm} \leq k P\hspace{-0.05cm}\Big(\sup_{0 \leq t \leq sn^2} \Big|l_{1}^{(n)}(t) - \overline{l}_{1}^{(n)}(t)\Big| > n^{\frac{7}{8}}\Big)\hspace{-0.05cm}.
\label{eq:eventAbar}
\end{equation}
Now applying Lemma \ref{lemma:pathsstayclose} with $\gamma_1 = \frac{3}{4}$ and $\gamma = \frac{7}{8}$, we have that the right hand side of the expression (\ref{eq:eventAbar}) converges to zero as $n$ goes to infinity, since $l_{1}^{(n)}$ and $\overline{l}_{1}^{(n)}$ will be different only in a time window that starts when $\overline{l}_{1}^{(n)}$ reaches a distance smaller than $n^{\frac{3}{4}}$ from another l-path $\overline{l}_{2}^{(n)}$ and ends when $\overline{l}_{1}^{(n)}$ and $\overline{l}_{2}^{(n)}$ coalesce. For $\ddot{\mathcal{A}}_{n,s}^c$, note that $\overline{r}_j^{(n)}(t)$ can begin to evolve in a different environment for two reasons: Either because $\overline{r}_j^{(n)}(t)$ reaches a distance smaller than $n^\frac{3}{4}$ from another $\overline{r}_l^{(n)}(t)$ or because $\overline{r}_j^{(n)}(t)$ reaches a distance smaller than $n^\frac{3}{4}$ from an l-path $\overline{l}_i^{(n)}$ at its right. Denote by $\Xi_{j}^r$ and $\Xi_{j}^l$ the set of times where $\overline{r}_j^{(n)}(t)$ is evolving in a different environment due to the first and the second reasons, respectively (note that $\Xi_{j}^r$ and $\Xi_{j}^l$ can have some indexes in common). Then we have that
\[
P(\ddot{\mathcal{A}}_{n,s}^c) \leq P\Big(\max_{1 \leq j \leq \tilde{k}} \sup_{t \in \Xi_{j}^r \cup \Xi_{j}^l} \left|r^{(n)}_j(t) - \overline{r}_j^{(n)}(t)\right| > n^{\frac{14}{15}}\Big)
\]
\begin{equation}
\leq P\Big( \bigcup_{j=1}^{\tilde{k}}\Big\{ \sup_{t \in \Xi_{j}^r} \left|r^{(n)}_j(t) - \overline{r}_j^{(n)}(t)\right| > n^{\frac{14}{15}}\Big\}\Big) + P\Big( \bigcup_{j=1}^{\tilde{k}} \Big\{ \sup_{t \in \Xi_{j}^l} \left|r^{(n)}_j(t) - \overline{r}_j^{(n)}(t)\right| > n^{\frac{14}{15}}\Big\}\Big).
\label{eq:Xi}
\end{equation}
The leftmost term in \eqref{eq:Xi} converges to zero as $n$ goes to infinity for the same reason that \eqref{eq:eventAbar} converges to zero as $n$ goes to infinity. About the rightmost term in \eqref{eq:Xi}, note that $r_j^{(n)}(t)$ and $\overline{r}_j^{(n)}(t)$ are different due to the second reason only during a maximum of $k$ time windows that start at the stopping times $\overline{T}_{i,j}^{(n)}$ and have some length $\mathcal{T}_{i,j}^{(n)}$, which has the same distribution of $\mathcal{T}^{(n)}$ defined in Lemma \ref{lemma:meetingtime}. Then, applying this lemma, we can bound the rightmost term in \eqref{eq:Xi} from above by
$$
P \Big( \bigcup_{j=1}^{\tilde{k}} \bigcup_{i=1}^k \Big\{ \sup_{\overline{T}_{i,j}^{(n)} \leq t \leq \overline{T}_{i,j}^{(n)} + \mathcal{T}_{i,j}^{(n)}} \left|r_j^{(n)}(t) - \overline{r}_j^{(n)}(t)\right| > n^{\frac{14}{15}}\Big\} \Big)
$$
which consequently is bounded above by 
\begin{align}\label{eq:eventAdots}
	& \tilde{k}k P \Big( \sup_{\overline{T}_{1,1}^{(n)} \leq t \leq \overline{T}_{1,1}^{(n)} + \mathcal{T}_{1,1}^{(n)}} \left|r_1^{(n)}(t) - \overline{r}_1^{(n)}(t)\right| > n^{\frac{14}{15}} \Big) \nonumber\\
	&\leq \tilde{k}k P\big(\mathcal{T}_{1,1}^{(n)} > 2n^{\frac{13}{7}}\big) + \tilde{k}k P \Big( \sup_{\overline{T}_{1,1}^{(n)} \leq t \leq \overline{T}_{1,1}^{(n)} + 2n^{\frac{13}{7}}} \left|r_1^{(n)}(t) - \overline{r}_1^{(n)}(t)\right| > n^{\frac{14}{15}} \Big).
\end{align}
The first term in \eqref{eq:eventAdots} converges to zero, as $n$ goes to infinity, by Lemma \ref{lemma:meetingtime} (choosing $\gamma_1 = \frac{3}{4}$ and $\gamma = \frac{13}{14}$). About the second term in \eqref{eq:eventAdots}, we have by Lemma \ref{lemma:individualconvergence} that both $r_1^{(n)}$ and $\overline{r}_1^{(n)}$ converge under diffusive scale to Brownian motions with diffusion coefficient $\lambda_p$ and drift $b_p$. So, by the Strong Markov property, this second term also converges to zero as $n$ goes to infinity and hence $P(\mathcal{A}_{n,s}^c)$ goes to zero as $n$ goes to infinity.

Now let us fix a uniformly continuous bounded function $H:C([0,s], \mathbb{R}^{k+\tilde{k}}) \rightarrow \mathbb{R}$. Denoting by $\left(L_1, \ldots, L_k, R_1, \ldots, R_{\tilde{k}}\right)$ the collection of left-right coalescence Brownian motions starting from $(z_1, \ldots, z_k, \tilde{z}_1, \ldots, \tilde{z}_{\tilde{k}})$, then proving the desired convergence is equivalent to show that
$$\lim_{n\to \infty} \Big| E \Big[ H \Big(
{\scriptstyle \frac{l_{1}^{(n)}(sn^2)}{n}, \ldots, \frac{l_{k}^{(n)}(sn^2)}{n}, \frac{r_{1}^{(n)}(sn^2)}{n}, \ldots, \frac{r_{\tilde{k}}^{(n)}(sn^2)}{n} }\Big)\Big] \hspace{-0.03cm}-\hspace{-0.03cm} E\big[H\hspace{-0.05cm}\big(L_1, \ldots, L_k, R_1 \ldots R_{\tilde{k}}\big)\big] \Big| = 0.$$
To prove it, use $(\overline{l}_1^{(n)}, \ldots, \overline{l}_k^{(n)}, \overline{r}_1^{(n)}, \ldots, \overline{r}_{\tilde{k}}^{(n)})$ and $(\check{l}_1^{(n)}, \ldots, \check{l}_k^{(n)}, \check{r}_1^{(n)}, \ldots, \check{r}_{\tilde{k}}^{(n)})$ to bound the expression above by
$$\Big| E\Big[H\Big( {\scriptstyle\frac{l_{1}^{(n)}(sn^2)}{n}, \ldots, \frac{l_{k}^{(n)}(sn^2)}{n}, \frac{r_1^{(n)}(sn^2)}{n}, \ldots, \frac{r_{\tilde{k}}^{(n)}(sn^2)}{n}\Big)\Big]} - E\Big[H\Big( {\scriptstyle\frac{\overline{l}_{1}^{(n)}(sn^2)}{n}, \ldots, \frac{\overline{l}_{k}^{(n)}(sn^2)}{n}, \frac{\overline{r}_1^{(n)}(sn^2)}{n}, \ldots, \frac{\overline{r}_{\tilde{k}}^{(n)}(sn^2)}{n}}\Big)\Big] \Big|$$
$$+ \Big|  E\Big[H\Big( { \scriptstyle \frac{\overline{l}_{1}^{(n)}(sn^2)}{n}, \ldots, \frac{\overline{l}_{k}^{(n)}(sn^2)}{n}, \frac{\overline{r}_1^{(n)}(sn^2)}{n}, \ldots, \frac{\overline{r}_{\tilde{k}}^{(n)}(sn^2)}{n} }\Big)\Big] - E\Big[H\Big( { \scriptstyle \frac{\check{l}_{1}^{(n)}(sn^2)}{n}, \ldots, \frac{\check{l}_{k}^{(n)}(sn^2)}{n}, \frac{\check{r}_1^{(n)}(sn^2)}{n}, \ldots, \frac{\check{r}_{\tilde{k}}^{(n)}(sn^2)}{n}\Big)}\Big] \Big|$$
\begin{equation}
+ \Big| E\Big[H\Big( {\scriptstyle \frac{\check{l}_{1}^{(n)}(sn^2)}{n}, \ldots, \frac{\check{l}_{k}^{(n)}(sn^2)}{n}, \frac{\check{r}_1^{(n)}(sn^2)}{n}, \ldots, \frac{\check{r}_{\tilde{k}}^{(n)}(sn^2)}{n} }\Big)\Big] - E\big[H\big( L_1, \ldots, L_k, R_1 \ldots R_{\tilde{k}}\big)\big] \Big|.
\label{eq:multihellybray}
\end{equation}

The first term in \eqref{eq:multihellybray} is bounded above by
\begin{align} \label{eq:hellybrayfirstterm}
 &{\scriptstyle \Big| E\Big[\Big(H\Big(\frac{l_{1}^{(n)}(sn^2)}{n}, \ldots, \frac{l_{k}^{(n)}(sn^2)}{n}, \frac{r_1^{(n)}(sn^2)}{n}, \ldots, \frac{r_{\tilde{k}}^{(n)}(sn^2)}{n}\Big) - H\Big(\frac{\overline{l}_{1}^{(n)}(sn^2)}{n}, \ldots, \frac{\overline{l}_{k}^{(n)}(sn^2)}{n}, \frac{\overline{r}_1^{(n)}(sn^2)}{n}, \ldots, \frac{\overline{r}_{\tilde{k}}^{(n)}(sn^2)}{n}\Big)\Big)\mathbb{I}_{\mathcal{A}_{n,s}}\Big] \Big|} \nonumber\\
 & {\scriptstyle + \Big| E\Big[\Big(H\Big(\frac{l_{1}^{(n)}(sn^2)}{n}, \ldots, \frac{l_{k}^{(n)}(sn^2)}{n}, \frac{r_1^{(n)}(sn^2)}{n}, \ldots, \frac{r_{\tilde{k}}^{(n)}(sn^2)}{n}\Big) - H\Big(\frac{\overline{l}_{1}^{(n)}(sn^2)}{n}, \ldots, \frac{\overline{l}_{k}^{(n)}(sn^2)}{n}, \frac{\overline{r}_1^{(n)}(sn^2)}{n}, \ldots, \frac{\overline{r}_{\tilde{k}}^{(n)}(sn^2)}{n}\Big)\Big)\mathbb{I}_{\mathcal{A}^c_{n,s}}\Big] \Big|}.
\end{align}
On the event $\mathcal{A}_{n,s}$, we have that $\sup_{0\le t\le sn^2} \left| \frac{l_{i}^{(n)}(t)}{n} - \frac{\overline{l}_{i}^{(n)}(t)}{n} \right| \xrightarrow[n \to \infty]{} 0$, for $i = 1, \ldots, k$, and $\sup_{0\le t\le sn^2} \left| \frac{r_{j}^{(n)}(sn^2)}{n} - \frac{\overline{r}_{j}^{(n)}(sn^2)}{n} \right| \xrightarrow[n \to \infty]{} 0$, for $j = 1, \ldots, \tilde{k}$. Then, since $H$ is a continuous function, we have that the first term on the right hand side of \eqref{eq:hellybrayfirstterm} converges to zero as $n$ goes to infinity. Besides that, since $P(\mathcal{A}_{n,s}^c) \to 0$ as $n$ goes to infinity and $H$ is a and bounded function, we have that the second term on the right hand side of \eqref{eq:hellybrayfirstterm} also converges to zero as n goes to infinity.

About the second term in \eqref{eq:multihellybray}, note that conditioned to $\mathcal{B}_{n,s}$, we have that the processes $(\overline{l}_1^{(n)}, \ldots, \overline{l}_k^{(n)}, \overline{r}_1^{(n)}, \ldots, \overline{r}_{\tilde{k}}^{(n)})$ and $(\check{l}_1^{(n)}, \ldots,\check{l}_k^{(n)},$ $\check{r}_1^{(n)}, \ldots, \check{r}_{\tilde{k}}^{(n)})$ are equal in $[0,sn^2]$. Thus
$${\scriptstyle \Big|  E\Big[H\Big(\frac{\overline{l}_{1}^{(n)}(sn^2)}{n}, \ldots, \frac{\overline{l}_{k}^{(n)}(sn^2)}{n}, \frac{\overline{r}_1^{(n)}(sn^2)}{n}, \ldots, \frac{\overline{r}_{\tilde{k}}^{(n)}(sn^2)}{n}\Big)\Big] - E\Big[H\Big(\frac{\check{l}_{1}^{(n)}(sn^2)}{n}, \ldots, \frac{\check{l}_{k}^{(n)}(sn^2)}{n}, \frac{\check{r}_1^{(n)}(sn^2)}{n}, \ldots, \frac{\check{r}_{\tilde{k}}^{(n)}(sn^2)}{n}\Big)\Big] \Big| =}$$
$${\scriptstyle = \Big| E\Big[\Big(H\Big(\frac{\overline{l}_{1}^{(n)}(sn^2)}{n}, \ldots, \frac{\overline{l}_{k}^{(n)}(sn^2)}{n}, \frac{\overline{r}_1^{(n)}(sn^2)}{n}, \ldots, \frac{\overline{r}_{\tilde{k}}^{(n)}(sn^2)}{n} \Big) - H\Big(\frac{\check{l}_{1}^{(n)}(sn^2)}{n}, \ldots, \frac{\check{l}_{k}^{(n)}(sn^2)}{n}, \frac{\check{r}_1^{(n)}(sn^2)}{n}, \ldots, \frac{\check{r}_{\tilde{k}}^{(n)}(sn^2)}{n}\Big)\Big) \mathbb{I}_{\mathcal{B}^c_{n,s}} \Big] \Big|}$$
$$\leq 2(k+\tilde{k})sn^2\exp\left\{-\frac{c}{4}\left(n^{\frac{3}{4}}-4\right)\right\} \left\|H\right\|_{\infty},$$
which converges to zero as n goes to infinity since $H$ is a bounded function.

Now it remains to deal with the last term in \eqref{eq:multihellybray}. To show that it converges to zero, we use the same argument of \cite[Proposition 5.2]{ss1} to prove an equivalent result for simple coalescing random walks with branching. First, note that, for any fixed $t \in [0,sn^2]$, we have that all l-paths $\check{l}^{(n)}_i, i=1,\ldots,k$, are independent until they meet or try to cross each other, i.e. when they coalesce. Now we will argue that the jump probabilities of $\check{r}_j^{(n)}(t)$, for any $j = 1,\ldots,\tilde{k}$, given the position of the nearest l-path at its left-hand side, which we denote by $\check{l}^{(n)}_m$, is independent of any other l-path $\check{l}^{(n)}_i, i=1,\ldots,k$ that have not coalesced yet with $\check{l}^{(n)}_m$: 
\begin{itemize}
\item If $i < m$ and $\check{l}_{i}^{(n)}(t) \neq \check{l}_{m}^{(n)}(t)$, then either (1) $\check{l}_{i}^{(n)}(t)$ is at a distance greater than $n^{\frac{3}{4}}$ from $\check{r}_j^{(n)}(t)$
or (2) $\check{l}_{i}^{(n)}(t)$ is at a distance smaller than or equal to $n^{\frac{3}{4}}$ from $\check{r}_j^{(n)}(t)$. In situation (1) $\check{r}_j^{(n)}(t)$ jumps independently of $\check{l}_{i}^{(n)}(t)$ because the paths $\check{r}_j^{(n)}$ and $\check{l}_{i}^{(n)}$ can not make jumps of size greater than $n^{\frac{3}{4}}/4$. In situation (2) we also have independence because $\check{l}_{i}^{(n)}(t)$ is also at a distance smaller than $n^{\frac{3}{4}}$ from $\check{l}_{m}^{(n)}(t)$ and consequently, $\check{l}_{i}^{(n)}(t)$ is evolving according to a different environment. Note that if $\check{l}_{i}^{(n)}(t)$ would make a jump trying to cross $\check{r}_j^{(n)}(t)$, then $\check{l}_{i}^{(n)}(t)$ would also try to cross $\check{l}_{m}^{(n)}(t)$ and hence these two l-paths would coalesce.
\item If $i > m$ and $\check{l}_{i}^{(n)}(t) \neq \check{l}_{m}^{(n)}(t)$, then either (1) if $\check{l}_{i}^{(n)}(t)$ is at a distance smaller than $n^{\frac{3}{4}}/2$ from $\check{r}_j^{(n)}(t)$ or (2) $\check{l}_{i}^{(n)}(t)$ is at a distance greater than $n^{\frac{3}{4}}/2$ from $\check{r}_j^{(n)}(t)$. In situation (1) we have that $\check{r}_j^{(n)}(t)$ is independent of $\check{l}_{i}^{(n)}(t)$, because $\check{r}_j^{(n)}(t)$ is evolving according to a different environment. In situation (2) we have independence because $\check{r}_j^{(n)}$ and $\check{l}_{i}^{(n)}$ can not make jumps of size greater than $n^{\frac{3}{4}}/4$.
\end{itemize}

\vspace{-0.08cm}

Besides that, we also have independence between pairs of r-paths $\check{r}_j, j=1,\ldots,\tilde{k}$. When these r-paths are at a distance greater than $n^{\frac{3}{4}}$ of each other, we have independence because $\check{r}_j$ do not make jumps of size greater than $\frac{n^{\frac{3}{4}}}{4}$ and when they are at a distance smaller than $n^{\frac{3}{4}}$ of each other, we also have independence because they will be evolving in different environments.

So, due to the Markov property and the independence mentioned above, we have that $(\check{l}_1^{(n)}, \ldots, \check{l}_k^{(n)},\check{r}_1^{(n)}, \ldots, \check{r}_{\tilde{k}}^{(n)})$ can be constructed inductively by concatenating independent evolution of sets of paths, where each set consists of either a single l-path, a single r-path or a pair of left-right paths, like was done in \cite{ss1}.

Before we describe the construction, for a fixed $n$, define the non decreasing sequence of stopping times $\check{\tau}_j^{(n)}$ where $\check{\tau}_j^{(n)}$ is the j-th time where either a coalescence between two l-paths occurs, a coalescence between two r-paths occurs or a crossing between an r-path and an l-path occurs. Since we have at most $k+\tilde{k}-2$ coalescence events and at most $k\tilde{k}$ crossings, the sequence $\check{\tau}_j^{(n)}$ will have at most $k\tilde{k} + k + \tilde{k} -2$ stopping times.

We first construct the system up to $\check{\tau}_1^{(n)}$. The paths are initially ordered as $\check{r}_1^{(n)}, \ldots, \check{r}_{\tilde{k}}^{(n)},$ $\check{l}_1^{(n)}, \ldots, \check{l}_k^{(n)}$, so, we can partition them as $\{\check{r}_1^{(n)}\}\ldots\{\check{r}_{\tilde{k}}^{(n)}\}\{\check{l}_1^{(n)}\}\ldots\{\check{l}_k^{(n)}\}$ and each partition evolve independently until they hit $\check{\tau}_1^{(n)}$. When it happens, we have a change in this partition. For example, if $\check{\tau}_1^{(n)}$ is related to a crossing event between $\check{r}_{\tilde{k}}^{(n)}$ and $\check{l}_1^{(n)}$, we will begin to have the partition $\{\check{r}_1^{(n)}\}\ldots\{\check{r}_{\tilde{k}-1}^{(n)}\} \{\check{l}_1^{(n)},\check{r}_{\tilde{k}}^{(n)}\}\{\check{l}_2^{(n)}\}\ldots\{\check{l}_k^{(n)}\}$. If  $\check{\tau}_1^{(n)}$ is related to a coalescence event, let us say between $\check{l}_2^{(n)}$ and $\check{l}_3^{(n)}$ for example, we will have the partition $\{\check{r}_1^{(n)}\}\ldots\{\check{r}_{\tilde{k}}^{(n)}\}\{\check{l}_1^{(n)}\}\{\check{l}_3^{(n)}\}\{\check{l}_4^{(n)}\}\ldots\{\check{l}_k^{(n)}\}$. After that, we continue to let the partition elements evolve independently until they hit $\check{\tau}_2^{(n)}$, when we have another change in the partition. This procedure continues in the same way, changing the partition whenever the system hits a stopping time $\check{\tau}_j^{(n)}$. In addition, note that some of the stopping times $\check{\tau}_j^{(n)}$ can be equal. Since the sequence $\check{\tau}_j^{(n)}$ has at most a finite number of stopping times, we change the partition only a finite number of times and it eventually leads to a single pair $\{\check{l}_k^{(n)},\check{r}_{\tilde{k}}^{(n)}\}$. 

So, by Markov property of our system and the observation that the stopping times $\check{\tau}_j^{(n)}$ used in the construction are almost surely continuous functionals on $\Pi^{k+\tilde{k}}$ with respect to the law of independent evolution of paths in different partition elements, we can describe the weak limit under diffusive scale according to the limit of each partition element. By Proposition \ref{lemma:onepairconverge} we have that the partition elements composed of a pair of an l-path and an r-path converge to a pair of left-right coalescing Brownian motion. By Lemma \ref{lemma:individualconvergence}, we have that the partition elements composed by a single l-path or r-path converge to a Brownian motion with drift $-b_p$ and $+b_p$ respectively.

This completes the proof of convergence when all paths start at the same time.

\begin{remark}
If we have paths starting from different times, we just need to add more stopping times to the sequence $\check{\tau}_j^{(n)}$ to include the times when a new path starts in the system. But the argument can still be used and is basically the same. 
\end{remark}

Finally, the proofs presented this section give us the proof of Proposition \ref{lemma:conditionI}, since the general case will follow by induction. 

\vspace{0.2cm}


\section{Proof of Theorems \ref{Teo:leftright} and \ref{Teo:teoprincipal}}
\label{sec:lefttobrownian} 

This section is devoted to proving Theorems \ref{Teo:leftright} and \ref{Teo:teoprincipal}.

\smallskip

In the proof of Theorem \ref{Teo:leftright}, we follow the same steps as the proof of \cite[Theorem 5.3]{ss1}. We will need the next result, which is equivalent to \cite[Theorem 2.9]{dual} for the Drainage Network without branching:

\begin{proposition}
\label{prop:dualwebconv}
Under the conditions of Theorem \ref{Teo:leftright}, we have that $(W_n^l, \hat{W}_n^l) \Longrightarrow (\mathcal{W}_{\lambda,b_p}^l, \widehat{\mathcal{W}}_{\lambda,b_p}^l)$ and $(W_n^r, \hat{W}_n^r) \Longrightarrow (\mathcal{W}_{\lambda,b_p}^r, \widehat{\mathcal{W}}_{\lambda,b_p}^r)$.
\end{proposition}

\medskip

We will prove Proposition \ref{prop:dualwebconv} after proving Theorem \ref{Teo:leftright}.

\medskip

\noindent \textit{Proof of Theorem \ref{Teo:leftright}.} By Proposition \ref{prop:dualwebconv} we have that $\{(W_n^l,W_n^r,\hat{W}_n^l,\hat{W}_n^r)\}_{n \in \mathbb{N}}$ is tight. Let $(X^l,X^r,\mathcal{Z}^l,\mathcal{Z}^r)$ be a subsequential weak limit of this sequence. Then, again by Proposition \ref{prop:dualwebconv}, we have that $(X^l,\mathcal{Z}^l)$ is distributed as $(\mathcal{W}^l_{\lambda,b_p},\widehat{\mathcal{W}}^l_{\lambda,b_p})$ and $(X^r,\mathcal{Z}^r)$ is distributed as $(\mathcal{W}^r_{\lambda,b_p},\widehat{\mathcal{W}}^r_{\lambda,b_p})$. Recall that $(W_n^l, W_n^r)$ and $(\mathcal{W}^l_{\lambda,b_p},\mathcal{W}^r_{\lambda,b_p})$ determine their dual processes, which are respectively $(\hat{W}_n^l,\hat{W}_n^r)$ and $(\widehat{\mathcal{W}}^l_{\lambda,b_p},\widehat{\mathcal{W}}^r_{\lambda,b_p})$. Then, if we have that 
\begin{equation} \label{Gdualconv}
(W_n^l,  W_n^r) \Longrightarrow (\mathcal{W}_{\lambda,b_p}^l, \mathcal{W}_{\lambda,b_p}^r),
\end{equation}
it follows that $(X^l,X^r,\mathcal{Z}^l,\mathcal{Z}^r)$ has the same distribution as $(\mathcal{W}^l_{\lambda,b_p},\mathcal{W}^r_{\lambda,b_p},\widehat{\mathcal{W}}^l_{\lambda,b_p},\widehat{\mathcal{W}}^r_{\lambda,b_p})$, which proves Theorem \ref{Teo:leftright}.  \hfill $\square$\\

To prove \eqref{Gdualconv}, we have to verify that $(X^l,X^r)$ satisfies conditions (i), (ii) and (iii) in Proposition \ref{Result:leftrightweb}. Conditions (i) and (iii) follow from Proposition \ref{prop:dualwebconv}. In addition to that, by condition (I$_{\mathcal{N}}$) (Proposition \ref{lemma:conditionI}) we have that $(X^l,X^r)$ also satisfies condition (ii), consequently $(X^l,X^r)$ has the same distribution as the left-right Brownian Web $(\mathcal{W}^l_{\lambda,b_p},\mathcal{W}^r_{\lambda,b_p})$. \hfill 

\medskip

\noindent \textit{Proof of Proposition \ref{prop:dualwebconv}.} We will only prove Proposition \ref{prop:dualwebconv} for the system with l-paths, but the proof for the r-paths is analogous. Let us begin by introducing some maps. Let $\phi:\Pi \rightarrow \Pi$ such that, for any path $\pi_t \in \Pi$, we have: 
$$
\phi(\pi_t) = \pi_t + b_p t, \mbox{ for any } t \ge \sigma_\pi.
$$
Now let $\varphi:\mathcal{H} \rightarrow \mathcal{H}$ be such that, for any $K \in \mathcal{H}$,
$$
\varphi(K) = \{\phi(\pi), \pi \in K\}.
$$
Analogously, let $\varphi_2:\mathcal{H} \times \widehat{\mathcal{H}} \rightarrow \mathcal{H} \times \widehat{\mathcal{H}}$ be such that, for any $(K,\hat{K}) \in \mathcal{H} \times \widehat{\mathcal{H}}$,
$$
\varphi_2(K,\hat{K}) = \{(\phi(\pi),\phi(\hat{\pi})), \pi \in K, \hat{\pi} \in \hat{K}\}.
$$
These three maps are homeomorphisms in their respective domains. 

Note that the map $\varphi_2$, when applied to the set of l-paths and its dual, removes the drift of all these paths; and when applied to $(\mathcal{W}_{\lambda,b_p}^{l},\mathcal{W}_{\lambda,b_p}^l$) gives us the Brownian Web together with its dual. The main idea of the proof, following an approach used in \cite{fav} to prove the main theorem of that article, will be to verify that
\begin{equation}
\varphi_2(W_n^l,\hat{W}_n^l) \Longrightarrow (\mathcal{W}_\lambda,\widehat{\mathcal{W}}_\lambda), \mbox{ as } n \rightarrow \infty.
\label{eq:mapconvergence}
\end{equation}

Since the map $\varphi_2$ is a homeomorphism on $\mathcal{H} \times \widehat{\mathcal{H}}$, we have that \eqref{eq:mapconvergence} implies the statement of Proposition \ref{prop:dualwebconv}. To prove that \eqref{eq:mapconvergence} holds, we will begin by proving that
\begin{equation}
\varphi(W_n^l) \Longrightarrow \mathcal{W}_\lambda, \mbox{ as } n \rightarrow \infty.
\label{eq:partialmapconvergence}
\end{equation}

To verify that \eqref{eq:partialmapconvergence} holds, we will follow the converge criterion to Brownian Web presented in Section \ref{subsec:webcritterion}. So, according to that criterion, since we do not have crossings among the paths of $\varphi(W_n^l)$, it is only necessary to prove conditions (I) and (E$'$), also described in Section \ref{subsec:webcritterion}. To simplify the notation, we denote $W_n = \varphi(W_n^l)$ and $(W_n,\hat{W}_n) = \varphi_2(W_n^l,\hat{W}_n^l)$.

We already have the condition (I) as a consequence of Proposition \ref{lemma:conditionI}.

As in Section \ref{subsec:webcritterion}, denote by $W_n^{t_0^{-}}$ the subset of paths in $W_n$ which starts before or at time $t_0$ and let $\mathcal{Z}_{t_0}$ be a subsequential limit of $W_n^{t_0^{-}}$ for any fixed $t_0 \in \mathbb{R}$. Recalling the condition (E$'$), we have to prove that 
$$
E[\hat{\eta}_{\mathcal{Z}_{t_0}}(t_0,t,a,b)] \leq E[\hat{\eta}_{\mathcal{W}}(t_0,t,a,b)] = \frac{b-a}{\sqrt{\pi t}},
$$
where $\hat{\eta}_{\mathcal{Z}_{t_0}}(t_0,t,a,b)$, for $a<b$, is the number of distinct points in $(a,b) \times \{t_0+t\}$ that are occupied by a path of $\mathcal{Z}_{t_0}$ that started before or at time $t_0$.

The verification of condition (E$'$) follows the same arguments presented in \cite[Section 6 ]{nrs}. The aim is to show that \cite[Lemmas 6.2 and 6.3]{nrs} also hold in our case, since it was proved in that paper that these two lemmas imply condition (E$'$) for general models of coalescing random walks, in particular to the DNB, see also \cite{cfd} and \cite{cv}.

Fix $\delta > 0$ and denote by $\mathcal{Z}_{t_0}(t_0+\delta)$ the intersection of paths in $\mathcal{Z}_{t_0}$ with the time line $t_0 + \delta$. To prove that \cite[Lemma 6.2]{nrs} holds in our case we need to show that $\mathcal{Z}_{t_0}(t_0+\delta)$ is almost surely locally finite for any $\delta > 0$. To prove that \cite[Lemma 6.3]{nrs} holds in our case we need to show that the set of paths in $\mathcal{Z}_{t_0}$ truncated before time $t_0 + \delta$, for any $\delta > 0$ is distributed as coalescing Brownian motions starting from the random set $\mathcal{Z}_{t_0}(t_0+\delta) \subset \mathbb{R}^2$.

We can argue that $\mathcal{Z}_{t_0}(t_0+\delta)$ is almost sure locally finite for any $\delta > 0$ as a consequence of Lemma \ref{lemma:dnbcoalescencetime}. Define the event 

\vspace{0.2cm}

$O_t = \{$There is an l-path that starts from a point $z \in \mathbb{Z}$ at time $0$ that visits the origin at time $t \}$. 

\vspace{0.2cm}

\noindent One can see that
\begin{equation}
\label{eq:locallyfiniteness}
P(O_t) \leq \frac{C}{\sqrt{t}},
\end{equation}
and it gives the locally finiteness property of $\mathcal{Z}_{t_0}(t_0+\delta)$ (\cite[Lemma 6.2]{nrs}). The statement in \eqref{eq:locallyfiniteness} is equivalent to  \cite[Lemma 3.3]{cv} for the generalized Drainage Network and \cite[Lemma 2.7]{nrs} for independent random walks. The proof that \eqref{eq:locallyfiniteness} holds given the upper bound of Lemma \ref{lemma:dnbcoalescencetime} is analogous to the proof of \cite[Lemma 2.7]{nrs}.

Now, since we have locally finiteness of $\mathcal{Z}_{t_0}(t_0+\delta)$, condition (I) implies that $\mathcal{Z}_{t_0}^{(t_0+\delta)_T}$ is distributed as coalescing Brownian motions starting from the random set $\mathcal{Z}_{t_0}(t_0+\delta)$, where $\mathcal{Z}_{t_0}^{(t_0+\delta)_T}$ is the set of paths in $\mathcal{Z}_{t_0}$ that start before or at time $t_0$ and are truncated before time $t_0+\delta$. Then, we get the statement of \cite[Lemma 6.3]{nrs} and condition (E$'$) from the fact that a family of coalescing Brownian motions starting from a random locally finite subset of the real line is stochastically dominated by a system of coalescing Brownian motions starting from every point in $\mathbb{R}$ at time $t_0 + \delta$ and taking limit when $\delta \to 0$. Hence, \eqref{eq:partialmapconvergence} holds.

The next step will be to show \eqref{eq:mapconvergence} using that $\varphi(W_n^l) \Longrightarrow \mathcal{W}$ as $n \rightarrow \infty$. We will use an approach similar to the one used in \cite[Theorem $2.9$]{dual} to prove an equivalent result for a Drainage Network model without branching. The main idea here is to use the following result:

\begin{proposition}
\label{prop:dualwebcritter}
\textbf{(\cite[Proposition $2.7$]{dual})} Let $(\mathcal{W},\mathcal{Z})$ be a $(\mathcal{H} \times \widehat{\mathcal{H}},\mathcal{B}_{\mathcal{H}} \times \mathcal{B}_{\widehat{\mathcal{H}}})$-valued random variable with $\mathcal{W}$ distributed as the Brownian Web. If the following conditions are satisfied:
\begin{enumerate}
\item[(i)] For any deterministic $(x,t) \in \mathbb{R}^2$, there exists a path $\hat{\pi}^{(x,t)} \in \mathcal{Z}$ starting at $(x,t)$ and going backward in time almost surely;
\item[(ii)] Paths in $\mathcal{Z}$ do not cross paths in $\mathcal{W}$ almost surely. That is, there does not exist any $\pi \in \mathcal{W}$, $\hat{\pi} \in \mathcal{Z}$ and $t_1, t_2 \in (\sigma_{\pi},\sigma_{\hat{\pi}})$ such that $(\hat{\pi}(t_1) - \pi(t_1))(\hat{\pi}(t_2)-\pi(t_2)) < 0$ almost surely;
\item[(iii)] Paths in $\mathcal{Z}$ and paths in $\mathcal{W}$ do not coincide over any time interval almost surely. That is, for any $\pi \in \mathcal{W}$ and $\hat{\pi} \in \mathcal{Z}$, does not exist a pair of points $t_1 < t_2$ with $\sigma_{\pi} \leq t_1 < t_2 \leq \sigma_{\hat{\pi}}$ such that $\hat{\pi}_t = \pi_t$ for all $t \in [t_1,t_2]$ almost surely;
\end{enumerate}
Then, $\mathcal{Z} = \widehat{\mathcal{W}}$ almost surely.
\end{proposition}

Since $\hat{W}_n^l$ consists only of paths that can not cross each other, tightness of the family $\{\hat{W}_n^l\}_{n \in \mathbb{N}}$ follows from tightness of the non dual system obtained from \eqref{eq:partialmapconvergence}. However we provide a more direct proof of the convergence for dual paths in Lemma \ref{lemma:onedualpath} in Appendix \ref{sec:conditionU} which implies tightness according to \cite[Proposition B.2]{finr1} (Lemma \ref{lemma:onedualpath} and its proof are of interest by themselves, see Remark \ref{remark-A1}). So, the joint families $\{(W_n^l,\hat{W}_n^l)\}_{n \in \mathbb{N}}$ and $\{(W_n,\hat{W}_n)\}_{n \in \mathbb{N}}$ are also tight since both have tight coordinates. Then, to prove Proposition \ref{prop:dualwebconv} it is enough to show that for any subsequential limit $(\mathcal{W},\mathcal{Z})$ of $\{(W_n,\hat{W}_n)\}_{n \in \mathbb{N}}$, the random variable $\mathcal{Z}$ satisfies the conditions given in Proposition \ref{prop:dualwebcritter}. By Skorohod's representation theorem, we may assume that the convergence to $(\mathcal{W},\mathcal{Z})$ happens almost surely.

By Lemma \ref{lemma:onedualpath} for any deterministic $(x,t) \in \mathbb{R}^2$ there exists a path $\hat{\pi} \in \mathcal{Z}$ starting at $(x,t)$ going backward in time almost surely, which gives us (i) of Proposition \ref{prop:dualwebcritter}.

Since we are assuming the almost surely convergence of $\{(W_n,\hat{W}_n)\}_{n \in \mathbb{N}}$ to $(\mathcal{W},\mathcal{Z})$, if a dual path in $\mathcal{Z}$ crosses a path in $\mathcal{W}$, there exists a dual path in $\hat{W}_n^l$ which crosses a path in $W_n^l$, for some $n \geq 1$, which, by construction, is a contradiction. Hence, the paths in $\mathcal{Z}$ do not cross paths in $\mathcal{W}$ almost surely, which gives us (ii) in Proposition \ref{prop:dualwebcritter}.

Now, to verify condition (iii), for $\delta > 0$ and each integer $m \geq 1$, define $A_{\delta,m} = \{$there exist paths $\pi \in \mathcal{W}$, $\hat{\pi} \in \mathcal{Z}$ with $\sigma_{\pi},\sigma_{\hat{\pi}} \in (-m,m)$, and there exists $t_0$ such that $\sigma_{\pi} < t_0 < t_0 + \delta < \sigma_{\hat{\pi}}$ with $-m < \pi(t) = \hat{\pi}(t) < m$ for all $t \in [t_0,t_0+\delta] \}$.

\noindent It is enough to show that for any fixed $\delta > 0$ and for $m \geq 1$, we have $P(A_{\delta,m}) = 0$. Note that the $P(A_{\delta,m}) = 0$ if and only if the equivalent event considering $\pi \in \mathcal{W}_{b_p}^{l}$ and $\hat{\pi} \in \mathcal{Z}^l$ also have probability zero, where $(\mathcal{W}_{b_p}^{l}, \mathcal{Z}^l)$ is a subsequential weak limit of $\{(W_n^l,\hat{W}_n^l)\}_{n \in \mathbb{N}}$. Let us call this alternative version of event $A_{\delta,m}$ by $\overline{A}_{\delta,m}$.

In \cite[Theorem 2.9 and Lemma 2.11]{dual} it is proved that $P(A_{\delta,m}) = 0$ considering $(\mathcal{W},\mathcal{Z})$ as a subsequential weak limit of paths from the Drainage Network without branching and its dual under diffusive scaling. The proof here that $P(\overline{A}_{\delta,m})=0$ is analogous since the only difference is that the paths from $\mathcal{W}_{b_p}^{l}$ have a drift to the left, which implies that paths from $\mathcal{Z}^l$ will go back in time with a drift to the right, but the behavior of the distance between them is the same. It concludes the proof of Proposition \ref{prop:dualwebconv} and part I. \hfill $\square$ 

\medskip

\noindent \textit{Proof of Theorem \ref{Teo:teoprincipal}} By Proposition \ref{Result:convergencecritter2} we need to show that conditions (C), (I$_{\mathcal{N}}$), (H) and (U$_{\mathcal{N}}^{'}$) hold for $\mathcal{X}_{\epsilon_n}^n$. Condition (I$_{\mathcal{N}}$) was already proved in Section \ref{sec:conditionI}. Conditions (C), (H) and (U$_{\mathcal{N}}^{'}$) follow directly from the construction of the DNB and its dual. \hfill $\square$ 

\vspace{0.2cm}

\appendix
\section{Auxiliary results, proofs and additional discussions}
\label{sec:conditionU}

Here we prove two useful results for the dual DNB. The first result states that single l-paths and single r-paths from the dual converge to Brownian motions with drift. The second one gives an upper bound on the probability that two dual l-paths or two dual r-paths do not coalesce until time $t$ (analogous to the one obtained for non dual paths in Lemma \ref{lemma:dnbcoalescencetime}).

\subsection{Convergence of one l-path or r-path from the dual}\
\label{sec:dualconverge}

\medskip
Recall the definition of $b_p$ from the statement of Theorem \ref{Teo:teoprincipal}.

\begin{lemma}
\label{lemma:onedualpath}
For any deterministic point $(x,t) \in \mathbb{R}^2$, there exists a sequence $(x_n,t_n)$ converging to $(x,t)$, a sequence of l-paths $\hat{X}_n^l \subset \hat{W}^l_n$ starting from $(x_n,t_n)$ and a sequence of r-paths $\hat{X}_n^r \subset \hat{W}^r_n$ starting from $(x_n,t_n)$ which converge to Brownian motions starting at $x$ at time $t$, evolving backward in time, with drift $+b_p$ and $-b_p$ respectively, when $n \rightarrow \infty$ and the branching parameter is $\epsilon_n = \fb n^{-1}$. 
\end{lemma}

\begin{remark}\label{remark-A1}
We only mention Lemma \ref{lemma:onedualpath} in Section \ref{sec:lefttobrownian} as an alternative way to argue that the family $\{\hat{W}_n\}_{n \in \mathbb{N}}$ is tight, but even this is a more direct consequence of tightness for the forward system of paths and the non-crossing property. With tightness, the convergence to Brownian motions would be a consequence of Proposition \ref{prop:dualwebcritter}. So the reason to include Lemma \ref{lemma:onedualpath} here is twofold: 1. To describe the transition probability function of the dual paths. 2. The proof gives insight into the evolution of dual paths, which could be useful to prove a full convergence to the Brownian Net. It consists of decomposing the paths in a mean-zero martingale plus a drift correction term and establish convergence for both parts. For the martingale convergence we use well known results that only require appropriate moment estimates. 
\end{remark}

Recall the construction and notation from Section \ref{subsec:dualdef}. By translation invariance, it is enough to consider $(x,t) = (0,0)$ and $(x,t) = (\frac{1}{2},0)$ depending if either $x \in \mathbb{Z}$ or $x \notin \mathbb{Z}$, but since the proof of case $(x,t) = (\frac{1}{2},0)$ would be analogous, we will consider only the case $(x,t) = (0,0)$. Also recall that $\hat{Y}^l_t(\hat{z})$ and $\hat{Y}^r_t(\hat{z})$ denotes the positions respectively of the dual l-path and the dual r-path that start from the dual vertex $\hat{z}$ after $t$ steps, i.e, at time $\hat{z_2}-t$.

\begin{proposition}
\label{proposition:markov}
For any $\hat{z} \in \hat{V}$,  $(\hat{Y}_k^l(\hat{z}))_{k \geq 0}$ and $(\hat{Y}_k^r(\hat{z}))_{k \geq 0}$ are time homogeneous Markov processes with respect to the filtration $\hat{\mathcal{F}}_k = \sigma\big( (\omega(v),\theta(v)): v \in \mathbb{Z}^2, \hat{z}_2 \geq v_2 \geq \hat{z}_2 - k \} \big)$.
\end{proposition}

\noindent \textit{Proof.} We will present the proof for $(\hat{Y}_k^l(\hat{z}))_{k \geq 0}$ and the proof for $(\hat{Y}_k^r(\hat{z}))_{k \geq 0}$ is analogous.

We need the specification of the transition probabilities of $(\hat{Y}_k^l(\hat{z}))_{k \geq 0}$. Recall that $\hat{Y}_k^l(\hat{z})$ assumes values in $\frac{1}{2}\mathbb{Z}$. The description of the transition probabilities depends on whether the current value of the process is an integer or not, so we study each case separately.

\noindent (i) If the current value of $\hat{Y}_k^l(\hat{z})$ is non-integer: In this case we cannot have a vertex from $V$ in front of $(\hat{Y}_k^l(\hat{z}),\hat{z}_2-k)$ since $(\hat{Y}_k^l(\hat{z}),\hat{z}_2-(k+1)) \notin \mathbb{Z}^2$. This implies that the dual path can not be blocked by paths in the DNB that have just branched, so we also can not have a branching occurring on the dual process at this vertex. Thus this dual vertex will be connected to the nearest one that it can reach without crossing a path from DNB.

It is straightforward to verify from definitions that if $z\in \hat V$ and $\hat{z}\in \hat V$ satisfy that either $z_2 \neq \hat{z}_2$ or $z_2 = \hat{z}_2$ and $z_1 >\hat{z}_1$, then $K^r(z)$ and $K^l(\hat{z})$ are IID $Geom(p)$.
To simplify notation, let $(G_1,G_2)$ be a random vector distributed as IID $Geom(p)$.

For $u \notin \mathbb{Z}$ and $v \in \frac{1}{2}\mathbb{Z}$ we have:
\begin{eqnarray*}
\lefteqn{
P \hspace{-0.05cm}\left( \hat{Y}_{k+1}^l(\hat{z}) - \hat{Y}_k^l(\hat{z}) = v \ | \ \hspace{-0.05cm} \hat{Y}_k^l(\hat{z}) = u \right) } \\
& = & P \hspace{-0.05cm}\left( K^r \hspace{-0.05cm}\left( u-\frac{1}{2}, z_2-(k+1) \right) - K^l \hspace{-0.05cm}\left( u+\frac{1}{2}, z_2-(k+1)\right) = 2v \hspace{-0.05cm} \right) \\
& = & P \left( G_1 - G_2 = 2v \right).
\end{eqnarray*}

\noindent (ii) If the current value of $\hat{Y}_k^l(\hat{z})$ is an integer: In this case we have three distinct scenarios which are (a) either we do not have a vertex from $V$ in front of $(\hat{Y}_k^l(\hat{z}),\hat{z}_2-k)$ or (b) that vertex exists and the DNB branches (branching is possible since the dual vertex is the midpoint between two consecutive open vertices) or (c) that vertex exists and the DNB does not branch. Situations, (a), (b) and (c) are formally described just below. To help the reader, in Figure \ref{demomarkovfig} we exemplify the occurrence of the situations (a), (b) and (c).
\begin{itemize}
\item[(a)] The vertex $(\hat{Y}_k^l(\hat{z}),\hat{z}_2-(k+1))$ is not in $V$ (not open): This event occurs with probability $(1-p)$ and, for $u \in \mathbb{Z}$ and $v \in \frac{1}{2}\mathbb{Z}$, we have that 
$$
P \hspace{-0.05cm}\left( \hat{Y}_{k+1}^l(\hat{z}) - \hat{Y}_k^l(\hat{z}) = v, \ (\hat{Y}_k^l(\hat{z}),\hat{z}_2-(k+1)) \notin V \ | \ \hspace{-0.05cm} \hat{Y}_k^l(\hat{z}) = u \right)
$$
\begin{eqnarray*}
&=& P(\omega(u,z_2-(k+1)) = 0) \\
& & \quad P \left( \left. \hat{Y}_{k+1}^l(\hat{z}) - \hat{Y}_k^l(\hat{z}) = v \ \right| \ \hat{Y}_k^l(\hat{z}) = u, \omega(u,z_2-(k+1)) = 0 \right) \\
& = & (1-p) P \hspace{-0.05cm}\left( K^r \hspace{-0.05cm}\left( u, z_2-(k+1) \right) - K^l \hspace{-0.05cm}\left( u, z_2-(k+1)\right) = 2v \hspace{-0.05cm} \right) \\
& = & (1-p) P \left( G_1 - G_2 = 2v \right).
\end{eqnarray*}
\end{itemize}
\begin{itemize}
\item[(b)] The vertex $(\hat{Y}_k^l(\hat{z}),\hat{z}_2-(k+1))$ is in $V$ and a branching of the DNB occurs at this vertex. This event occurs with probability $p\epsilon_n$ implying that the dual process also branches at vertex $(\hat{Y}_k^l(\hat{z}),\hat{z}_2-k)$. Also, recall that we are dealing with dual l-paths that must jump to the left when a branching occurs. Thus for $u \in \mathbb{Z}$ and $v \in \frac{1}{2}\mathbb{Z}$ put $\tilde{z}_k = (u,z_2-(k+1))$ and write
\begin{eqnarray*}
\lefteqn{\!\!\!\!
P \hspace{-0.05cm}\left( \hat{Y}_{k+1}^l(\hat{z}) - \hat{Y}_k^l(\hat{z}) = v, \, \omega(\tilde{z}_k) = 1, \, \theta(\tilde{z}_k) = 0 \ | \ \hspace{-0.05cm} \hat{Y}_k^l(\hat{z}) = u \right) } \\
&=& P(\omega(\tilde{z}_k) = 1) \, P(\theta(\tilde{z}_k) = 0)  \\
& & P \left( \left. \hat{Y}_{k+1}^l(\hat{z}) - \hat{Y}_k^l(\hat{z}) = v \ \right| \ \hat{Y}_k^l(\hat{z}) = u, \omega(\tilde{z}_k) = 1, \theta(\tilde{z}_k) = 0  \right) \\
& = & p \epsilon_n P \hspace{-0.05cm}\left( K^r \hspace{-0.05cm}\left( \tilde{z}_k\right) = 2v \hspace{-0.05cm} \right) 
\, = \, p \epsilon_n P(G_1 = 2v). 
\end{eqnarray*}
Note that if $v \leq 0$, the second term is equal to zero since the dual l-path always chooses the left option in case when a branching occurs. This is the only term that would change in the proof for $(\hat{Y}_k^r(\hat{z}))_{k \geq 0}$ and it would be replaced by $p \epsilon_n P(G_1 = -2v)$.
\item[(c)] The vertex in position $(\hat{Y}_k^l(\hat{z}),\hat{z}_2-(k+1))$ is in $V$ but we do not have a branching of the DNB at this vertex: This event occurs with probability $p(1-\epsilon_n)$ and here the dual process have a conditional probability of $1/2$ of being forced to go to left (transition to a higher value) and a conditional probability $1/2$ of being forced to go to the right (transition to a lower value) to avoid a crossing with paths of the DNB. Thus for $u \in \mathbb{Z}$ and $v \in \frac{1}{2}\mathbb{Z}$ put $\tilde{z}_k = (u,z_2-(k+1))$ and write 
\begin{eqnarray*}
\lefteqn{\!\!\!\!
P \hspace{-0.05cm}\left( \hat{Y}_{k+1}^l(\hat{z}) - \hat{Y}_k^l(\hat{z}) = v, \, \omega(\tilde{z}_k) = 1, \, |\theta(\tilde{z}_k)| = 1 \ | \ \hspace{-0.05cm} \hat{Y}_k^l(\hat{z}) = u \right) } \\
&=& P(\omega(\tilde{z}_k) = 1) \, P(|\theta(\tilde{z}_k)| = 1)  \\
& & P \left( \left. \hat{Y}_{k+1}^l(\hat{z}) - \hat{Y}_k^l(\hat{z}) = v \ \right| \ \hat{Y}_k^l(\hat{z}) = u, \omega(\tilde{z}_k) = 1, \theta(\tilde{z}_k) = 0  \right) 
\end{eqnarray*}
which is equal to
$$
p (1-\epsilon_n) \Big( \frac{P \hspace{-0.05cm}\left( K^r \hspace{-0.05cm}\left( \tilde{z}_k\right) = 2v \hspace{-0.05cm} \right)}{2} + \frac{P \hspace{-0.05cm}\left( K^l \hspace{-0.05cm}\left( \tilde{z}_k\right) = 2v \hspace{-0.05cm} \right)}{2}
\Big) = \frac{p}{2} (1-\epsilon_n) P(G_1 = \left| 2v \right|).
$$
\end{itemize}
\begin{figure}[h!]
\centering
\includegraphics[scale=0.6]{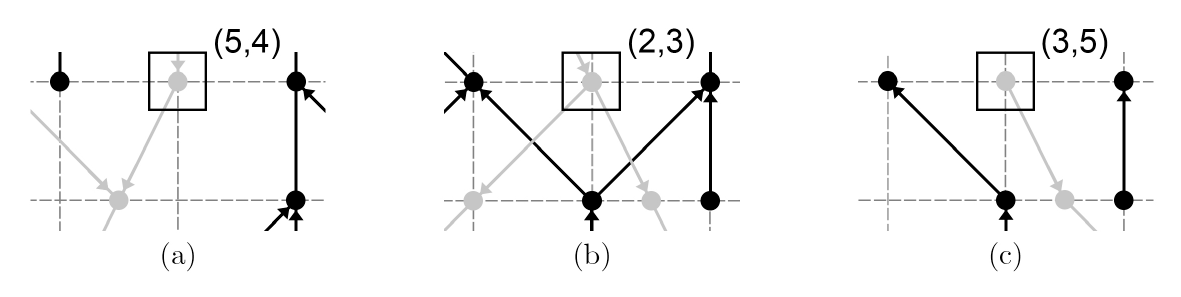}
\caption{Example of vertices (marked with a square) from Figure \ref{dualfig} where the situations (a), (b) and (c) described in the proof of proposition \ref{proposition:markov} occurs.}
\label{demomarkovfig}
\end{figure}

From situations (a), (b) and (c) above, for $u \in \mathbb{Z}$ and $v \in \frac{1}{2}\mathbb{Z}$ we have: 
$$P \left( \left. \hat{Y}_1^l(\hat{z}) - \hat{Y}_0^l(\hat{z}) = v \ \right| \ \hat{Y}_0^l(\hat{z}) = u \right) = $$
$$= (1-p) P(G_1 - G_2 = 2v) +  p \epsilon_n P(G_1 = 2v) 
+ \frac{p}{2}(1-\epsilon_n) P(G_1 = \left| 2v \right|).$$
\hfill $\square$ 

\vspace{0.2cm}

Note that the increments of $(\hat{Y}_k^l(\hat{z}))_{k \geq 0}$ do not have mean zero:
\begin{equation}
\label{eq:dualexpected}
E \left( \left. \hat{Y}_{k+1}^l(\hat{z}) - \hat{Y}_k^l(\hat{z}) \ \right| \ \hat{Y}_k^l(\hat{z}) \right) = \begin{cases} \displaystyle{0, \mbox{ if } \hat{Y}_k^l(\hat{z}) \notin \mathbb{Z}} ,  \\ \displaystyle{p \epsilon_n \frac{E(G)}{2} = \frac{\epsilon_n}{2}, \mbox{ if } \hat{Y}_k^l(\hat{z}) \in \mathbb{Z}} \end{cases},
\ \ \forall \, k\ge 0.
\end{equation}
For each fixed $\hat{z} \in \hat{V}$ and $n \geq 1$, we consider the following process:
$$M_{n,0}(\hat{z}) = \hat{z}, \ \displaystyle M_{n,k}(\hat{z}) = \hat{Y}_k^l(\hat{z}) - \sum_{j=1}^{k-1} \frac{\epsilon_n}{2} \mathbb{I}_{\{\mathbb{Z}\}}(\hat{Y}_j^l(\hat{z})), \ k \geq 1.$$
For any $\hat{z} \in \hat{V}{\cap \mathbb{Z}^2}$, we have that the process $(M_{n,k}(\hat{z}))_{k\geq0}$ is an $L^2$-martingale with respect to the reversed time filtration $\hat{\mathcal{F}}_k = \sigma(\{\omega(z),\theta(z): z \in \mathbb{Z}$ with $z_2 \geq \hat{z}_2-k\})$. By the invariance under integer translations of the system, we can suppose that $(0,0) \in \hat{V}$. Write $M_{n,j} = M_{n,j}(0,0)$ and $\hat{Y}_j = \hat{Y}_j(0,0)$ for the sake of simplicity and  define:
$$s_k^2 = E(M_{n,k}^2) = \sum_{j=1}^k E[(M_{n,j} - M_{n,j-1})^2], \quad \sigma_k^2 = \sum_{j=1}^k E[(M_{n,j} - M_{n,j-1})^2 | \hat{\mathcal{F}}_{j-1}],$$
$$\displaystyle \mathcal{M}_{n,k}(t) = s_k^{-1} \left( M_{n,j} + \frac{(M_{n,j+1} - M_{n,j})(ts_k^2 - s_j^2)}{(s_{j+1}^2 - s_j^2)} \right),$$
for $k \in \mathbb{N}$, $t \in [0,\infty)$ and $s_j^2 \leq ts_k^2 < s_{j+1}^2$. So, by definition, $(\mathcal{M}_{n,k}(t))_{t \geq 0}$ is the linearly interpolated normalization of the martingale $(M_{n,k}(t))_{t \geq 0}$ and note that since $M_{n,j}(\hat{z})$ is a time homogeneous process, we have the same values of $s_k^2$ and the same distribution of $\sigma_k^2$ for any vertex $\hat{z} \in \hat{V} \cap \mathbb{Z}^2$. 

\begin{proposition}
\label{proposition:martingale}
The process $\mathcal{M}_{n,k}(t)$ converges in distribution to a standard Brownian motion when $k \rightarrow \infty$.
\end{proposition}

\vspace{0.1cm}

For the proof of Proposition \ref{proposition:martingale} we will make use of the following result: 

\begin{proposition}
\label{Result:martingaleTCL}
\textbf{(\cite[Theorem $3$]{martingale})} Let $\{C,\mathcal{B},P_w\}$ be the probability space where $C=C[0,1]$ with the sup norm topology, $\mathcal{B}$ being the Borel $\sigma$-field generated by open sets in $C$ and $P_w$ the Wiener measure. Let $\{P_k\}$ be the sequence of probability measures on $\{C,\mathcal{B}\}$ determined by the distribution of $\{\mathcal{M}_{n,k}(t), 0 \leq t \leq 1 \}$. Then $P_k \rightarrow P_w$ weakly as $k \rightarrow \infty$ if the following two conditions hold:

(i) $\displaystyle \frac{\sigma_k^2}{s_k^2} \stackrel{P}{\longrightarrow} 1$ as $k \rightarrow \infty$;

(ii) Lindeberg condition:

$\displaystyle \frac{1}{s_k^2} \sum_{j=1}^{k} \displaystyle E\left[\left(M_{n,j} - M_{n,j-1} \right)^2 I_{\{|M_{n,j} - M_{n,j-1}| \geq \epsilon_n s_k \}} \right] \stackrel{}{\longrightarrow} 0$ as $k \rightarrow \infty$ for all $\epsilon_n > 0$.
\end{proposition}

\vspace{0.3cm}

\noindent \textit{Proof of Proposition \ref{proposition:martingale}.} For this proof we will verify that $\mathcal{M}_{n,k}(t)$ satisfies the two sufficient conditions of Proposition \ref{Result:martingaleTCL}. For condition (ii), using the bounds on the probability function of $(\hat{Y}_k^l(\hat{z}))_{k \geq 0}$ obtained in Proposition \ref{proposition:markov}, we can check that the increments $M_{n,j} - M_{n,j-1}$ have uniformly bounded variance. So $s_k^2$ is of order $k$ and (ii) follows. For condition (i), first, we need to do some auxiliary calculus using results obtained in the proof of Proposition \ref{proposition:markov}:
$$E \left[ \left(M_{n,j} - M_{n,j-1} \right)^2 \ \left| \ \hat{\mathcal{F}}_{j-1}, \hat{Y}_{j-1}^l \in \left\{ \mathbb{Z} + \frac{1}{2} \right. \right\} \right] =$$
$$= E \left[ \left( \hat{Y}_j^l - \hat{Y}_{j-1}^l - \frac{\epsilon_n}{2} I_{\mathbb{\{Z\}}}\hat{Y}_{j-1}^l \right)^2 \ \left| \ \hat{Y}_{j-1}^l \in \left\{ \mathbb{Z} + \frac{1}{2} \right. \right\} \right] =$$
$$= E \left[ \left( \hat{Y}_j^l - \hat{Y}_{j-1}^l \right)^2 \ \left| \ \hat{Y}_{j-1}^l \in \left\{ \mathbb{Z} + \frac{1}{2} \right.
\right\} \right] = E \left[ \left( \frac{1}{2} (G_1 - G_2) \right)^2 \right] =$$
\begin{equation}
= \frac{1}{4} \left[ E(G_1^2) + E(G_2^2) - 2E(G_1)E(G_2) \right] = \frac{1-p}{2p^2}.
\label{eq:varnonint} 
\end{equation}
Now the same computation for the case where $\hat{Y}_{j-1}^l$ is an integer:
$$E \left[ \left(M_{n,j} - M_{n,j-1} \right)^2 \ \left| \ \hat{\mathcal{F}}_{j-1}, \hat{Y}_{j-1}^l \in \mathbb{Z}  \right. \right] = E \left[ \left( \hat{Y}_j^l - \hat{Y}_{j-1}^l - \frac{\epsilon_n}{2} \right)^2 \ \left| \ \hat{Y}_{j-1}^l \in  \mathbb{Z}  \right. \right] $$
$$= E \left[ \left( \hat{Y}_j^l - \hat{Y}_{j-1}^l \right)^2 - \epsilon_n \left(\hat{Y}_j^l - \hat{Y}_{j-1}^l \right) + \frac{\epsilon_n^2}{4} \ \left| \ \hat{Y}_{j-1}^l \in  \mathbb{Z}  \right. \right] $$
$$= \frac{\epsilon_n^2}{4} - \epsilon_n E\left(\hat{Y}_j^l - \hat{Y}_{j-1}^l \ \left| \ \hat{Y}_{j-1}^l \in  \mathbb{Z} \right. \right) + E \left[ \left( \hat{Y}_j^l - \hat{Y}_{j-1}^l \right)^2 \ \left| \ \hat{Y}_{j-1}^l \in  \mathbb{Z} \right. \right]$$
$$= \frac{\epsilon_n^2}{4} - \frac{\epsilon_n^2}{2} + \frac{1-p}{4}E\left[ (G_1-G_2)^2 \right] + \frac{p}{4} E \left[G^2 \right]$$
\begin{equation}
= - \frac{\epsilon_n^2}{4} + \frac{1-p}{4} \left( \frac{2(1-p)}{p^2} \right) + \frac{p}{4} \left( \frac{2-p}{p^2} \right) = \frac{2(1-p)^2 + p(2-p)}{4p^2} - \frac{\epsilon_n^2}{4}.
\label{eq:varint} 
\end{equation}
Before we proceed, note that since $(\hat{Y}_j^l)_{j \geq 0}$ is a time homogeneous Markov process, we can create a discrete homogeneous Markov chain $(Z_j)_{j \geq 0}$ with space state $\{0,1\}$ that describes if either $\hat{Y}_j^l$ is integer ($Z_j = 1)$ or not $(Z_j = 0)$. Since $(Z_j)_{j \geq 0}$ is irreducible and aperiodic, it converges to a unique invariant distribution that we will denote by $\pi = (\pi(0),\pi(1))$. Now we are ready to verify (i):
$$\frac{\sigma_k^2}{k} = \frac{1}{k}\sum_{j=1}^k E[(M_{n,j} - M_{n,j-1})^2 | \hat{\mathcal{F}}_{j-1}] =$$
$$= \frac{1}{k}\sum_{j=1}^{k} \left[ \frac{1-p}{2p^2} \mathbb{I}_{\{\mathbb{Z}+\frac{1}{2}\}}(\hat{Y}_{j-1}^l) + \left( \frac{2(1-p)^2 + p(2-p)}{4p^2} - \frac{\epsilon_n^2}{4} \right) \mathbb{I}_{\{\mathbb{Z}\}}(\hat{Y}_{j-1}^l) \right] \stackrel{a.s.}{\longrightarrow} $$
$$\stackrel{a.s.}\longrightarrow \left(\frac{1-p}{2p^2}\right) \pi(0) + \left( \frac{2(1-p)^2 + p(2-p)}{4p^2} - \frac{\epsilon_n^2}{4} \right) \pi(1),$$
where the last equality is justified by (\ref{eq:varnonint}), (\ref{eq:varint}) and this convergence is justified by the Ergodic Theorem. We also have:
$$\frac{s_k^2}{k} = \frac{1}{k} \sum_{j=1}^k E[(M_{n,j} - M_{n,j-1})^2] = \frac{1}{k} \sum_{j=1}^k E \left[ E \left( (M_{n,j} - M_{n,j-1})^2 \ \left| \ \hat{\mathcal{F}}_{j-1} \right. \right) \right] = $$
$$= E \left[ \frac{\sigma_k^2}{k} \right] \stackrel{}{\longrightarrow} \left(\frac{1-p}{2p^2}\right) \pi(0) + \left( \frac{2(1-p)^2 + p(2-p)}{4p^2} - \frac{\epsilon_n^2}{4} \right) \pi(1),$$
where this convergence is justified by Dominated Convergence Theorem, since $\displaystyle \frac{\sigma_k^2}{k}$ is dominated by the maximum between $\displaystyle \frac{1-p}{2p}$ and $\displaystyle \left( \displaystyle \frac{2(1-p)^2 + p(2-p)}{4p^2} - \frac{\epsilon_n^2}{4} \right)$. Since $\sigma_k^2/k$ and $s_k^2/k$ converge a.s. to the same limit, (i) is satisfied and it concludes the proof. \hfill $\square$ \

\vspace{0.2cm}

\begin{remark}
The Proposition \ref{proposition:martingale} can be adapted for the dual r-paths. We just need to define the process $\tilde{\mathcal{M}}_{n,k}(t)$ using the martingale $\displaystyle \tilde{M}_{n,k}(\hat{z}) = \hat{Y}_k^r(\hat{z}) + \sum_{j=1}^{k-1} \frac{\epsilon_n}{2} \mathbb{I}_{\{\mathbb{Z}\}}(\hat{Y}_j^r(\hat{z}))$ instead of $M_{n,k}(\hat{z})$ and its proof would be analogous.
\end{remark}

\begin{remark}
The proof of Proposition \ref{proposition:martingale} considering that the initial vertex $\hat{z}$ has a non integer first coordinate would be analogous.
\end{remark}

Now we are ready to present the proof of Lemma \ref{lemma:onedualpath}.

\noindent \textit{Proof of Lemma \ref{lemma:onedualpath}.} This proof will be done only for $\hat{X}_n^l \subset \hat{W}_n^l$ since the proof for the case $\hat{X}_n^r \subset \hat{W}_n^r$ is analogous.

By the invariance under integer translations of the system, if the starting point of the dual l-path $\hat{X}_n^l$ has an integer first coordinate, we can assume that this path is starting from $(0,0)$. In the case where the starting point has a non integer first coordinate, we can restart the path $\hat{X}_n^l$ when it reaches an integer position for the first time and the number of steps that we need to wait until it happens is stochastically dominated by a geometric distribution with a parameter that does not depend on $n$. So, in both cases, we can assume that $\hat{X}_n^l$ starts from $(0,0)$.\

Now recall that we are assuming that $\epsilon_n = \fb n^{-1}$ and so we have:
$$\frac{1}{n}\sum_{j=1}^{n^2-1} \frac{\epsilon_n}{2} \mathbb{I}_{\{\mathbb{Z}\}}\left(\hat{Y}_j^l \right) = \frac{\fb}{n^2} \sum_{j=1}^{n^2-1} \mathbb{I}_{\{\mathbb{Z}\}}\left(\hat{Y}_j^l \right) \stackrel{a.s.}{\longrightarrow} \fb \pi(1),$$
where $\pi(1)$ is the same as defined in the proof of Proposition \ref{proposition:martingale} and this convergence is justified by the Ergodic Theorem.

Now since we have from Proposition \ref{proposition:martingale} that $\mathcal{M}_{n,k}(t)$ converges to a standard Brownian motion, it follows from Slutsky's Theorem, and the same arguments used in the proof of Lemma \ref{lemma:individualconvergence}, that $\hat{X}_n^l$ converges to a Brownian motion with drift, going backward in time.

To conclude, we will argue that this drift is equal to $b_p$. Since l-paths from dual never crosses l-paths from DNB, we have that $\hat{X}_n^l$ is always compressed between two l-paths from DNB, one that visits a position to the right of $x_n$ at time $t_n$ and other that visits a position to the left of $x_n$ at time $t_n$. By Lemma \ref{lemma:individualconvergence} we have that these two l-paths converge to Brownian motions with drift $-b_p$, under diffusive scale, which implies that the limiting process of $\hat{X}_n^l$, given that it exists, has drift $b_p$. \hfill $\square$ 

\vspace{0.2cm}

\subsection{Estimating coalescence times for a pair of dual l-paths (or dual r-paths)}\
\label{sec:dualcoalesce}

\vspace{0.2cm}

This subsection is devoted to proving Lemma \ref{lemma:dualcoalescencetime}. Before we present the proof, recall the notation of Section \ref{sec:coalescence}. Our strategy will consist of two steps: First, we will analyze specific random times where both dual paths are in integer positions and show that these times almost surely appear with enough frequency to allow us to do the proof considering the distance between the two dual paths only in these convenient random times. For this first part, we will prove Proposition \ref{proposition:renewaldual2}. In the second part we will show that, restricted to these random times, Proposition \ref{Result:coalescencetime} holds.

\vspace{0.2cm}

\begin{proposition}
\label{proposition:renewaldual2}
Let $\hat{\pi}_1$ and $\hat{\pi}_2$ be two dual l-paths starting from different integer points $x_1$ and $x_2$ at the same time, with $x_2 > x_1$. There exists a sequence of random times $\{\hat{S}_n : n \geq 0\}$ such that: 
\begin{enumerate}
\item[(a)] $\hat{\pi}_1(\hat{S}_n) \in \mathbb{Z}$ and $\hat{\pi}_2(\hat{S}_n) \in \mathbb{Z}$ for all $n \geq 0$;
\item[(b)] The sequence $\{\hat{S}_{n+1} - \hat{S}_n : n \geq 0 \}$ is independent and identically distributed;
\item[(c)] There exists a constant $C_s$ such that $\mathbb{E} [ |\hat{S}_1 - \hat{S}_0|] \leq C_s$;
\item[(d)] For any $l \geq 0$, a.s. $E[((\hat{\pi}_2(\hat{S}_{l+1}) - \hat{\pi}_1(\hat{S}_{l+1})) - (\hat{\pi}_2(\hat{S}_l) - \hat{\pi}_1(\hat{S}_l)))|\hat{\pi}_2(\hat{S}_l) - \hat{\pi}_1(\hat{S}_l)] \leq 0$.
\end{enumerate}
\end{proposition}

\noindent \textit{Proof.} First recall that if at least one of $\hat{\pi}_1$ and $\hat{\pi}_2$ is not in an integer position, the amount of time required until we have both $\hat{\pi}_1$ and $\hat{\pi}_2$ in an integer position is stochastically dominated by a geometric random variable $G_s$, with parameter $p^4(1-p)^2$, as we argued in Section \ref{sec:coalescence}.

So, define $\hat{S}_0$ as the first time such that both $\hat{\pi}_1$ and $\hat{\pi}_2$ are in an integer position. After that, if both of them stay in integer positions, we have $\hat{S}_1 = \hat{S}_0 - 1$. Otherwise, if at least one of them jumps to a non integer position, the remaining amount of time until both $\hat{\pi}_1$ and $\hat{\pi}_2$ reach integer positions is a random variable $T$, which is stochastically dominated by $G_s$, and in this case, we have $\hat{S}_1 = \hat{S}_0 - 1 - T$. Now, we can define $\hat{S}_n$ for any $n \geq 1$ following this procedure.

In this way, we have the statement $(a)$ of the proposition by construction. Statement $(b)$ follows from Markov property of the system. Statement $(c)$ holds since we can take $C_s = E[G_s] + 1 = p^{-4}(1-p)^{-2} + 1$.  

Now it remains to prove statement (d). First note that since $x_2 > x_1$, $\hat{\pi}_2(\hat{S}_l) - \hat{\pi}_1(\hat{S}_l)$ will be either a positive value or zero. If we have $\hat{\pi}_2(\hat{S}_l) - \hat{\pi}_1(\hat{S}_l) = 0$, it means that these two paths have already coalesced and then the distance between them will remain equal to zero forever. So, we do not need to worry about this situation. If we have $\hat{\pi}_2(\hat{S}_l) - \hat{\pi}_1(\hat{S}_l) > 0$, we can divide our analysis into two cases depending on whether these paths coalesce or not until time $\hat{S}_{l+1}$. In the case where they coalesce, we have that $\hat{\pi}_2(\hat{S}_{l+1}) - \hat{\pi}_1(\hat{S}_{l+1}) = 0$ and consequently the distance between them will reduce. For the case where they do not coalesce, note that we do not have information about the environment between the times $\hat{S}_l$ and $\hat{S}_{l+1}$; and that both are dual l-paths, so on average both paths have the same behavior. Because of that, we expect to see the same value for the distance at time $\hat{S}_{l+1}$ that we have in $\hat{S}_{l}$. It implies that the expectation value in part (d) has to be less or equal to zero. \hfill $\square$  

\vspace{0.2cm}

\begin{remark}
All statements in the Proposition \ref{proposition:renewaldual2} are also valid for dual r-paths in the same way and the proof would be analogue.
\end{remark}

Now we are ready to present the proof of Lemma \ref{lemma:dualcoalescencetime}.

\medskip

\noindent \textit{Proof of Lemma \ref{lemma:dualcoalescencetime}.} For this proof, we can consider that $\hat{\tau}_k$ refers to coalescence between l-paths since the proof considering r-paths would be the same.

First recall the following definition, from Section \ref{sec:coalescence}:
$$\hat{Z}^k_t = | \hat{Y}^l_t((0,0)) - \hat{Y}^l_t((k,0)) |, \ t \geq 0.$$

The main idea of this proof will be to verify that our distance process $\hat{Z}^k_t$ and the coalescence time $\hat{\tau}_k$ satisfies the conditions of Proposition \ref{Result:coalescencetime} when we consider the process $\hat{Z}^k_t$ only at times $\hat{S}_n$ as defined in Proposition \ref{proposition:renewaldual2}. To simplify the notation, from now on we will keep the notation $\hat{Z}^k_t$, but what we will analyze here is the sequence $\hat{Z}^k_{S_n}$. Proposition \ref{proposition:renewaldual2} will allow us to extend our conclusions to the whole process $\hat{Z}^k_t$, that is, if we prove the statement of Lemma \ref{lemma:dualcoalescencetime} considering the dual paths only in the times $S_n$, we will have it valid for the entire time line too.

Now denote $\hat{\mathcal{F}}_t = \sigma\{(\omega(z),\theta(z)), z=(z_1,z_2), z_1 \in \mathbb{Z}, z_2 \geq t\}$, $t \in \mathbb{Z}$. Note that $\hat{Z}^k_t$ only take values in $\mathbb{R}_{+}$ because we can not have a crossing between l-paths and since $\hat{\tau}_k$ is the first time that this process reaches zero, if all two conditions of Proposition \ref{Result:coalescencetime} hold, considering the filtration $\hat{\mathcal{F}}_t$, we will immediately have Lemma \ref{lemma:dualcoalescencetime} proved. So, from now on, we will prove these conditions.

Condition (i) of Proposition \ref{Result:coalescencetime} follows directly from $(d)$ in Proposition \ref{proposition:renewaldual2}.

About the first inequality of the condition (ii) of Proposition \ref{Result:coalescencetime}, note that the increments of $\hat{Z}^k_t$ are not spatially homogeneous, but given $\hat{\mathcal{F}}_l$ and that $\hat{Z}^k_l > 0$, it is always possible to assure that $(\hat{Z}^k_{l+1}-\hat{Z}^k_l)^2 \geq 1$ with a convenient choice of a finite number of open and closed vertices at time $S_{l+1}$.

About the second part of the condition (ii), we can write:
$$E[ \ |\hat{Z}^k_{j+1} - \hat{Z}^k_j|^3 \ |\mathcal{\hat{F}}_j] = E_{S_j - S_{j+1}} \left[ E\left[ \left. \ |\hat{Z}^k_{j+1} - \hat{Z}^k_j|^3 \ \right|\mathcal{\hat{F}}_j, S_j - S_{j+1}\right] \right] \leq$$
\begin{equation}
\leq 2 E(S_j - S_{j+1})  E \left[ \left. |\hat{Y}^l_{j+1}(\hat{u}) - \hat{Y}^l_{j}(\hat{u})|^3 \right| \hat{Y}^l_{j}(\hat{u}) \right] \leq 2C_sE[ \ |G|^3 ],
\label{eq:renewaldual2}
\end{equation}
where $G$ is a geom(p) and $C_s$ is the same constant that appears in Proposition \ref{proposition:renewaldual2}. The last inequality of \eqref{eq:renewaldual2} holds due to Proposition \ref{proposition:renewaldual2} and the fact that the increments of a dual path are stochastically dominated by a Geom(p). Hence, since the third absolute moment of the geometric distribution is finite, we have that condition  (ii) is also satisfied. 

\vspace{0.2cm}

\section{Proofs of some auxiliary results for Section \ref{sec:conditionI}.}\
\label{subsec:conditionIresults}

This subsection is devoted to proving some statements and lemmas of Section \ref{sec:conditionI}. So, the notations and definitions used here are the same used in Section \ref{sec:conditionI}. We begin by proving the upper bound of $E[\xi_{+}^{(n)}]$ that was used in the proof of Lemma \ref{proposition:RtilandR}.

\begin{claim}
\label{claim:overshoot}
Let $\xi_{+}^{(n)}$ be the overshoot distribution that is defined during the proof of Lemma \ref{proposition:RtilandR}. There exists $N \in \mathbb{N}$ such that $E[\xi_{+}^{(n)}] \leq \frac{5}{p}, \forall n \geq N$.
\end{claim}

\begin{proof}
Recall that $\xi_{+}^{(n)}$ is the overshoot distribution of $\Delta_t^{(n)} = \tilde{R}^{(n)}_t - L^{(n)}_t$ when it crosses the position $n$ from left to right at a time $t$. To simplify the notation, let us say here that it will happen at time $t=1$. We also suppress $n$ writing simply $\tilde{R}^{(n)}_t = R_t$ and $L^{(n)}_t = L_t$.

By Lemma \ref{lemma:longjump} and assuming that $n$ is large enough, we can consider that $n-n^{\frac{3}{4}} \le R_0-L_0 \le n$. Denote by $B_{+}$ the event $\Delta_1^{(n)} > n$. First note that if $L_1 > R_0 - \frac{n}{2}$, then vertex $(L_1,1)$ is open and either $L_1, \ R_1 \in [R_0-\frac{n}{2},R_0+\frac{n}{2}]$ or $L_1 = R_1$.

Thus $R_1 - L_1 \le n$ and $B_+$ does not happen. Analogously, if $R_1 < L_0 + \frac{n}{2}$, then $B_+$ does not happen either. Put $A_{+} = \{L_1 < R_0 - \frac{n}{2}, \, R_1 > L_0 + \frac{n}{2}\}$, then $B_+ \subset A_+$.

We also define $X_-$ the distance between $R_0-n$ and the first open vertex at the left hand side of $R_0-n$ at time $t=1$ and $Y_+$ as the distance between
$$
\mathfrak{a} = \max\Big(R_0, \, L_1+n, \, \big((n-(R_0-2L_0))+X_- \big) \mathbb{I}_{\{L_1 < R_0-n\}} \Big) -1
$$ 
and the first open vertex at the right hand side of $\mathfrak{a}$ at time $t=1$. Given $B_{+}$, and thus also $A_+$, our strategy from here will be to also condition on the position of $L_1$. Specifically it is straightforward to see that
$$
\xi_{+}^{(n)} \le Y_+ \textrm{ on } L_1=j \textrm{ for } R_0-n \le j \le R_0-\frac{n}{2}, 
$$
and
$$
\xi_{+}^{(n)} \le 2 X_- + Y_+ \textrm{ on } L_1 < R_0-n.
$$
The second inequality above requires $R_0-L_0 \ge \frac{n}{2}$ which is less restrictive than our hypothesis. Therefore
\begin{eqnarray} \label{cond-xi}
E[\xi_{+}^{(n)}] & \le & \sum_{j=R_0-n}^{R_0-\frac{n}{2}} E[ Y_+ | A_+, B_+,L_1=j] P(L_1=j|A_+, B_+) \nonumber \\
& & \ + E[2X_- + Y_+ | A_+, B_+,L_1<R_0-n] P(L_1<R_0-n|A_+, B_+).
\end{eqnarray}

In Figure \ref{overshootfar} we have an illustration of the occurrence of event $B_{+}$ according to either $L_1$ be smaller than $R_0 - n$ (scenario (a)) or not (scenario (b)).

\begin{figure}[h!]
\centering
\includegraphics[scale=0.50]{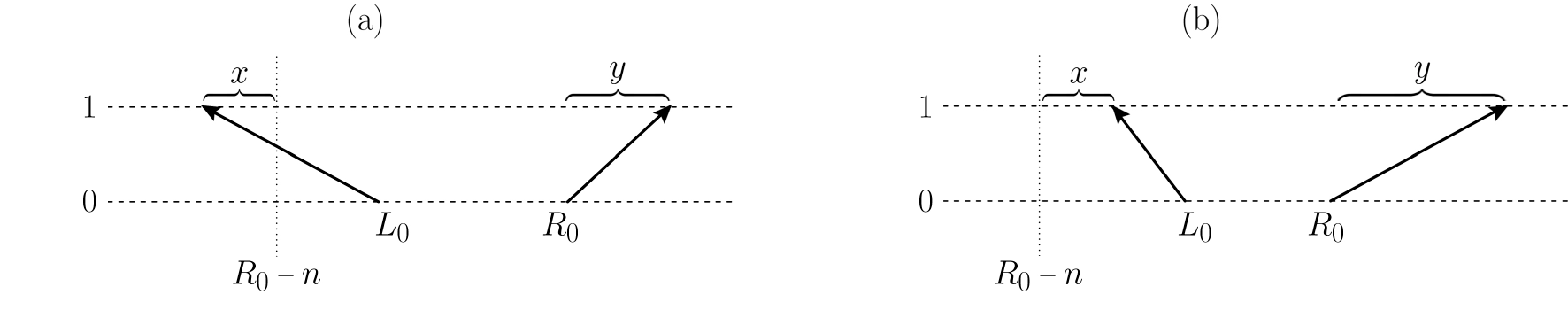}
\caption{Illustration of event $B_{+}$. In scenario (a) we have a situation where $L_1 < R_0-n$ and here $\xi_{+}^{(n)} = x + y$. In scenario (b) we have a situation where $R_0-n \le j \le R_0-\frac{n}{2}$ and here $\xi_{+}^{(n)} = y - x$.}
\label{overshootfar}
\end{figure}

On the event $A_+ \cap B_+ \cap \{L_1 = j\}$, $R_0-n \le j \le R_0-\frac{n}{2}$, the configuration on vertices $(x,1) \in \mathbb{Z}^2$ such that
$R_0 \le j + n \le x \le 2 R_0 -j$
is distributed as the conditional of IID Bernoulli with parameter $p$ given that at least one of these Bernoullis is one. This gives that 
\begin{equation}
\label{espY+}
E[ Y_+ | A_+, B_+,L_1=j] \le E[ Geom(p) | A_+, B_+,L_1=j] = \frac{1}{p}.
\end{equation}
Now write $P(X_{-} > x|A_{+},B_{+},L_1<R_0-n)$ as
\begin{eqnarray*}
&& \frac{P(B_{+},A_{+}|X_{-}>x,L_1<R_0-n)P(X_{-}>x|L_1<R_0-n)P(L_1<R_0-n)}{P(B_{+},A_+|L_1<R_0-n)P(L_1<R_0-n)} \\
& & \qquad \leq 2 P(X_{-}>x|L_1<R_0-n) 
\, \leq \, 2 P(\mbox{Geom}(p) > x),
\label{eq:overshootfarA}
\end{eqnarray*}
where the first inequality follows since conditioned to the occurrence of $\{L_1<R_0-n\}$, $R_t$ jumps to the right or stays in the same position at time $1$ is enough to assure that $B_{+} \cap A_+$ happens, which implies that $P(B_{+},A_+|L_1<R_0-n) \geq \frac{1}{2}$. Therefore
\begin{equation}
\label{espX-}
E[X_-| A_+, B_+,L_1<R_0-n] \le \frac{2}{p} .
\end{equation}

Going back to \eqref{cond-xi}, from \eqref{espX-} and \eqref{espY+} we obtain that
$E[\xi_{+}^{(n)}] \le 5/p.$
\end{proof}

\medskip

Now we present the proof of Lemmas \ref{lemma:pathsstayclose} and \ref{lemma:meetingtime}.

\smallskip

\noindent \textit{Proof of Lemma \ref{lemma:pathsstayclose}.}  We will only consider l-paths, but the proof for r-paths is analogous. By the time homogeneity of paths in the DNB, we can assume that $s=0$. 

Now fix $\gamma_2$ such that $\gamma_1 < \gamma_2 < \gamma$ and then, we can write:
\begin{align}
	&P\left(\sup_{0 \leq t \leq \tau^{(n)}_x} \left| \overline{l}^{(n)}_{x}(t) -  l^{(n)}_x(t) \right| > n^{\gamma} \right)  = P\left(\sup_{0 \leq t \leq \tau^{(n)}_x} \left| \overline{l}^{(n)}_{x}(t) -  l^{(n)}_x(t) \right| > n^{\gamma} ; \tau^{(n)}_x > n^{2\gamma_2}\right) \nonumber\\
	&\qquad \qquad \qquad \qquad + P\left(\sup_{0 \leq t \leq \tau^{(n)}_x} \left| \overline{l}^{(n)}_{x}(t) -  l^{(n)}_x(t) \right| > n^{\gamma} ; \tau^{(n)}_x \leq n^{2\gamma_2}\right) \nonumber
	\end{align}
which is bounded from above by
\begin{align}\label{eq:lemma6}
	&P\left( \tau^{(n)}_x > n^{2\gamma_2}\right) +  P\left(\sup_{0 \leq t \leq n^{2\gamma_2}} \left| \overline{l}^{(n)}_{x}(t) -  l^{(n)}_x(t) \right| > n^{\gamma} \right).
	\end{align}
About the first term in \eqref{eq:lemma6}, we have by Lemma \ref{lemma:dnbcoalescencetime} (and Remark \ref{remark:indepcoalesce} if $l^{(n)}_x$ and $\overline{l}^{(n)}_{x}$ are evolving according to independent environments) that:
$$P\left( \tau^{(n)}_x > n^{2\gamma_2}\right) \leq \frac{C_0 n^{\gamma_1}}{n^{\gamma_2}},$$
which converges to zero as $n \rightarrow \infty$. About the second term in \eqref{eq:lemma6}, we have by Lemma \ref{lemma:individualconvergence} that both $l^{(n)}_x$ and $\overline{l}^{(n)}_{x}$ converge under diffusive scale to Brownian motions with drift $-b_p$ and diffusion coefficient $\lambda_p^2$. So, this second term also converges to zero as $n \rightarrow \infty$.
\hfill $\square$

\medskip

\noindent \textit{Proof of Lemma \ref{lemma:meetingtime}.} First note that $\mathcal{T}_{x,s}^{(n)}$ does not depend on $x$ or $s$ due to the translation invariance in time and space of the DNB paths.

To prove the limit, let us decompose $\mathcal{T}^{(n)} = \tilde{\mathcal{T}}^{(n)} + \ddot{\mathcal{T}}^{(n)}$, where $\tilde{\mathcal{T}}^{(n)}$ is the amount of time that we have to wait until both $r^{(n)}_x$ and $\tilde{r}^{(n)}_x$ cross $l^{(n)}_{x}$, and $\ddot{\mathcal{T}}^{(n)}$ is the amount of time that until $r^{(n)}_x$ and $\tilde{r}^{(n)}_x$ meet each other for the first time after they cross $l^{(n)}_{x}$.

Now fix $\gamma_2$ and $\gamma_3$ such that $\gamma_1 < \gamma_2 < \gamma_3 < \gamma$. About $\tilde{\mathcal{T}}^{(n)}$, recall that $\tau^c_k$ is the time until we have a crossing between an l-path at the right-hand side of an r-path initially at distance $k$ from each other. By Corollary \ref{corollary:dnbcrossingtime} and Remark \ref{remark:indepcoalesce}, we have that
$$P\left(\tilde{\mathcal{T}}^{(n)} > n^{2\gamma}\right) \leq P\left(\tilde{\mathcal{T}}^{(n)} > n^{2\gamma_2}\right) \leq 2P\left(\tau^c_{n^{\gamma_1}} > n^{2\gamma_2}\right) \leq \frac{2 C_0 n^{\gamma_1}}{n^{\gamma_2}} = \tilde{C}_1 n^{\gamma_1 - \gamma_2},$$
which converges to zero as $n$ goes to infinity.

About $\ddot{\mathcal{T}}^{(n)}$, denote by $D_t^{(n)}$ the distance between $r^{(n)}_x$ and $\tilde{r}^{(n)}_x$ at time $t$. We have by Lemma \ref{lemma:individualconvergence} that both $r^{(n)}_x$ and $\tilde{r}^{(n)}_x$ converge under diffusive scaling to Brownian motions with drift $-b_p$ and diffusion coefficient $\lambda_p^2$. Then
	\begin{equation*}
	P\left(D^{(n)}_{\tilde{\mathcal{T}}^{(n)}} > n^{\gamma_3}, \tilde{\mathcal{T}}^{(n)} \le n^{2\gamma_2}\right)\le P\left( \sup_{0\le t \le n^{2\gamma_2}}D^{(n)}_{t} > n^{\gamma_3}\right)\to 0\, \text{as} \, n\to\infty\, . 
	\end{equation*}
	So, by Lemma \ref{lemma:dnbcoalescencetime}, we have:
	\begin{align*}
	P\left(\ddot{\mathcal{T}}^{(n)} > n^{2\gamma}\right) 
	& \le P\left( \left. \ddot{\mathcal{T}}^{(n)} > n^{2\gamma} \right| D^{(n)}_{\tilde{\mathcal{T}}^{(n)}} < n^{\gamma_3} \right) +  P\left(D^{(n)}_{\tilde{\mathcal{T}}^{(n)}} \geq n^{\gamma_3}\right)\\
	& \leq C_0 n^{\gamma_3-\gamma} + P\left(  D^{(n)}_{\tilde{\mathcal{T}}^{(n)}} \geq n^{\gamma_3}, \tilde{\mathcal{T}}^{(n)} \leq n^{2\gamma_2} \right) + P\left(\tilde{\mathcal{T}}^{(n)} > n^{2\gamma_2} \right)\, \\
	& \leq C_0 n^{\gamma_3-\gamma} + P\left( D^{(n)}_{\tilde{\mathcal{T}}^{(n)}} \geq n^{\gamma_3}, \tilde{\mathcal{T}}^{(n)} \leq n^{2\gamma_2} \right) + \tilde{C}_1 n^{\gamma_1 -\gamma_2}\, \to 0 \text{ as } n\to\infty\, ,
	\end{align*}
which completes the proof. \hfill $\square$

\vspace{0.2cm}

\section{On the condition ($U_{\mathcal{N}}^{''}$)}\
\label{subsec:conditionUdiscuss}

To prove that the DNB with branching parameter $\epsilon_n = \fb n^{-1}$ converges in distribution under diffusive scaling to the Brownian Net remains to verify that  any subsequential limit from Theorem \ref{Teo:teoprincipal} does not have more paths than the Brownian Net. The strategy we know so far follows from \cite{ss1} and consists of verifying that conditions ($U_{\mathcal{N}}^{'}$) and ($U_{\mathcal{N}}^{''}$).

We have proved ($U_{\mathcal{N}}^{'}$), but to prove condition ($U_{\mathcal{N}}^{''}$), we need to verify that for any limit point $(X,W^l,W^r,\widehat{W}^l,\widehat{W}^r)$ of $(X_n,W_n^l,W_n^r,\widehat{W}_n^l,\widehat{W}_n^r)$ and for any deterministic countable dense set $D \subset \mathbb{R}^2$, a.s. paths in $X$ do not enter any wedge of $(\widehat{W}^l(D),\widehat{W}^r(D))$ from outside. This last condition was not proved, but we conjecture that ($U_{\mathcal{N}}^{''}$) holds for the DNB. 

Let us make some considerations about ($U_{\mathcal{N}}^{''}$). The arguments used in \cite{ss1} to prove it for independent random walks with branching are the following: 

Let $\mathcal{D}^l, \mathcal{D}^r \subset \mathbb{R}^2$ be deterministic countable dense sets. For each $z \in \mathcal{D}^l$ (resp. $\tilde{z} \in \mathcal{D}^r$), fix a sequence $z_n \in \mathbb{Z}^2_{\mbox{even}}$ (resp. $\tilde{z}_n \in \mathbb{Z}^2_{\mbox{even}}$) such that $z_n \rightarrow z$ (resp. $\tilde{z}_n \rightarrow \tilde{z}$) under diffusive scaling and let $(\hat{l}_{z_n}^{(n)})_{n \geq 1}$ (resp. $(\hat{r}_{\tilde{z_n}}^{(n)})_{n \geq 1})$ be a sequence of dual l-paths (resp. dual r-paths) with branching parameter $o(n^{-1})$ starting from $z_n$ (resp. $\tilde{z}_n$). Also let $\tau(\hat{\pi}_1,\hat{\pi}_2)$ denote the "possibly infinite" first meeting time of dual paths $\hat{\pi}_1,\hat{\pi}_2$. Using independence, it is straightforward to show that for any $z \in \mathcal{D}^l$ and $\tilde{z} \in \mathcal{D}^r$, ($\hat{l}_{z_n}^{(n)}$,$\hat{r}_{z_n}^{(n)}$,$\tau(\hat{l}_{z_n}^{(n)},\hat{r}_{z_n}^{(n)}$)) jointly converges under diffusive scaling to ($\hat{l}_z$,$\hat{r}_z$,$\tau(\hat{l}_z,\hat{r}_z$)) where $(\hat{l}_z$,$\hat{r}_z)$ is a pair of left-right Brownian motions.

Let $\mathcal{N^{*}}$ be a weak limit of the system of independent random walks defined in \cite{ss1}. By Skorohod's representation theorem, we can construct a coupling such that the previous convergences occur almost surely and we will assume such coupling from now on. If $\pi \in \mathcal{N^{*}}$ enters a wedge $\psi(\hat{r},\hat{l})$ from outside, then by  \cite[Lemma 3.4(b)]{ss1}, $\pi$ must enter some skeletal wedge $\psi(\hat{r}_{\tilde{z}},\hat{l}_{z})$ from outside, with $z \in \mathcal{D}^l$ and $\tilde{z} \in \mathcal{D}^r$. By the a.s. convergence of the system to $\mathcal{N^{*}}$, there exist $\pi^{(n)}$ such that $\pi^{(n)} \rightarrow \pi$ under diffusive scaling. By the a.s. convergence of $\hat{l}_{z_n}^{(n)}$, \ $\hat{r}_{z_n}^{(n)}$ to $\hat{l}_{z}$, \ $\hat{r}_{z}$ and the convergence of their first meeting time, for $n$ large enough, $\pi^{(n)}$ must enter a discrete wedge from outside, which is impossible. 

So, the two key ingredients of their proof are: (1) the convergence of ($\hat{l}_{z_n}^{(n)}$,$\hat{r}_{z_n}^{(n)}$,$\tau(\hat{l}_{z_n}^{(n)},\hat{r}_{z_n}^{(n)}$)); (2) \cite[Lemma 3.4(b)]{ss1} which ensures that a path can only enter a wedge from the outside if it also enters a discrete version of this wedge from outside. It is worth noting that (1) is straightforward to verify for systems with independence before meeting times or coalescence. This follows from the fact that two independent Brownian paths with opposite drift will cross each other immediately after they meet, and this implies that pairs of random walks converging to the pair of Brownian motions will also cross each other (thus they meet) when they get close to each other by a distance negligible under diffusive scaling.

However, in our case, we do not have independence between paths and we were not able to prove the convergence of the first meeting time between one dual l-path and one dual r-path. Without that, it seems possible that two dual limiting paths $\hat{r}_z$ and $\hat{l}_{z}$ can create a wedge, although existing sequences of paths $\hat{r}_{z'}^{(n)}$ and $\hat{l}_{z}^{(n)}$ that converge to them and never meet each other. It would mean that we do not create a discrete version of the wedge and we cannot guarantee that do not exist paths in $\mathcal{X}$ that enter the wedge from outside.\\

\textbf{Acknowledgements.} This work derived from the PhD thesis of Rafael Santos, so we thank the committee members: Maria Eulália Vares, Giulio Iacobelli, Remy Sanchis and Simon Griffiths for all the helpful suggestions and comments. Finally, we are also grateful to an anonymous referee for his careful reading of the text and comments, which brought major improvements to the presentation of this paper.  


\end{document}